\documentclass[11pt]{amsart}
\usepackage{amsmath,amsfonts,amsthm,amssymb,url,graphicx}

\DeclareFontFamily{OT1}{rsfs}{}
\DeclareFontShape{OT1}{rsfs}{n}{it}{<-> rsfs10}{}
\DeclareMathAlphabet{\mathscr}{OT1}{rsfs}{n}{it}

\DeclareMathOperator{\mo}{\,mod}

\DeclareMathOperator{\err}{err}

\newcommand{\sume}{\mathop{\sum\nolimits^{*\mkern-12mu}}\limits}
\newtheorem{prop}{Proposition}[section]
\newtheorem{thm}[prop]{Theorem}
\newtheorem*{main}{Main Theorem}
\newtheorem{cor}[prop]{Corollary}
\newtheorem{lem}[prop]{Lemma}

\newtheorem*{defn*}{Definition}
\newenvironment{Rem}{{\bf Remark.}}{}
\numberwithin{equation}{section}
\title{The ternary Goldbach conjecture is true}
\author{H. A. Helfgott}
\address{Harald Helfgott, 
\'Ecole Normale Sup\'erieure, D\'epartement de Math\'ematiques, 45 rue d'Ulm, F-75230 Paris, France}
\email{harald.helfgott@ens.fr}
\begin{document}
\begin{abstract}
The ternary Goldbach conjecture, or three-primes problem, asserts that
every odd integer $n$ greater than $5$ is the sum of three primes. 
The present paper proves this conjecture.

Both the ternary Goldbach conjecture
and the binary, or strong, Goldbach conjecture had their origin in an
exchange of letters between Euler and Goldbach in 1742. We will follow
an approach based on the circle method, the large sieve and exponential
sums. Some ideas coming from Hardy, Littlewood and Vinogradov
are reinterpreted from a modern perspective. While all work here has to be
explicit, the focus is on qualitative gains.

The improved estimates on exponential sums are proven in the author's papers
on major and minor arcs for Goldbach's problem. One of the highlights of
the present paper is an optimized large sieve for primes. Its ideas get reapplied
to the circle method to give an improved estimate for the minor-arc integral.
\end{abstract}
\maketitle
\tableofcontents
\section{Introduction}
\subsection{Results}
The ternary Goldbach conjecture (or {\em three-prime problem}) 
states that every odd number $n$ greater than $5$ can be written as
the sum of three primes. Both the ternary Goldbach conjecture and
the (stronger) binary Goldbach conjecture (stating that every even number
greater than $2$ can be written as the sum of two primes) have their
origin in the correspondence between Euler and Goldbach (1742).
See \cite[Ch.~XVIII]{MR0245499} for the early history of the problem.

I. M. Vinogradov \cite{Vin} showed in 1937 that the ternary Goldbach conjecture
 is true for all $n$ above
a large constant $C$. Unfortunately, while the value of $C$ has been
improved several times since then, it has always remained much too large
($C=e^{3100}$, \cite{MR1932763}) for a mechanical verification up to $C$ to
be even remotely feasible. The situation was paradoxical: the conjecture was
known above an explicit $C$, but, even after seventy years of improvements,
this $C$ was so large that it could not be said that the problem could be
attacked by any conceivable computational means within our physical universe.
(The number of picoseconds since the beginning of the universe is less than 
$10^{30}$, whereas the number of protons in the observable universe is currently
estimated at $\sim 10^{80}$ \cite{EB}, thereby making even parallel computers somewhat
limited.) Thus, the only way forward was a series of drastic improvements in 
the mathematical, rather than computational, side.

The present paper proves the ternary Goldbach conjecture.
\begin{main}
Every odd integer $n$ greater than $5$ can be expressed as the sum of three
primes.
\end{main}
The proof given here works for all $n\geq C = 10^{27}$. (It is typical
of analytic proofs to work for all $n$ larger than a constant; see 
\S \ref{subs:hisres}.) Verifying the main theorem
 for $n<10^{27}$ is really a minor
computational task; it was already done for all
$n\leq 8.875\cdot 10^{30}$ in \cite{HelPlat}. (Appendix \ref{sec:appa}
provides an alternative approach.) This finishes the proof of the main theorem 
for all $n$.


We are able to set 
major arcs to be few and narrow because the minor-arc estimates in \cite{Helf} 
are very strong; we are 
forced to take them to be few and narrow because of the kind of $L$-function bounds
we will rely upon. (``Major arcs'' are small intervals around rationals of
small denominator; ``minor arcs'' are everything else. See the
definitions at the beginning of \S \ref{subs:sastre}.)

As has been the case since Hardy and Littlewood \cite{MR1555183}, the
approach is based on Fourier analysis, and, more particularly, on 
a study of exponential sums $\sum_p e(\alpha p) \eta(p/x)$, where
$\eta$ is a weight of our choice (a ``smoothing function'', or simply
a ``smoothing''). Such exponential
sums are estimated in \cite{HelfMaj} and \cite{Helf}
 for $\alpha$ lying in the major and minor arcs, respectively. Here we will
focus on the efficient use of such estimates to solve the main problem.

One of the main lessons of the proof -- also present in \cite{Helf} -- is
the close relation between the circle method and the large sieve; rather
than see large-sieve methods as a black box, we will use them as a source
for ideas. This applies, in particular, to the ideas behind an improved
large sieve for primes, which we derive here following and completing
Ramar\'e's ideas
on the subject \cite{MR2493924}.

Another guiding thought is really a relativization of a common dictum 
(``{\em always smooth}''). Smoothing is more useful for some tasks than for
others, and different kinds of 
smoothing functions may be appropriate for different parts of one problem.
The main results in \cite{HelfMaj} and \cite{Helf} are stated in terms
of different smoothing functions. Here, we will see how to coordinate the
use of different smoothings. We will also discuss how to choose smoothings
so as to make the main term as large as possible with respect to the error
term. (The emphasis elsewhere is, of course, on giving upper
bounds for the error term that are as small as possible.)

\subsection{History}
The following brief remarks are here to provide some background; no claim to 
completeness is made. Results on exponential sums over the primes
are discussed more specifically in \cite[\S 1]{Helf}.

\subsubsection{Results towards the ternary Goldbach conjecture}\label{subs:hisres}

Hardy and Littlewood \cite{MR1555183} proved that every odd number larger than
a constant $C$ is the sum of three primes, conditionally on the generalized
Riemann Hypothesis. This showed, as they said, that the problem was not
{\em unangreifbar} (as it had been called by Landau in \cite{Land}).

Vinogradov \cite{Vin} made the result unconditional. An explicit value for
$C$ (namely, $C = 3^{3^{15}}$) was first found by Borodzin in 1939.
This value was improved to $C=3.33\cdot 10^{43000}$ by J.-R. Chen and T. Z.
Wang \cite{MR1046491} and to $C= 2\cdot 10^{1346}$ by M.-Ch. Liu and T. Wang
\cite{MR1932763}. (J.-R. Chen had also proven that every large enough
even number is either the sum of two primes or the sum $p_1 + p_2 p_3$
of a prime $p_1$ and the product $p_2 p_3$ of two primes.)

There is a good reason why analytic proofs generally establish a result
only for integers $n$ larger than a constant $C$.
An analytic proof, such as the one in this paper, gives not only the
existence of a way to express a number $n$ in a certain form (say, as the sum
of three primes), but also an estimate on the (weighted) number of ways
to do so. Such an estimate is of the form 
\[\text{main term} + \text{error term},\]
where the main term is a precise function $f(n)$ and the error term is
shown to be bounded from above by a function $g(n)$; the proof works
if $g(n)<f(n)$ asymptotically as $n\to \infty$. 
 Of course, this means that such a proof works only once $g(n)\leq f(n)$,
that is, once $n$ is greater than some constant $C$, thus leaving
 small $n$ to be verified by 
direct computation. 

In \cite{MR1469323}, 
the ternary Goldbach conjecture was proven for all $n$ conditionally on the
generalized Riemann hypothesis. There, as here, the theorem was proven
analytically for all $n$ larger than a moderate constant $C$, and then
the task was completed by a numerical check for all odd $n<C$.

\subsubsection{Checking Goldbach for small $n$}\label{subs:checkgold}

Numerical verifications
of the binary Goldbach conjecture for small $n$
were published already in the late nineteenth century; see 
\cite[Ch.~XVIII]{MR0245499}. Richstein \cite{MR1836932} showed that every 
even integer $4\leq n\leq 4 \cdot 10^{14}$ is the sum of two primes. Oliveira 
e Silva, Herzog and Pardi \cite{OSHP} have proven that every even integer
$4\leq n\leq 4 \cdot 10^{18}$ is the sum of two primes. 

Clearly, if one can show that every interval of 
length $\geq 4\cdot 10^{18}-4$ within $\lbrack 1,N\rbrack$ contains a prime,
then \cite{OSHP} implies that every odd number between $7$ and $N$ can be
written as the sum of three primes: we let $p$ be the largest prime 
$\leq N-4$, and observe that $p-N$ is an even number $\leq 4\cdot 10^{18}$, and 
thus 
can be written as the sum of two primes.

Appendix \ref{sec:appa} proves that every interval 
of 
length $\geq 4\cdot 10^{18}-4$ within $\lbrack 1,N\rbrack$ contains a prime
for $N=1.23\cdot 10^{27}$ using a rigourous verification \cite{Plattpi}
of the fact
that the first $1.1\cdot 10^{11}$ zeros of the Riemann zeta function lie
on the critical line. Alternatively, one can simply construct a sequence
of primes up to $N$ such that any two consecutive primes in the list differ
by at most $4\cdot 10^{18}-4$. This was done in \cite{HelPlat}
for $N = 8.875694 \cdot 10^{30}$; thus, the ternary Goldbach conjecture
has been verified up to that value of $N$. 

The task of constructing the sequence of primes
up to $10^{27}$ -- enough to complete the proof of  the main theorem -- takes only
about 25 hours on a single processor core on a modern computer
(or five hours on five cores, since the algorithm is trivially 
parallelizable), provided that \cite{OSHP} is taken as a given. In other
words, verifying the theorem up to the point where the analytic proof in the
present paper starts working is a small, easily replicable task well within 
home-computing range.
 
\subsubsection{Work on Schnirelman's constant}

``Schnirelman's constant'' is a term for the smallest $k$ such that
{\em every} integer $n>1$ is the sum of at most $k$ primes. (Thus, Goldbach's
binary and ternary conjecture, taken together, are equivalent to the statement
that Schnirelman's constant is $3$.)  In 1930, Schnirelman \cite{MR1512821} 
showed that Schnirelman's constant $k$ is finite, developing in the process some
of the bases of what is now called additive or arithmetic combinatorics. 

In 1969, Klimov proved that $k\leq 6\cdot 10^9$; he later
improved this result to $k\leq 115$ \cite{MR0414506} (with G. Z. Piltay
and T. A. Sheptiskaya) and $k\leq 55$.
 Results by Vaughan \cite{MR0437478} ($k=27$), 
Deshouillers \cite{MR0466050} 
($k=26$) and Riesel-Vaughan \cite{MR706639} ($k=19$) then followed.

Ramar\'e showed in 1995 that every even $n>1$ is the sum of at most $6$ primes
\cite{MR1375315}. Recently, Tao \cite{Tao} established that every odd number
$n>1$ is the sum of at most $5$ primes. These results imply that
$k\leq 6$ and $k\leq 5$, respectively. The present paper implies that $k\leq 4$.

\begin{cor}[to Main Theorem]
Every integer $n>1$ is the sum of at most $4$ primes.
\end{cor}
\begin{proof}
If $n$ is odd and $>5$, the main theorem applies. If $n$ is even and $>8$,
apply the main theorem to $n-3$. Do the cases $n\leq 8$ separately.
\end{proof}
\subsubsection{Other approaches}

Since \cite{MR1555183} and \cite{Vin}, the main line of attack on the
problem has gone through exponential sums. There are proofs
based on cancellation in other kinds of sums
(\cite{MR834356}, \cite[\S 19]{MR2061214}), but they have
not been made to yield practical estimates. The same goes for proofs based
on other principles, such as that of Schnirelman's result or the recent
work of X. Shao \cite{XShao}. (It deserves to be underlined that
 \cite{XShao} establishes Vinogradov's three-prime result without using
$L$-function estimates at all; its constant $C$ is, however, extremely large.)

\subsection{Main ideas}\label{subs:sastre}

We will limit the discussion here to the general setup and to the use
of exponential-sum estimates. The derivation
of new exponential-sum estimates is the subject of \cite{Helf} and 
\cite{HelfMaj}.

In the circle method, the number of representations of a number $N$ 
as the sum of three primes is represented as an integral over the
``circle'' $\mathbb{R}/\mathbb{Z}$, which is partitioned into major
arcs $\mathfrak{M}$ and minor arcs $\mathfrak{m} = (\mathbb{R}/\mathbb{Z})\setminus \mathfrak{M}$:
\begin{equation}\label{eq:korto}\begin{aligned}
\sum_{n_1+n_2+n_3 = N} &\Lambda(n_1) \Lambda(n_2) \Lambda(n_3) =
 \int_{\mathbb{R}/\mathbb{Z}} (S(\alpha,x))^3 e(-N \alpha) d\alpha\\
&= \int_{\mathfrak{M}} (S(\alpha,x))^3 e(-N \alpha) d\alpha +
\int_{\mathfrak{m}} (S(\alpha,x))^3 e(-N \alpha) d\alpha ,
\end{aligned}\end{equation}
where $S(\alpha,x) = \sum_{n\leq x} \Lambda(n) e(\alpha n)$, 
$e(t) = e^{2\pi i t}$ and $\Lambda$ is the von Mangoldt function
($\Lambda(n) = \log p$ if $n = p^\alpha$, $\alpha\geq 1$,
and $\Lambda(n) = 0$ if $n$ is not a power of a prime).
The aim is to show that the sum of the integral over $\mathfrak{M}$
and the integral over $\mathfrak{m}$ is positive; this will prove
the three-primes theorem.

The major arcs $\mathfrak{M}=\mathfrak{M}_{r_0}$ 
consist of intervals $(a/q- c r_0/q x, a/q + c r_0/q x)$ around the rationals
$a/q$, $q\leq r_0$, where $c$ is a constant. In previous work\footnote{Ramar\'e's work
\cite{MR2607306} is in principle strong enough to allow $r_0$ to be an 
unspecified large constant. Tao's work \cite{Tao} reaches this
standard only for $x$ of moderate size.}, $r_0$ grew with $x$; in
 our setup, $r_0$ is a constant. Smoothing changes the left side of
(\ref{eq:korto}) into a weighted sum, but, since we aim at an existence
result rather than at an asymptotic for the number of
representations
$p_1+p_2+p_3$ of $N$, this is obviously acceptable. 

Typically, work on major arcs yields rather precise estimates on the
integral over $\int_\mathfrak{M}$ in (\ref{eq:korto}), whereas work
on minor arcs gives upper bounds on the absolute value of the integral
over $\int_\mathfrak{m}$ in (\ref{eq:korto}).

\subsubsection{Using major arc bounds}

We will be working with smoothed sums
\begin{equation}\label{eq:mork}
S_{\eta}(\alpha,x) = \sum_{n=1}^\infty \Lambda(n) \chi(n) e(\delta n/x)
\eta(n/x).\end{equation}
Our integral will actually be of the form
\begin{equation}\label{eq:kamon}\int_{\mathfrak{M}} S_{\eta_+}(\alpha,x)^2 S_{\eta_*}(\alpha,x) e(-N\alpha) d\alpha,\end{equation}
where $\eta_+$ and $\eta_*$ are two different smoothing functions.

Estimating the sums (\ref{eq:mork}) on $\mathfrak{M}$ reduces to estimating
the sums 
\begin{equation}\label{eq:mindy}
S_{\eta}(\delta/x,x) = \sum_{n=1}^\infty \Lambda(n) \chi(n) e(\delta n/x)
\eta(n/x)\end{equation}
for $\chi$ varying among
all Dirichlet characters modulo $q\leq r_0$ and
for $|\delta|\leq c r_0/q$, i.e., $|\delta|$ small.
The estimation of 
(\ref{eq:mindy}) for such $\chi$ and $\delta$ is the subject of
\cite{HelfMaj}.

 (It is in \cite{HelfMaj}, and not elsewhere, that the major 
$L$-function computation
in \cite{Plattfresh} gets used; it allows to give good estimates on sums
such as (\ref{eq:mindy}).)

The results in \cite{HelfMaj} allow us to use any smoothing based on the
Gaussian $\eta_\heartsuit(t) = e^{-t^2/2}$; this leaves us
some freedom in choosing $\eta_+$ and $\eta_*$. The main term in our
estimate for (\ref{eq:kamon}) is of the form
\begin{equation}\label{eq:liz}
C_0  \int_0^\infty \int_0^\infty \eta_+(t_1) \eta_+(t_2) 
\eta_*\left(\frac{N}{x}-(t_1+t_2)\right) dt_1 dt_2,\end{equation}
where $C_0$ is a constant. Our upper bound for the minor-arc integral,
on the other hand, will be proportional to $|\eta_+|_2^2 |\eta_*|_1$.
(Here, as is usual, we write $|f|_p$ for the $\ell_p$ norm of a function $f$.)
The question is then how to make (\ref{eq:liz}) divided by
$|\eta_+|_2^2 |\eta_*|_1$ as large as possible. A little thought will show
that it is best for $\eta_+$ to be symmetric, or nearly symmetric, 
around $t=1$ (say), and for $\eta_*$ be concentrated on a much shorter 
interval than $\eta_+$, while $x$ is set to be $x/2$ or slightly less.

It is easy to construct a function of the form $t\mapsto h(t)
\eta_\heartsuit(t)$ symmetric around $t=1$, with support on $\lbrack
0,2\rbrack$. We will define $\eta_+(t) = h_H(t) \eta_\heartsuit(t)$, where
$h_H$ is an approximation to $h$ that is band-limited in the Mellin sense.
This will mean that we will be able to use the estimates in 
\cite{HelfMaj}.

How to choose $\eta_*$? The bounds in \cite{Helf} were derived
for $\eta_2 = (2 I_{\lbrack 1/2,1\rbrack})\ast_M (2 I_{\lbrack
  1/2,1\rbrack})$, which is nice to deal with in the context of
combinatorially flavored analytic number theory, but it has a Mellin
transform that decays much too slowly.\footnote{This parallels the
  situation in the transition from Hardy and Littlewood \cite{MR1555183}
to Vinogradov \cite{Vin}. Hardy and Littlewood used the smoothing
$\eta(t)=e^{-t}$, whereas Vinogradov used the brusque (non-)smoothing 
$\eta(t) = I_{\lbrack 0,1\rbrack}$. Arguably, this is not just a case of
technological decay; $I_{\lbrack 0,1\rbrack}$ has compact support and is
otherwise easy to deal with in the minor-arc regime.} The solution
is to use a smoothing that is, so to speak, Janus-faced, viz.,
 $\eta_* = (\eta_2\ast_M \phi)(\varkappa t)$, where
$\phi(t) = t^2 e^{-t^2/2}$ and $\varkappa$ is a large constant. We
estimate sums of type $S_\eta(\alpha,x)$ by estimating
$S_{\eta_2}(\alpha,x)$ if $\alpha$ lies on a
minor arc, or by estimating $S_{\phi}(\alpha,x)$ if $\alpha$ lies on a
major arc. (The Mellin transform of $\phi$ is just a shift of that of
$\eta_\heartsuit$.) This is possible because $\eta_2$ has support bounded
away from zero, while $\phi$ is also concentrated away from $0$.

Now that the smoothing functions have been chosen, it remains 
to actually estimate (\ref{eq:kamon}) using the results from
\cite{HelfMaj},
which are estimates on (\ref{eq:mindy}) (and hence on 
(\ref{eq:mork})) for individual $\alpha$. Doing so well is a delicate task.
Some of the main features are the use of cancellation to prove a rather precise
estimate for the $\ell_2$ norm over the major arcs, and the arrangement of
error terms so that they are multiplied by the said $\ell_2$ norm. (The norm
will appear again later, in that it will be substracted from the integral
over a union of somewhat larger arcs, as in (\ref{eq:dadaj}).) We will actually
start by finding the main term, namely, (\ref{eq:notspel}); 
it is what one would expect, but extracting it
at the cost of only a small error term 
 will require some careful use of a smoothing $\eta_+$ 
 approximated by other smoothing $\eta_\circ$. (The main term is obtained
by completing several sums and integrals, whose terms must be shown to decrease
rapidly.)

\subsubsection{Minor arc bounds: exponential sums and the large sieve}
Let $\mathfrak{m}_{r}$ be the complement of $\mathfrak{M}_r$. 
In particular, $\mathfrak{m}=\mathfrak{m}_{r_0}$ is the complement of
$\mathfrak{M}=\mathfrak{M}_{r_0}$.
Exponential sum-estimates, such as those in \cite{Helf}, give
bounds on $\max_{\alpha \in \mathfrak{m}_r} |S(\alpha,x)|$ that decrease with
$r$. 

We need to do better than
\begin{equation}\label{eq:onedot}\begin{aligned}
\int_{\mathfrak{m}} \left|S(\alpha,x)^3 e(-N\alpha)\right| d\alpha &\leq
(\max_{\alpha\in \mathfrak{m}} |S(\alpha,x)|_\infty)\cdot 
\int_{\mathfrak{m}} |S(\alpha,x)|^2 d\alpha\\ &\leq
(\max_{\alpha \in \mathfrak{m}} |S(\alpha,x)|_\infty)\cdot 
\left(|S|_2^2 - \int_{\mathfrak{M}} |S(\alpha,x)|^2 d\alpha\right),
\end{aligned}\end{equation}
as this inequality involves a loss of a factor of $\log x$
(because $|S|_2^2 \sim x \log x$). Fortunately, minor arc estimates
are valid not just for a fixed $r_0$, but for the complement of
$\mathfrak{M}_{r}$, where $r$ can vary within a broad range. By
partial summation, these estimates can be combined with upper bounds for
\begin{equation}\label{eq:dadaj}
\int_{\mathfrak{M}_{r}} |S(\alpha,x)|^2 d\alpha - 
 \int_{\mathfrak{M}_{r_0}} |S(\alpha,x)|^2 d\alpha.
\end{equation}
Giving an estimate for the integral over $\mathfrak{M}_{r_0}$
($r_0$ a constant) will be
part of our task over the major arcs. The question is how to give an
upper bound for the integral over $\mathfrak{M}_r$ that is valid and
non-trivial over a broad range of $r$. 

The answer lies in the deep relation between the circle method and
the large sieve. (This was obviously not available to Vinogradov in 1937; 
the large sieve is a slightly later development (Linnik \cite{MR0004266}, 1941)
 that was optimized and fully understood later still.) 
A large sieve is, in essence, an inequality giving a 
discretized version of Plancherel's identity. Large sieves for primes
show that the inequality can be sharpened for sequences of prime support,
provided that, on the Fourier side, the sum over frequencies is shortened.
The idea here is that this kind of improvement can be adapted back to the
continuous context, so as to give upper bounds on the $L_2$ norms of
exponential sums with prime support when $\alpha$ is restricted to special
subsets of the circle. Such an $L_2$ norm is nothing other than
$\int_{\mathfrak{M}_{r}} |S(\alpha,x)|^2 d\alpha$.

The first version of \cite{Helf} used an idea of
Heath-Brown's\footnote{Communicated by Heath-Brown to the author, and by
the author to Tao, as acknowledged in \cite{Tao}. The idea is based on 
a lemma by Montgomery (as in, e.g., \cite[Lemma 7.15]{MR2061214}).} that can indeed be
understood in this framework. In \S \ref{subs:ramar}, we shall prove
a better bound, based on a large sieve for primes due to Ramar\'e 
\cite{MR2493924}. We will re-derive this sieve using an idea of Selberg's.
We will then make it fully explicit in the crucial range (\ref{subs:boquo}).
(This, incidentally, also gives fully explicit estimates for Ramar\'e's
large sieve in its original discrete context, making it the best large
sieve for primes in a wide range.)

The outcome is that $\int_{\mathfrak{M}_{r}} |S(\alpha,x)|^2 d\alpha$ is
bounded roughly by $2 x \log r$, rather than by $x\log x$ (or
by $2 e^\gamma x \log r$, as was the case when Heath-Brown's idea was
used). The lack of a factor of $\log x$ makes it possible to work with
$r_0$ equal to a constant, as we have done; the factor of $e^\gamma$
reduces the need for computations by more than an order of magnitude.








\subsection{Dependency diagram}
As usual, if two sections on the diagram are connected by a line, the upper
one
depends on the lower one. We use only the main results in \cite{HelfMaj}
and \cite{Helf}, namely, \cite[Main~Thm.]{HelfMaj}
and the statements in \cite[\S 1.1]{Helf}; these are labelled ``majarcs''
and
``minarcs'', respectively.

\begin{center}
  \includegraphics[height=2.2in]{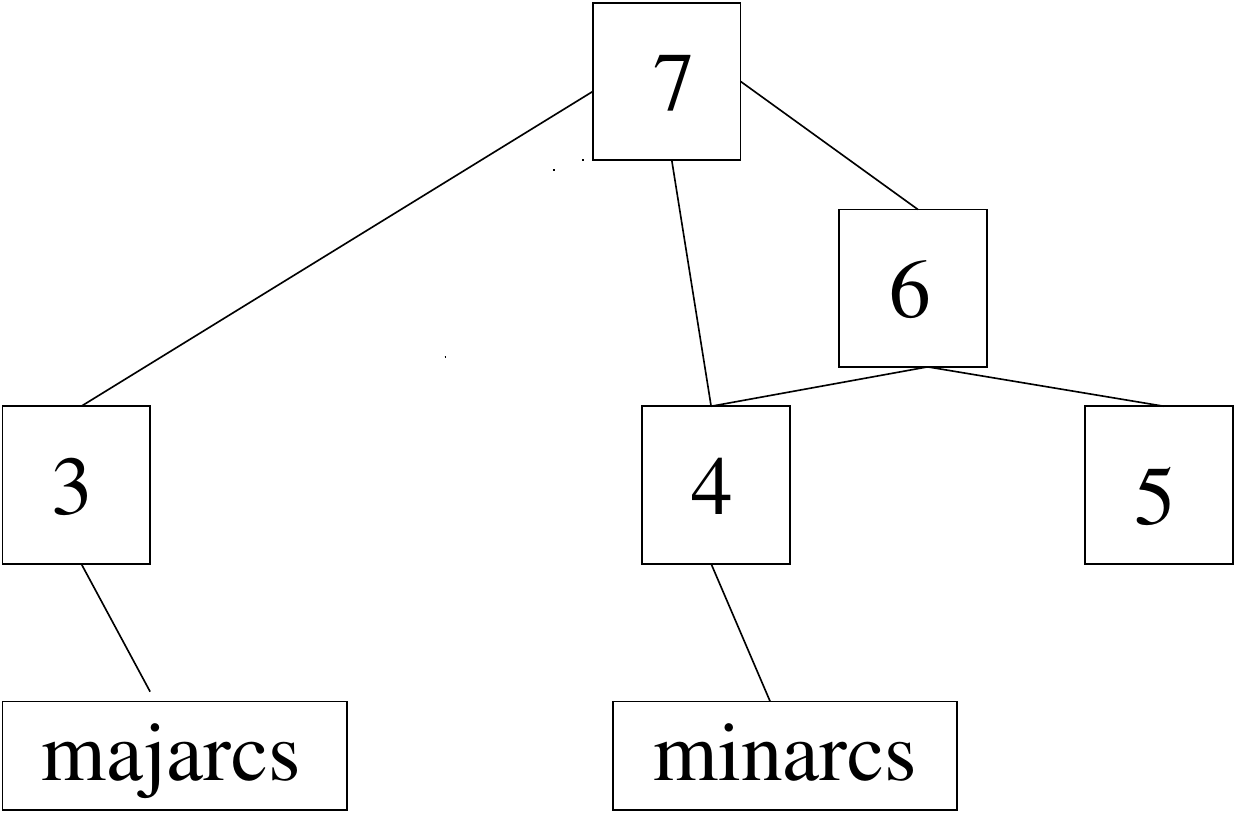}
\end{center}

\subsection{Acknowledgments}
The author is very thankful
to O. Ramar\'e for his help and feedback, especially regarding
\S \ref{sec:intri} and Appendix \ref{sec:sumphiq}. 
He is also much 
indebted to A. Booker, B. Green, 
H. Kadiri, D. Platt,
 T. Tao and M. Watkins for many discussions on Goldbach's problem and 
related issues. Thanks are also due to 
B. Bukh, A. Granville and P. Sarnak for their valuable advice.

Travel and other expenses were funded in part by 
the Adams Prize and the Philip Leverhulme Prize.
The author's work on the problem started at the
Universit\'e de Montr\'eal (CRM) in 2006; he is grateful to both the Universit\'e
de Montr\'eal and the \'Ecole Normale Sup\'erieure for providing pleasant
working environments.

The present work would most likely not have been possible without free and 
publicly available
software: PARI, Maxima, Gnuplot, VNODE-LP, PROFIL / BIAS, SAGE, and, of course, \LaTeX, Emacs,
the gcc compiler and GNU/Linux in general. Some exploratory work was done
in SAGE and Mathematica. Rigorous calculations used either D. Platt's
interval-arithmetic package (based in part on Crlibm) or the PROFIL/BIAS interval arithmetic package
underlying VNODE-LP.

\section{Preliminaries}
\subsection{Notation}
As is usual, we write $\mu$ for the Moebius function, $\Lambda$ for
the von Mangoldt function. We let $\tau(n)$ be the number of divisors of an
integer $n$ and $\omega(n)$ the number of prime divisors.
For $p$ prime, $n$ a non-zero integer,
we define $v_p(n)$ to be the largest non-negative integer $\alpha$ such
that $p^\alpha|n$.

We write $(a,b)$ for the greatest common divisor of $a$ and $b$. If there
is any risk of confusion with the pair $(a,b)$, we write $\gcd(a,b)$.
Denote by $(a,b^\infty)$ the divisor $\prod_{p|b} p^{v_p(a)}$ of $a$.
(Thus, $a/(a,b^\infty)$ is coprime to $b$, and is in fact the maximal
divisor of $a$ with this property.)

As is customary, we write $e(x)$ for $e^{2\pi i x}$.
We write $|f|_r$ for the $L_r$ norm of a function $f$.

We write $O^*(R)$ to mean a quantity at most $R$ in absolute value. 

\subsection{Dirichlet characters and $L$ functions}\label{subs:durian}
A {\em Dirichlet character} $\chi:\mathbb{Z}\to \mathbb{C}$ of modulus $q$
is a character $\chi$ of $(\mathbb{Z}/q \mathbb{Z})^*$ lifted to 
$\mathbb{Z}$ with the convention that $\chi(n)=0$ when
$(n,q)\ne 1$. Again by convention, there is a Dirichlet character of modulus
$q=1$, namely, the {\em trivial character} $\chi_T:\mathbb{Z}\to \mathbb{C}$
defined by $\chi_T(n)=1$ for every $n\in \mathbb{Z}$. 

If $\chi$ is a character modulo $q$ and $\chi'$ is a
character modulo $q'|q$ such that $\chi(n)=\chi'(n)$ for all $n$ coprime to $q$,
we say that $\chi'$ {\em induces} $\chi$. A character is 
{\em primitive} if it is not induced by any character of smaller modulus.
Given a character $\chi$, we write $\chi^*$ for the (uniquely defined)
primitive character inducing $\chi$. If a character $\chi$ mod $q$ is induced
by the trivial character $\chi_T$, we say that $\chi$ is {\em principal}
and write $\chi_0$ for $\chi$ (provided the modulus $q$ is clear from the 
context). In other words, $\chi_0(n)=1$ when $(n,q)=1$ and $\chi_0(n)=0$ when
$(n,q)=0$.

A Dirichlet $L$-function $L(s,\chi)$ ($\chi$ a Dirichlet character) is
defined as the analytic continuation of $\sum_n \chi(n) n^{-s}$ to the
entire complex plane; there is a pole at $s=1$ if $\chi$ is principal. 

A non-trivial zero of $L(s,\chi)$ is any $s\in \mathbb{C}$
such that $L(s,\chi)=0$ and $0 < \Re(s) < 1$. (In particular, a zero
at $s=0$ is called ``trivial'', even though its contribution can be a little
tricky to work out. The same would go for the other zeros with $\Re(s)=0$
occuring for $\chi$ non-primitive, though we will avoid this issue by
working mainly with $\chi$ primitive.) The zeros that occur at (some) negative 
integers are called {\em trivial zeros}.

The {\em critical line} is the line $\Re(s)=1/2$ in the complex
plane. Thus, the generalized Riemann hypothesis for Dirichlet $L$-functions
reads: for every Dirichlet character $\chi$,
all non-trivial zeros of $L(s,\chi)$ lie on the critical line. 
Verifiable finite versions of the generalized Riemann hypothesis generally 
read: for every Dirichlet character $\chi$ of modulus $q\leq Q$,
all non-trivial zeros of $L(s,\chi)$ with $|\Im(s)|\leq f(q)$ lie
on the critical line (where $f:\mathbb{Z}\to \mathbb{R}^+$ is some given 
function).

\subsection{Fourier transforms}
The Fourier transform on $\mathbb{R}$ is normalized as follows:
\[\widehat{f}(t) = \int_{-\infty}^\infty e(-xt) f(x) dx\]
for $f:\mathbb{R}\to \mathbb{C}$.

The trivial bound is $|\widehat{f}|_\infty \leq |f|_1$.
Integration by parts gives that, if $f$ is 
differentiable $k$ times outside finitely many points,
then
\begin{equation}\label{eq:madge}\begin{aligned}
\widehat{f}(t) &= 
O^*\left(\frac{|\widehat{f^{(k)}}|_\infty}{2\pi t}\right) =
O^*\left(\frac{|f^{(k)}|_1}{(2\pi t)^k}\right).
\end{aligned} \end{equation}
It could happen that $|f^{(k)}|_1=\infty$, in which case (\ref{eq:madge})
is trivial (but not false). In practice, we require $f^{(k)}\in L_1$. 
In a typical situation, $f$ is differentiable $k$ times except at
$x_1,x_2,\dotsc,x_k$, where it is differentiable only $(k-2)$ times;
the contribution of $x_i$ (say) to $|f^{(k)}|_1$ is then 
$|\lim_{x\to x_i^+} f^{(k-1)}(x) - \lim_{x\to x_i^-} f^{(k-1)}(x)|$.

\subsection{Mellin transforms}\label{subs:milly}

The {\em Mellin transform} of a function $\phi:(0,\infty)\to \mathbb{C}$ is
\begin{equation}\label{eq:souv}
M \phi(s) := \int_0^{\infty} \phi(x) x^{s-1} dx .\end{equation}

In general, $M (f\ast_M g) = Mf \cdot Mg$ and 
\begin{equation}\label{eq:mouv}
M(f\cdot g)(s) = \frac{1}{2\pi i}\int_{\sigma-i\infty}^{\sigma+i\infty}
Mf(z) Mg(s-z) dz\;\;\;\;\;\;\;\;\text{\cite[\S 17.32]{MR1773820}}\end{equation}
provided that $z$ and $s-z$ are within the strips on which $Mf$ and $Mg$
(respectively) are well-defined.

The Mellin transform is an isometry, in the sense that
\begin{equation}\label{eq:victi}
\int_0^\infty |f(t)|^2 t^{2\sigma} \frac{dt}{t} = \frac{1}{2\pi} \int_{-\infty}^\infty
|Mf(\sigma+it)|^2 dt.\end{equation}
provided that $\sigma+i\mathbb{R}$ is within the strip on which $Mf$ is
defined.
We also know that, for general $f$,
\begin{equation}\label{eq:harva}\begin{aligned}
M(t f'(t))(s) &= - s\cdot  Mf(s),\\
M((\log t) f(t))(s) &= (Mf)'(s)\end{aligned}\end{equation}
(as in, e.g., \cite[Table 1.11]{Mellin}).

 Since (see, e.g., \cite[Table 11.3]{Mellin}
or \cite[\S 16.43]{MR1773820})
\[(M I_{\lbrack a,b\rbrack})(s) = \frac{b^s - a^s}{s},\]
we see that
\begin{equation}\label{eq:envy}
M\eta_2(s) = \left(\frac{1 - 2^{-s}}{s}\right)^2,\;\;\;\;\;
M\eta_4(s) = \left(\frac{1 - 2^{-s}}{s}\right)^4 .\end{equation}

Let $f_z = e^{-zt}$, where $\Re(z)>0$. Then
\[\begin{aligned}
(Mf)(s) &= \int_0^\infty e^{-zt} t^{s-1} dt = \frac{1}{z^s} \int_0^\infty
e^{-t} dt\\&= \frac{1}{z^s} \int_0^{z\infty} e^{-u} u^{s-1} du = 
\frac{1}{z^s} \int_0^\infty e^{-t} t^{s-1} dt = \frac{\Gamma(s)}{z^s},
\end{aligned}\]
where the next-to-last step holds by contour integration,
and the last step holds by the definition of the Gamma function $\Gamma(s)$.

\section{The integral over the major arcs}\label{subs:prewo}
Let
\begin{equation}\label{eq:fellok}
S_\eta(\alpha,x) = \sum_n \Lambda(n) e(\alpha n) \eta(n/x),\end{equation}
where $\alpha \in \mathbb{R}/\mathbb{Z}$, $\Lambda$ is the von Mangoldt
function and $\eta:\mathbb{R}\to \mathbb{C}$ is of fast enough decay
for the sum to converge.

Our ultimate goal is to bound from below
\begin{equation}\label{eq:tripsum}
\sum_{n_1+n_2+n_3 = N} \Lambda(n_1) \Lambda(n_2) \Lambda(n_3)
\eta_1(n_1/x) \eta_2(n_2/x) \eta_3(n_3/x),\end{equation}
where $\eta_1, \eta_2, \eta_3:\mathbb{R}\to \mathbb{C}$. As can be readily seen, (\ref{eq:tripsum}) equals
\begin{equation}\label{eq:osto}\int_{\mathbb{R}/\mathbb{Z}} S_{\eta_1}(\alpha,x) 
S_{\eta_2}(\alpha,x) S_{\eta_3}(\alpha,x) 
e(-N \alpha) d\alpha .\end{equation}
In the circle method, the set $\mathbb{R}/\mathbb{Z}$ gets partitioned into
the set of {\em major arcs} $\mathfrak{M}$ and the set of {\em minor arcs}
$\mathfrak{m}$; the contribution of each of the two sets to the integral
(\ref{eq:osto}) is evaluated separately.

Our object here is to treat the major arcs: we wish to estimate
\begin{equation}\label{eq:russie}\int_{\mathfrak{M}} 
S_{\eta_1}(\alpha,x) S_{\eta_2}(\alpha,x) S_{\eta_3}(\alpha,x)
 e(-N \alpha) d\alpha\end{equation}
for $\mathfrak{M} = \mathfrak{M}_{\delta_0,r}$, where
\begin{equation}\label{eq:majdef}
\mathfrak{M}_{\delta_0,r} = \mathop{\bigcup_{q\leq r}}_{\text{$q$ odd}} \mathop{\bigcup_{a \mo q}}_{(a,q)=1}
 \left(\frac{a}{q} - \frac{\delta_0 r}{2 q x}, \frac{a}{q} + \frac{\delta_0 r}{2 q x}\right) \cup 
\mathop{\bigcup_{q\leq 2 r}}_{\text{$q$ even}} \mathop{\bigcup_{a \mo q}}_{(a,q)=1}
\left(\frac{a}{q} - \frac{\delta_0 r}{q x}, 
      \frac{a}{q} + \frac{\delta_0 r}{q x}\right) 
\end{equation}
and $\delta_0>0$, $r\geq 1$ are given.

In other words, our major arcs will be few (that is, a constant number)
and narrow. While \cite{MR1932763} used relatively narrow major arcs
as well, their number, as in all previous proofs of Vinogradov's result,
is not bounded by a constant. (In his proof of the five-primes theorem, 
\cite{Tao} is able to take a single major arc around $0$; this is not possible 
here.)


What we are about to see is the general framework of the major arcs. This is
naturally the place where the overlap with the existing literature is largest.
Two important differences can nevertheless be singled out.
\begin{itemize}
\item
 The most obvious one
is the presence of smoothing. At this point, it improves and simplifies error
terms, but it also means that we will later need estimates for exponential sums
on major arcs, and not just at the middle of each major arc. (If there is
smoothing, we cannot use summation by parts to reduce the problem of estimating 
sums to a problem of counting primes in arithmetic progressions, or weighted
by characters.)
\item Since our $L$-function estimates for exponential sums will give bounds
that are better than the trivial one by only a constant -- even if it is a 
rather large constant -- we need to be especially careful when estimating error
terms, finding cancellation when possible.
\end{itemize}

\subsection{Decomposition of $S_\eta(\alpha,x)$ by characters} 
What follows is largely classical; compare to 
 \cite{MR1555183} or, say, \cite[\S 26]{MR0217022}. The only difference
from the literature lies in the treatment of $n$ non-coprime to $q$,
and the way in which we show that our exponential sum (\ref{eq:beatit}) 
is equal to a linear combination of twisted sums $S_{\eta,\chi^*}$
over {\em primitive} characters $\chi^*$. (Non-primitive characters would give
us $L$-functions with some zeroes inconveniently placed on the line $\Re(s)=0$.)

Write $\tau(\chi,b)$ for the Gauss sum 
\begin{equation}\label{eq:caundal}
\tau(\chi,b) = \sum_{a \mo q} \chi(a) e(ab/q)\end{equation}
associated to a $b\in \mathbb{Z}/q\mathbb{Z}$ and a
Dirichlet character $\chi$ with modulus $q$.
 We let $\tau(\chi) = \tau(\chi,1)$.
If $(b,q)=1$, then $\tau(\chi,b) = \chi(b^{-1}) \tau(\chi)$.

Recall that $\chi^*$ denotes the primitive character inducing a given Dirichlet
character $\chi$.
Writing $\sum_{\chi \mo q}$ for a sum over all characters $\chi$ of
$(\mathbb{Z}/q \mathbb{Z})^*$), we see that,
for any $a_0\in \mathbb{Z}/q
\mathbb{Z}$,
\begin{equation}\label{eq:mouche}\begin{aligned}
\frac{1}{\phi(q)} &\sum_{\chi \mo q} \tau(\overline{\chi},b) \chi^*(a_0) =
\frac{1}{\phi(q)} 
\sum_{\chi \mo q} \mathop{\sum_{a \mo q}}_{(a,q)=1}
\overline{\chi(a)} e(ab/q) \chi^*(a_0) \\&= 
\mathop{\sum_{a \mo q}}_{(a,q)=1}
\frac{e(ab/q)}{\phi(q)}  \sum_{\chi \mo q} \chi^*(a^{-1} a_0) = 
\mathop{\sum_{a \mo q}}_{(a,q)=1}
\frac{e(ab/q)}{\phi(q)}  \sum_{\chi \mo q'} \chi(a^{-1} a_0),
\end{aligned}\end{equation}
where $q'=q/\gcd(q,a_0^\infty)$. Now,
$\sum_{\chi \mo q'} \chi(a^{-1} a_0)=0$ unless $a = a_0$ (in which case
$\sum_{\chi \mo q'} \chi(a^{-1} a_0)=\phi(q')$). Thus, (\ref{eq:mouche})
equals 
\[\begin{aligned}
\frac{\phi(q')}{\phi(q)}
&\mathop{\mathop{\sum_{a \mo q}}_{(a,q)=1}}_{a\equiv a_0 \mo q'}
e(ab/q) =  
\frac{\phi(q')}{\phi(q)}
\mathop{\sum_{k \mo q/q'}}_{(k,q/q')=1}
e\left(\frac{(a_0+kq') b}{q}\right)\\
&=  
\frac{\phi(q')}{\phi(q)} e\left(\frac{a_0 b}{q}\right)
\mathop{\sum_{k \mo q/q'}}_{(k,q/q')=1}
e\left(\frac{k b}{q/q'}\right) = 
\frac{\phi(q')}{\phi(q)} e\left(\frac{a_0 b}{q}\right)
\mu(q/q')
\end{aligned}\]
provided that $(b,q)=1$. (We are evaluating a {\em Ramanujan sum} in the
last step.)
Hence, for $\alpha = a/q +\delta/x$, $q\leq x$, $(a,q)=1$,
\[
\frac{1}{\phi(q)} \sum_\chi \tau(\overline{\chi},a) 
\sum_n \chi^*(n) \Lambda(n) e(\delta n/x) \eta(n/x)\]
equals
\[
\sum_n \frac{\mu((q,n^\infty))}{\phi((q,n^\infty))}
\Lambda(n) e(\alpha n) \eta(n/x).\]
Since $(a,q)=1$, $\tau(\overline{\chi},a)= \chi(a) \tau(\overline{\chi})$.
The factor $\mu((q,n^\infty))/\phi((q,n^\infty))$ equals $1$ when $(n,q)=1$;
the absolute value of the factor is at most $1$ for every $n$.
Clearly
\[\mathop{\sum_n}_{(n,q)\ne 1} \Lambda(n) \eta\left(\frac{n}{x}\right) = 
\sum_{p|q} \log p \sum_{\alpha\geq 1} \eta\left(\frac{p^\alpha}{x}\right).\]
Recalling the definition (\ref{eq:fellok}) of
$S_\eta(\alpha,x)$, we conclude that

\begin{equation}\label{eq:beatit}\begin{aligned}
S_{\eta}(\alpha,x) =  \frac{1}{\phi(q)} \sum_{\chi \mo q} 
\chi(a) \tau(\overline{\chi})
 S_{\eta,\chi^*}\left(\frac{\delta}{x},x\right) + O^*\left(
2 \sum_{p|q} \log p \sum_{\alpha\geq 1} \eta\left(\frac{p^\alpha}{x}\right)
\right),
\end{aligned}\end{equation}
where
\begin{equation}\label{eq:shangh}S_{\eta,\chi}(\beta,x) = 
\sum_n \Lambda(n) \chi(n) e(\beta n) \eta(n/x) .\end{equation}

Hence
$S_{\eta_1}(\alpha,x) S_{\eta_2}(\alpha,x) S_{\eta_3}(\alpha,x) e(-N\alpha)$ equals
\begin{equation}\label{eq:orgor}\begin{aligned}
\frac{1}{\phi(q)^3} \sum_{\chi_1} \sum_{\chi_2} \sum_{\chi_3} 
 &\tau(\overline{\chi_1}) \tau(\overline{\chi_2})
 \tau(\overline{\chi_3}) \chi_1(a) \chi_2(a) \chi_3(a) e(-N a/q)\\
 &\cdot S_{\eta_1,\chi_1^*}(\delta/x,x) S_{\eta_2,\chi_2^*}(\delta/x,x) 
S_{\eta_3,\chi_3^*}(\delta/x,x) e(-\delta N/x) \end{aligned}\end{equation}
plus an error term of absolute value at most
\begin{equation}\label{eq:joko}
2 \sum_{j=1}^3 \prod_{j'\ne j} |S_{\eta_{j'}}(\alpha,x)|  
\sum_{p|q} \log p \sum_{\alpha\geq 1} \eta_j\left(\frac{p^\alpha}{x}\right)
.\end{equation}
We will later see that the integral of (\ref{eq:joko}) over $S^1$
 is negligible -- for our choices of $\eta_j$, it will, in fact, be
of size $O(x (\log x)^A)$, $A$ a constant. (In (\ref{eq:orgor}), we
have reduced our problems to estimating $S_{\eta,\chi}(\delta/x,x)$
for $\chi$ {\em primitive}; a more obvious way of reaching the same goal
would have made (\ref{eq:joko}) worse
by a factor of about $\sqrt{q}$. The error term $O(x (\log x)^A)$ should
be compared to the main term, which will be
 of size about a constant times $x^2$.)



\subsection{The integral over the major arcs: the main term}
We are to estimate the integral (\ref{eq:russie}), where the
major arcs $\mathfrak{M}_{\delta_0,r}$ are defined as in (\ref{eq:majdef}).
We will use $\eta_1 = \eta_2 = \eta_+$, $\eta_3(t) = \eta_\ast(\varkappa t)$, 
where $\eta_+$ and $\eta_\ast$ will be set later. 


We can write 
\begin{equation}\label{eq:glenkin}\begin{aligned}
S_{\eta,\chi}(\delta/x,x) = S_{\eta}(\delta/x,x) &= 
\int_0^\infty \eta(t/x) e(\delta t/x) dt + O^*(\err_{\eta,\chi}(\delta,x))\cdot x\\
&= \widehat{\eta}(- \delta) \cdot x + O^*(\err_{\eta,\chi_T}(\delta,x))\cdot x
\end{aligned} \end{equation}
for $\chi=\chi_T$ the trivial character, and
\begin{equation}\label{eq:brahms}
S_{\eta,\chi}(\delta/x) = O^*(\err_{\eta,\chi}(\delta,x)) \cdot x\end{equation}
for $\chi$ primitive and non-trivial. The estimation of the error 
terms $\err$ will come later; let us focus on (a) obtaining the contribution
of the main term, (b) using estimates on the error terms efficiently.

{\em The main term: three principal characters.}
The main contribution will be given by the term in (\ref{eq:orgor})
with
$\chi_1 = \chi_2 = \chi_3 = \chi_0$, where $\chi_0$ is the principal
character mod $q$.

The sum $\tau(\chi_0,n)$ is a {\em Ramanujan sum};
as is well-known (see, e.g., \cite[(3.2)]{MR2061214}),
\begin{equation}\label{eq:selb}\tau(\chi_0,n) = \sum_{d|(q,n)} \mu(q/d) d.
\end{equation}
This simplifies to $\mu(q/(q,n)) \phi((q,n))$ for $q$ square-free.
The special case $n=1$ gives us that $\tau(\chi_0) = \mu(q)$.

Thus, the term in (\ref{eq:orgor}) with 
$\chi_1 = \chi_2 = \chi_3 = \chi_0$ equals
\begin{equation}\label{eq:lookatme}\frac{e(-N a/q)}{\phi(q)^3}  
\mu(q)^3 S_{\eta_+,\chi_0^*}(\delta/x,x)^2 S_{\eta_*,\chi_0^*}(\delta/x,x) 
e(-\delta N/x),\end{equation}
where, of course, $S_{\eta,\chi_0^*}(\alpha,x) = S_\eta(\alpha,x)$ (since
$\chi_0^*$ is the trivial character). 
Summing (\ref{eq:lookatme}) for $\alpha = a/q+\delta/x$
and $a$ going over all residues mod $q$ coprime to $q$, 
we obtain
\[\frac{\mu\left(\frac{q}{(q,N)}\right) \phi((q,N))}{\phi(q)^3} \mu(q)^3
S_{\eta_+, \chi_0^*}(\delta/x,x)^2 S_{\eta_*, \chi_0^*}(\delta/x,x) 
e(-\delta N/x).\]
The integral of (\ref{eq:lookatme}) over all of $\mathfrak{M} =
\mathfrak{M}_{\delta_0,r}$ (see (\ref{eq:majdef})) thus equals
\begin{equation}\label{eq:henki}\begin{aligned}
&\mathop{\sum_{q\leq r}}_{\text{$q$ odd}} \frac{\phi((q,N))}{\phi(q)^3}  \mu(q)^2 \mu((q,N))
\int_{-\frac{\delta_0 r}{2 q x}}^{\frac{\delta_0 r}{2 q x}} 
S_{\eta_+, \chi_0^*}^2(\alpha,x) S_{\eta_*, \chi_0^*}(\alpha,x) 
e(-\alpha N) d\alpha \\
+ &\mathop{\sum_{q\leq 2 r}}_{\text{$q$ even}} 
\frac{\phi((q,N))}{\phi(q)^3}  \mu(q)^2 \mu((q,N))
\int_{-\frac{\delta_0 r}{q x}}^{\frac{\delta_0 r}{q x}} 
S_{\eta_+, \chi_0^*}^2(\alpha,x) S_{\eta_*, \chi_0^*}(\alpha,x) 
e(-\alpha N) d\alpha .
\end{aligned}\end{equation}

The main term in (\ref{eq:henki}) is
\begin{equation}\label{eq:arger}\begin{aligned}
&x^3 \cdot \mathop{\sum_{q\leq r}}_{\text{$q$ odd}} 
\frac{\phi((q,N))}{\phi(q)^3}  \mu(q)^2 \mu((q,N))
\int_{-\frac{\delta_0 r}{2 q x}}^{\frac{\delta_0 r}{2 q x}} 
(\widehat{\eta_+}(-\alpha x))^2 \widehat{\eta_*}(-\alpha x) e(-\alpha N) d\alpha\\
+ &x^3 \cdot \mathop{\sum_{q\leq 2 r}}_{\text{$q$ even}}
 \frac{\phi((q,N))}{\phi(q)^3}  \mu(q)^2 \mu((q,N))
\int_{-\frac{\delta_0 r}{q x}}^{\frac{\delta_0 r}{q x}} 
(\widehat{\eta_+}(-\alpha x))^2 \widehat{\eta_*}(-\alpha x) e(-\alpha N) d\alpha
.\end{aligned}
\end{equation}

We would like to complete both the sum and the integral. Before, we should
say that we will want to be able to use smoothing functions $\eta_+$ whose
Fourier transforms are not easy to deal with directly. All we want to require
is that there be a smoothing function $\eta_\circ$, easier to deal with,
such that $\eta_\circ$ be close to $\eta_+$ in $\ell_2$ norm.

Assume, then, that 
\[|\eta_+-\eta_\circ|_2\leq \epsilon_0 |\eta_\circ|,\]
where $\eta_\circ$ is thrice differentiable outside finitely many points
and satisfies $\eta_\circ^{(3)} \in L_1$. Then (\ref{eq:arger})
equals
\begin{equation}\label{eq:pasture}\begin{aligned}
&x^3 \cdot \mathop{\sum_{q\leq r}}_{\text{$q$ odd}} 
\frac{\phi((q,N))}{\phi(q)^3}  \mu(q)^2 \mu((q,N))
\int_{-\frac{\delta_0 r}{2 q x}}^{\frac{\delta_0 r}{2 q x}} 
(\widehat{\eta_\circ}(-\alpha x))^2 \widehat{\eta_*}(-\alpha x) 
e(-\alpha N) 
d\alpha\\
+ &x^3 \cdot \mathop{\sum_{q\leq 2 r}}_{\text{$q$ even}} 
\frac{\phi((q,N))}{\phi(q)^3}  \mu(q)^2 \mu((q,N))
\int_{-\frac{\delta_0 r}{q x}}^{\frac{\delta_0 r}{q x}} 
(\widehat{\eta_\circ}(-\alpha x))^2 \widehat{\eta_*}(-\alpha x) 
e(-\alpha N) 
d\alpha.
\end{aligned}\end{equation}
plus 
\begin{equation}\label{eq:pommes}
O^*\left(x^2 \cdot \sum_{q} \frac{\mu(q)^2}{\phi(q)^2}
\int_{-\infty}^{\infty} |(\widehat{\eta_+}(-\alpha))^2 - 
(\widehat{\eta_\circ}(-\alpha))^2| |\widehat{\eta_*}(-\alpha)| d\alpha\right).
\end{equation}
Here (\ref{eq:pommes}) is bounded by $2.82643 x^2$ (by (\ref{eq:massacre}))
times
\[\begin{aligned}
|\widehat{\eta_*}(-\alpha)|_\infty &\cdot
\sqrt{\int_{-\infty}^{\infty} |\widehat{\eta_+}(-\alpha)-
\widehat{\eta_\circ}(-\alpha)|^2 d\alpha \cdot 
\int_{-\infty}^{\infty} |\widehat{\eta_+}(-\alpha) +
\widehat{\eta_\circ}(-\alpha)|^2 d\alpha}\\
&\leq |\eta_*|_1 \cdot |\widehat{\eta_+}-\widehat{\eta_\circ}|_2
|\widehat{\eta_+}+\widehat{\eta_\circ}|_2 =
|\eta_*|_1 \cdot |\eta_+-\eta_\circ|_2
|\eta_++\eta_\circ|_2 \\ &\leq
|\eta_*|_1 \cdot |\eta_+-\eta_\circ|_2
(2|\eta_\circ|_2 + |\eta_+-\eta_\circ|_2) =
|\eta_*|_1 |\eta_\circ|_2^2 \cdot (2 + \epsilon_0) \epsilon_0.
\end{aligned}\]
Now, (\ref{eq:pasture}) equals
\begin{equation}\label{eq:rusko}\begin{aligned}
x^3 
&\int_{-\infty}^{\infty}
(\widehat{\eta_\circ}(-\alpha x))^2 \widehat{\eta_*}(-\alpha x) 
e(-\alpha N) \mathop{\sum_{\frac{q}{(q,2)}\leq \min\left(
\frac{\delta_0 r}{2 |\alpha| x},r\right)}}_{\mu(q)^2 = 1}
\frac{\phi((q,N))}{\phi(q)^3} \mu((q,N))
d\alpha \\
= x^3 &\int_{-\infty}^\infty
(\widehat{\eta_\circ}(-\alpha x))^2 \widehat{\eta_*}(-\alpha x) 
e(-\alpha N) d\alpha \cdot \left(
\sum_{q\geq 1}
\frac{\phi((q,N))}{\phi(q)^3}  \mu(q)^2 \mu((q,N))
\right)\\
- x^3 &\int_{-\infty}^{\infty}
(\widehat{\eta_\circ}(-\alpha x))^2 \widehat{\eta_*}(-\alpha x) 
e(-\alpha N) 
\mathop{\sum_{\frac{q}{(q,2)}> \min\left(
\frac{\delta_0 r}{2 |\alpha| x},r\right)}}_{\mu(q)^2=1}
\frac{\phi((q,N))}{\phi(q)^3}  \mu((q,N))
d\alpha .\end{aligned}\end{equation}
The last line in (\ref{eq:rusko}) is bounded\footnote{This is obviously
crude, in that we are bounding $\phi((q,N))/\phi(q)$ by $1$. We are doing so
in order to avoid a potentially harmful dependence on $N$.}
 by
\begin{equation}\label{eq:boussole}x^2 |\widehat{\eta_*}|_\infty
\int_{-\infty}^{\infty}
|\widehat{\eta_\circ}(-\alpha)|^2  \sum_{\frac{q}{(q,2)}
> \min\left(\frac{\delta_0 r}{2|\alpha|},r\right)}
\frac{\mu(q)^2}{\phi(q)^2} d\alpha .\end{equation}
By (\ref{eq:madge}) (with $k=3$), (\ref{eq:gat1o}) and (\ref{eq:gat1e}), 
this is at most
\[\begin{aligned}x^2 &|\eta_*|_1 
\int_{-\delta_0/2}^{\delta_0/2} |\widehat{\eta_\circ}(-\alpha)|^2
\frac{4.31004}{r} d\alpha\\ &+ 
2 x^2 |\eta_*|_1 
\int_{\delta_0/2}^{\infty} \left(\frac{|\eta_\circ^{(3)}|_1}{
(2\pi \alpha)^3}\right)^2  
\frac{8.62008 |\alpha|}{\delta_0 r} d\alpha\\
&\leq |\eta_*|_1 \left(4.31004 |\eta_\circ|_2^2 +
0.00113 \frac{|\eta_\circ^{(3)}|_1^2}{\delta_0^5}\right)
 \frac{x^2}{r}
.\end{aligned}\]

It is easy to see that
\[\sum_{q\geq 1} \frac{\phi((q,N))}{\phi(q)^3}  \mu(q)^2 \mu((q,N))
= \prod_{p|N} \left(1 - \frac{1}{(p-1)^2}\right)
\cdot \prod_{p\nmid N} \left(1 + \frac{1}{(p-1)^3}\right).\]

Expanding the integral implicit in the definition of $\widehat{f}$,
\begin{equation}\label{eq:hosto}\begin{aligned}
\int_\infty^\infty &(\widehat{\eta_\circ}(-\alpha x))^2 \widehat{\eta_*}(- \alpha x)
 e(-\alpha N) d\alpha =\\ 
&\frac{1}{x} \int_0^\infty \int_0^\infty \eta_\circ(t_1)  \eta_\circ(t_2) 
\eta_*\left(\frac{N}{x}-(t_1+t_2)\right) 
dt_1 dt_2.
\end{aligned}\end{equation}
(This is standard. One rigorous way to obtain (\ref{eq:hosto}) is to approximate the integral over
$\alpha\in (-\infty,\infty)$ by an integral with a smooth weight, at
different scales; as the scale becomes broader, the Fourier transform
of the weight approximates (as a distribution) the $\delta$ function.
Apply Plancherel.)



Hence, (\ref{eq:arger}) equals
\begin{equation}\label{eq:notspel}\begin{aligned}
x^2 &\cdot \int_0^\infty \int_0^\infty \eta_\circ(t_1) \eta_\circ(t_2) 
\eta_*\left(\frac{N}{x}-(t_1+t_2)\right) 
dt_1 dt_2 \\ &\cdot \prod_{p|N} \left(1 - \frac{1}{(p-1)^2}\right)
\cdot \prod_{p\nmid N} \left(1 + \frac{1}{(p-1)^3}\right).\end{aligned}
\end{equation}
(the main term) plus  
\begin{equation}\label{eq:stev}
\left(2.82643 |\eta_\circ|_2^2 (2+\epsilon_0)\cdot \epsilon_0
+ \frac{4.31004 |\eta_\circ|_2^2 +
0.00113 \frac{|\eta_\circ^{(3)}|_1^2}{\delta_0^5}}{r} \right) |\eta_*|_1 x^2
\end{equation}

Here (\ref{eq:notspel}) is
 just as in the classical case \cite[(19.10)]{MR2061214}, except
for the fact that a factor of $1/2$ has been replaced by a double integral.
We will later see how to choose our smoothing functions (and $x$,
in terms of $N$) so as to make the double integral as large as possible.

What remains to estimate is the contribution of all the terms of the
form $\err_{\eta,\chi}(\delta,x)$ in (\ref{eq:glenkin}) and (\ref{eq:brahms}).
Let us first deal with another matter -- bounding the $\ell_2$ norm
of $|S_\eta(\alpha,x)|^2$ over the major arcs.

\subsection{The $\ell_2$ norm over the major arcs} 
We can always bound the integral of
$|S_\eta(\alpha,x)|^2$ on the whole circle by Plancherel. If we only want
the integral on certain arcs, we use the bound in Prop.~\ref{prop:bellen}
(based on work by Ramar\'e).
If these arcs are really the major arcs -- that is,
the arcs on which we have useful analytic estimates -- then we can hope
to get better bounds using $L$-functions. This will be useful both
to estimate the error terms in this section
 and to make the use of Ramar\'e's bounds more efficient later.

By (\ref{eq:beatit}),
\[\begin{aligned}
&\mathop{\sum_{a \mo q}}_{\gcd(a,q)=1} \left|S_{\eta}\left(
\frac{a}{q}+ \frac{\delta}{x},\chi\right)\right|^2 \\
&= \frac{1}{\phi(q)^2} \sum_\chi \sum_{\chi'} \tau(\overline{\chi})
\overline{\tau(\overline{\chi'})} \left(\mathop{\sum_{a \mo q}}_{\gcd(a,q)=1} \chi(a) \overline{\chi'(a)}
\right)\cdot S_{\eta,\chi^*}(\delta/x,x) \overline{S_{\eta,\chi'^*}(\delta/x,x)}
\\ &+ O^*\left(2 (1+\sqrt{q}) (\log x)^2 |\eta|_\infty 
\max_\alpha |S_\eta(\alpha,x)| + \left((1+\sqrt{q}) (\log x)^2 |\eta|_\infty\right)^2\right)\\
&= \frac{1}{\phi(q)} \sum_\chi |\tau(\overline{\chi})|^2
|S_{\eta,\chi^*}(\delta/x,x)|^2 + K_{q,1} (2 |S_\eta(0,x)| + K_{q,1}),
\end{aligned}\]
where
\[K_{q,1} = (1+\sqrt{q}) (\log x)^2 |\eta|_\infty .\]
As is well-known (see, e.g., \cite[Lem. 3.1]{MR2061214})
\[\tau(\chi) = \mu\left(\frac{q}{q^*}\right)
\chi^*\left(\frac{q}{q^*}\right)
\tau(\chi^*),\]
where $q^*$ is the modulus of $\chi^*$ (i.e., the conductor of $\chi$), and
 \[|\tau(\chi^*)| = \sqrt{q^*}. 
\]
Using the expressions (\ref{eq:glenkin}) and (\ref{eq:brahms}), we obtain
 \[\begin{aligned}&\mathop{\sum_{a \mo q}}_{(a,q)=1}
 \left|S_{\eta}\left(
\frac{a}{q}+ \frac{\delta}{x},x\right)\right|^2 =  
\frac{\mu^2(q)}{\phi(q)} 
\left|\widehat{\eta}(-\delta) x
 + O^*\left(\err_{\eta,\chi_T}(\delta,x) \cdot x\right)\right|^2\\ &+
\frac{1}{\phi(q)} \left(\sum_{\chi\ne \chi_T}  \mu^2\left(\frac{q}{q^*}\right) q^*
\cdot O^*\left(|\err_{\eta,\chi}(\delta,x)|^2 x^2\right)\right)
 + K_{q,1} (2 |S_\eta(0,x)| + K_{q,1})
\\ &= \frac{\mu^2(q) x^2}{\phi(q)} 
\left(|\widehat{\eta}(-\delta)|^2 +
O^*\left(\left|\err_{\eta,\chi_T}(\delta,x)
(2 |\eta|_1 + \err_{\eta,\chi_T}(\delta,x))\right|\right)
\right) \\ &+ 
O^*\left(q \max_{\chi\ne \chi_T} |\err_{\eta,\chi^*}(\delta,x)|^2 x^2 + K_{q,2} x\right),\end{aligned}\]
where $K_{q,2} = K_{q,1} (2 |S_\eta(0,x)|/x + K_{q,1}/x)$.

Thus, the integral of $|S_\eta(\alpha,x)|^2$ over
$\mathfrak{M}$ (see (\ref{eq:majdef})) is
\begin{equation}\label{eq:juto}\begin{aligned}
&\mathop{\sum_{q\leq r}}_{\text{$q$ odd}} \mathop{\sum_{a \mo q}}_{(a,q)=1}
\int_{\frac{a}{q}-\frac{\delta_0 r}{2 q x}}^{\frac{a}{q}+\frac{\delta_0
    r}{2 q x}}
\left|S_\eta(\alpha,x)\right|^2 d\alpha +
\mathop{\sum_{q\leq 2r}}_{\text{$q$ even}} \mathop{\sum_{a \mo q}}_{(a,q)=1}
\int_{\frac{a}{q}-\frac{\delta_0 r}{q x}}^{\frac{a}{q}+\frac{\delta_0 r}{q x}}
\left|S_\eta(\alpha,x)\right|^2 d\alpha \\
&=
\mathop{\sum_{q\leq r}}_{\text{$q$ odd}} \frac{\mu^2(q) x^2}{\phi(q)} 
\int_{-\frac{\delta_0 r}{2 q x}}^{\frac{\delta_0 r}{2 q x}}
\left|\widehat{\eta}(-\alpha x)\right|^2 d\alpha
+ 
\mathop{\sum_{q\leq 2 r}}_{\text{$q$ even}} \frac{\mu^2(q) x^2}{\phi(q)} 
\int_{-\frac{\delta_0 r}{q x}}^{\frac{\delta_0 r}{q x}}
\left|\widehat{\eta}(-\alpha x)\right|^2 d\alpha\\
&+
O^*\left(\sum_q \frac{\mu^2(q) x^2}{\phi(q)} \cdot 
\frac{\gcd(q,2) \delta_0 r}{qx} \left(ET_{\eta,\frac{\delta_0 r}{2}} (2|\eta|_1 + 
ET_{\eta,\frac{\delta_0 r}{2}})\right)\right) 
 \\ &+ 
\mathop{\sum_{q\leq r}}_{\text{$q$ odd}} \frac{\delta_0 r x}{q} \cdot O^*\left(
q \mathop{\mathop{\max_{\chi \mo q}}_{\chi \ne \chi_T}}_{|\delta|\leq
  \delta_0 r/2 q}
 |\err_{\eta,\chi^*}(\delta,x)|^2  + \frac{K_{q,2}}{x}\right)\\
&+ 
\mathop{\sum_{q\leq 2 r}}_{\text{$q$ even}} \frac{2 \delta_0 r x}{q} \cdot O^*\left(
q \mathop{\mathop{\max_{\chi \mo q}}_{\chi \ne \chi_T}}_{|\delta|\leq
  \delta_0 r/q}
 |\err_{\eta,\chi^*}(\delta,x)|^2  + \frac{K_{q,2}}{x}\right)
,\end{aligned}\end{equation}
where \[ET_{\eta,s} = \max_{|\delta|\leq s} |\err_{\eta,\chi_T}(\delta,x)|\]
and $\chi_T$ is the trivial character. If all we want is an upper bound,
we can simply remark that

\[\begin{aligned}
&x \mathop{\sum_{q\leq r}}_{\text{$q$ odd}} \frac{\mu^2(q)}{\phi(q)} 
\int_{-\frac{\delta_0 r}{2 q x}}^{\frac{\delta_0 r}{2 qx }}
\left|\widehat{\eta}(-\alpha x)\right|^2 d\alpha
+ 
x \mathop{\sum_{q\leq 2 r}}_{\text{$q$ even}} \frac{\mu^2(q)}{\phi(q)} 
\int_{-\frac{\delta_0 r}{q x}}^{\frac{\delta_0 r}{q x}}
\left|\widehat{\eta}(-\alpha x)\right|^2 d\alpha\\
&\leq \left(\mathop{\sum_{q\leq r}}_{\text{$q$ odd}} \frac{\mu^2(q)}{\phi(q)} +
\mathop{\sum_{q\leq 2 r}}_{\text{$q$ even}} \frac{\mu^2(q)}{\phi(q)}
\right) |\widehat{\eta}|_2^2 =
2 |\eta|_2^2 
\mathop{\sum_{q\leq r}}_{\text{$q$ odd}} \frac{\mu^2(q)}{\phi(q)}. 
\end{aligned}\]
If we also need a lower bound, we proceed as follows.

Again, we will work with an approximation $\eta_\circ$ such that (a)
$|\eta-\eta_\circ|_2$ is small, (b) $\eta_\circ$ is thrice
differentiable outside finitely many points, (c) $\eta_\circ^{(3)}\in L_1$.
Clearly,
\[\begin{aligned}
&x \mathop{\sum_{q\leq r}}_{\text{$q$ odd}} \frac{\mu^2(q)}{\phi(q)}
 \int_{-\frac{\delta_0 r}{2 q x}}^{\frac{\delta_0 r}{2 q x}}
\left|\widehat{\eta}(-\alpha x)\right|^2 d\alpha \\ &\leq
\mathop{\sum_{q\leq r}}_{\text{$q$ odd}} \frac{\mu^2(q)}{\phi(q)}
 \left(\int_{-\frac{\delta_0 r}{2 q}}^{\frac{\delta_0 r}{2 q}}
\left|\widehat{\eta_\circ}(-\alpha)\right|^2 d\alpha 
+ 2 \langle \left|\widehat{\eta_\circ}\right|,
  \left|\widehat{\eta} - \widehat{\eta_\circ}\right|\rangle 
+ \left|\widehat{\eta} - 
        \widehat{\eta_\circ}\right|_2^2 \right)\\
&=\mathop{\sum_{q\leq r}}_{\text{$q$ odd}} \frac{\mu^2(q)}{\phi(q)}
\int_{-\frac{\delta_0 r}{2 q}}^{\frac{\delta_0 r}{2 q}}
\left|\widehat{\eta_\circ}(-\alpha)\right|^2 d\alpha \\ &+
O^*\left(\frac{1}{2} \log r + 0.85\right) \left(2 
\left|\eta_\circ\right|_2 \left| \eta - \eta_\circ\right|_2 +
 \left|\eta_\circ - \eta\right|_2^2\right),
\end{aligned}\]
where we are using (\ref{eq:marmo}) and isometry.
Also,
\[\mathop{\sum_{q\leq 2 r}}_{\text{$q$ even}} \frac{\mu^2(q)}{\phi(q)}
 \int_{-\frac{\delta_0 r}{q x}}^{\frac{\delta_0 r}{q x}}
\left|\widehat{\eta}(-\alpha x)\right|^2 d\alpha =
\mathop{\sum_{q\leq r}}_{\text{$q$ odd}} \frac{\mu^2(q)}{\phi(q)}
 \int_{-\frac{\delta_0 r}{2 q x}}^{\frac{\delta_0 r}{2 q x}}
\left|\widehat{\eta}(-\alpha x)\right|^2 d\alpha.\]
By (\ref{eq:madge}) and Plancherel,
\[\begin{aligned}\int_{-\frac{\delta_0 r}{2 q}}^{\frac{\delta_0 r}{2 q}}
\left|\widehat{\eta_\circ}(-\alpha)\right|^2 d\alpha
&= \int_{-\infty}^{\infty} \left|\widehat{\eta_\circ}(-\alpha)\right|^2 d\alpha
- O^*\left(2 \int_{\frac{\delta_0 r}{2 q}}^{\infty} \frac{|\eta_\circ^{(3)}|_1^2}{
(2 \pi \alpha)^6} d\alpha\right)\\
&= |\eta_\circ|_2^2 + 
O^*\left(\frac{|\eta_\circ^{(3)}|_1^2 q^5}{5 \pi^6 (\delta_0 r)^5}\right), 
\end{aligned}\] 
Hence
\[
\mathop{\sum_{q\leq r}}_{\text{$q$ odd}} \frac{\mu^2(q)}{\phi(q)}
 \int_{-\frac{\delta_0 r}{2 q}}^{\frac{\delta_0 r}{2 q}}
\left|\widehat{\eta_\circ}(-\alpha)\right|^2 d\alpha = 
|\eta_\circ|_2^2 \cdot \mathop{\sum_{q\leq r}}_{\text{$q$ odd}}
\frac{\mu^2(q)}{\phi(q)} +
O^*\left(\mathop{\sum_{q\leq r}}_{\text{$q$ odd}}
\frac{\mu^2(q)}{\phi(q)} 
\frac{|\eta_\circ^{(3)}|_1^2 q^5}{5 \pi^6 (\delta_0 r)^5}\right).
\]
Using (\ref{eq:gatosbuenos}), we get that
\[\begin{aligned}\mathop{\sum_{q\leq r}}_{\text{$q$ odd}}
\frac{\mu^2(q)}{\phi(q)} 
\frac{|\eta_\circ^{(3)}|_1^2 q^5}{5 \pi^6 (\delta_0 r)^5}
&\leq \frac{1}{r} \mathop{\sum_{q\leq r}}_{\text{$q$ odd}}
\frac{\mu^2(q) q}{\phi(q)} 
\cdot \frac{|\eta_\circ^{(3)}|_1^2}{5 \pi^6 \delta_0^5}\\
&\leq \frac{|\eta_\circ^{(3)}|_1^2}{5 \pi^6 \delta_0^5} \cdot 
\left(0.64787 + \frac{\log r}{4 r} + \frac{0.425}{r}\right).
\end{aligned}\]

Going back to (\ref{eq:juto}), we use (\ref{eq:nagasa}) to bound 
\[\sum_q \frac{\mu^2(q) x^2}{\phi(q)} \frac{\gcd(q,2) \delta_0 r}{q x} \leq
2.59147 \cdot \delta_0 r x.\] We also note that
\[\begin{aligned}
&\mathop{\sum_{q\leq r}}_{\text{$q$ odd}} \frac{1}{q} +
\mathop{\sum_{q\leq 2r}}_{\text{$q$ even}} \frac{2}{q} =
\sum_{q\leq r} \frac{1}{q} - \sum_{q\leq \frac{r}{2}} \frac{1}{2q} +
\sum_{q\leq r} \frac{1}{q} \\ 
&\leq 2 \log e r - \log \frac{r}{2} \leq \log 2 e^2 r.
\end{aligned}\]

We have proven the following result.
\begin{lem}\label{lem:drujal}
Let $\eta:\lbrack 0,\infty) \to \mathbb{R}$ be in $L_1 \cap L_\infty$.
Let $S_\eta(\alpha,x)$ be as in (\ref{eq:fellok}) and
let $\mathfrak{M}=\mathfrak{M}_{\delta_0,r}$ be as in (\ref{eq:majdef}).
Let $\eta_\circ:\lbrack 0,\infty)\to \mathbb{R}$ be thrice differentiable outside finitely
many points. Assume $\eta_\circ^{(3)}\in L_1$.

Assume $r\geq 182$. Then
\begin{equation}\label{eq:bfpink}\begin{aligned}
\int_{\mathfrak{M}} |S_{\eta}(\alpha,x)|^2 d\alpha &= 
L_{r,\delta_0} x +
O^*\left(5.19 \delta_0 x r \left(ET_{\eta,\frac{\delta_0 r}{2}} \cdot \left(|\eta|_1 + 
\frac{ET_{\eta,\delta_0 r/2}}{2}\right)\right)\right) \\
&+ O^*\left(\delta_0 r (\log 2 e^2 r) \left( x\cdot
E_{\eta,r,\delta_0}^2 
+ K_{r,2}\right)\right),\end{aligned}\end{equation}
where \begin{equation}\label{eq:sreda}\begin{aligned}
E_{\eta,r,\delta_0} &= 
\mathop{\mathop{\max_{\chi \mo q}}_{q\leq r\cdot \gcd(q,2)}}_{|\delta|\leq
  \gcd(q,2) \delta_0 r/2 q}
\sqrt{q} |\err_{\eta,\chi^*}(\delta,x)|,\;\;\;\;\;\;\;
ET_{\eta,s} = \max_{|\delta|\leq s} |\err_{\eta,\chi_T}(\delta,x)|,\\
K_{r,2} &= (1+\sqrt{2 r}) (\log x)^2 |\eta|_\infty
(2 |S_{\eta}(0,x)|/x + (1+\sqrt{2 r}) (\log x)^2 |\eta|_\infty/x)
\end{aligned}\end{equation}
and $L_{r,\delta_0}$ satisfies both 
\begin{equation}\label{eq:mardi}
L_{r,\delta_0}\leq 2 |\eta|_2^2 
\mathop{\sum_{q\leq r}}_{\text{$q$ odd}} \frac{\mu^2(q)}{\phi(q)}
\end{equation} 
and 
\begin{equation}\label{eq:chetvyorg}\begin{aligned}
L_{r,\delta_0} &= 2 |\eta_\circ|_2^2 
  \mathop{\sum_{q\leq r}}_{\text{$q$ odd}} \frac{\mu^2(q)}{\phi(q)}
+ O^*(\log r + 1.7) \cdot 
\left(2 
\left|\eta_\circ\right|_2 \left| \eta - \eta_\circ\right|_2 +
 \left|\eta_\circ - \eta\right|_2^2\right)\\
&+ O^*\left(\frac{2 |\eta_\circ^{(3)}|_1^2}{5 \pi^6 \delta_0^5}\right)
\cdot \left(0.64787 + \frac{\log r}{4 r} + \frac{0.425}{r}\right).
\end{aligned}\end{equation}
\end{lem}

The error term $xr ET_{\eta,\delta_0 r}$ 
will be very small, since it will be estimated
using the Riemann zeta function; the error term involving $K_{r,2}$ will
be completely negligible. The term involving
$x r (r+1) E_{\eta,r,\delta_0}^2$; we see that it constrains us to have $|\err_{\eta,\chi}(x,N)|$ less than
a constant times $1/r$ if we do not want the main term in the bound (\ref{eq:bfpink}) to be overwhelmed.

\subsection{The integral over the major arcs: error terms. Conclusion}
There are at least two ways we can evaluate
(\ref{eq:russie}). One is to substitute (\ref{eq:orgor}) into 
(\ref{eq:russie}). The disadvantages here are that (a) this can give
rise to pages-long formulae, 
 (b)
this gives error terms proportional to $x r |\err_{\eta,\chi}(x,N)|$, meaning
that, to win, we would have to show that $|\err_{\eta,\chi}(x,N)|$ is much
smaller than $1/r$.
What we will do instead is to use our $\ell_2$ estimate (\ref{eq:bfpink})
in order to bound the contribution of non-principal terms. This will give
us a gain of almost $\sqrt{r}$ on the error terms; in other words,
to win, it will be enough to show later that $|\err_{\eta,\chi}(x,N)|$ is
much smaller than $1/\sqrt{r}$.

The contribution of the error terms in $S_{\eta_3}(\alpha,x)$ (that is,
all terms involving the quantities $\err_{\eta,\chi}$ in expressions 
(\ref{eq:glenkin}) and (\ref{eq:brahms})) to (\ref{eq:russie}) is
\begin{equation}\label{eq:huppert}\begin{aligned}
\mathop{\sum_{q\leq r}}_{\text{$q$ odd}} \frac{1}{\phi(q)} \sum_{\chi_3 \mo q}
 &\tau(\overline{\chi_3}) 
\mathop{\sum_{a \mo q}}_{(a,q)=1}
\chi_3(a) e(-Na/q)\\
&\int_{-\frac{\delta_0 r}{2 q x}}^{\frac{\delta_0 r}{2 q x}}
 S_{\eta_+}(\alpha+a/q,x)^2 \err_{\eta_*,\chi_3^*}(\alpha x,x) e(-N \alpha)
 d\alpha\\
+ \mathop{\sum_{q\leq 2 r}}_{\text{$q$ even}} \frac{1}{\phi(q)} \sum_{\chi_3 \mo q}
 &\tau(\overline{\chi_3}) 
\mathop{\sum_{a \mo q}}_{(a,q)=1}
\chi_3(a) e(-Na/q)\\
&\int_{-\frac{\delta_0 r}{q x}}^{\frac{\delta_0 r}{q x}}
 S_{\eta_+}(\alpha+a/q,x)^2 \err_{\eta_*,\chi_3^*}(\alpha x,x) e(-N \alpha)
 d\alpha.
\end{aligned}\end{equation}
We should also remember 
 the terms in (\ref{eq:joko}); we can integrate them over all of $\mathbb{R}/\mathbb{Z}$,
and obtain that they contribute at most
\[\begin{aligned}
\int_{\mathbb{R}/\mathbb{Z}} &2 \sum_{j=1}^3 \prod_{j'\ne j} |S_{\eta_{j'}}(\alpha,x)|  
\cdot \max_{q\leq r}
\sum_{p|q} \log p \sum_{\alpha\geq 1} \eta_j\left(\frac{p^\alpha}{x}\right)
d\alpha
\\
&\leq 2 \sum_{j=1}^3 \prod_{j'\ne j} |S_{\eta_{j'}}(\alpha,x)|_2  
\cdot \max_{q\leq r} 
\sum_{p|q} \log p \sum_{\alpha\geq 1} \eta_j\left(\frac{p^\alpha}{x}\right)
\\
&=
2 \sum_n \Lambda^2(n) \eta_+^2(n/x) \cdot
\log r \cdot \max_{p\leq r}
\sum_{\alpha\geq 1} \eta_*\left(\frac{p^\alpha}{x}\right)\\
&+
4 \sqrt{ \sum_n \Lambda^2(n) \eta_+^2(n/x)
 \cdot \sum_n \Lambda^2(n) \eta_*^2(n/x)} \cdot
\log r \cdot \max_{p\leq r}
\sum_{\alpha\geq 1} \eta_*\left(\frac{p^\alpha}{x}\right)\\
\end{aligned}\]
by Cauchy-Schwarz and Plancherel.

The absolute value of (\ref{eq:huppert}) is at most
\begin{equation}\label{eq:frainf}\begin{aligned}
\mathop{\sum_{q\leq r}}_{\text{$q$ odd}} \mathop{\sum_{a \mo q}}_{(a,q)=1} &\sqrt{q}
\int_{-\frac{\delta_0 r}{2 q x}}^{\frac{\delta_0 r}{2 q x}}
 \left|S_{\eta_+}(\alpha+a/q,x)\right|^2 d\alpha \cdot
\mathop{\max_{\chi \mo q}}_{|\delta|\leq \delta_0 r/2q}
|\err_{\eta_*,\chi^*}(\delta,x)| \\ +
\mathop{\sum_{q\leq 2 r}}_{\text{$q$ even}} \mathop{\sum_{a \mo q}}_{(a,q)=1} &\sqrt{q}
\int_{-\frac{\delta_0 r}{q x}}^{\frac{\delta_0 r}{q x}}
 \left|S_{\eta_+}(\alpha+a/q,x)\right|^2 d\alpha \cdot
\mathop{\max_{\chi \mo q}}_{|\delta|\leq \delta_0 r/q}
|\err_{\eta_*,\chi^*}(\delta,x)| \\
&\leq \int_{\mathfrak{M}_{\delta_0,r}}
\left|S_{\eta_+}(\alpha)\right|^2 d\alpha \cdot
\mathop{\mathop{\max_{\chi \mo q}}_{q\leq r\cdot \gcd(q,2)}}_{
|\delta|\leq \gcd(q,2) \delta_0 r/q}
\sqrt{q} |\err_{\eta_*,\chi^*}(\delta,x)| 
.\end{aligned}\end{equation}
We can bound the integral of $|S_{\eta_+}(\alpha)|^2$ by 
(\ref{eq:bfpink}).

What about the contribution of the error part of $S_{\eta_2}(\alpha,x)$?
We can obviously proceed in the same way, except that, to avoid double-counting,
$S_{\eta_3}(\alpha,x)$ needs to be replaced by 
\begin{equation}\label{eq:massac}\frac{1}{\phi(q)}
\tau(\overline{\chi_0}) \widehat{\eta_3}(- \delta) \cdot x
= \frac{\mu(q)}{\phi(q)} \widehat{\eta_3}(- \delta) \cdot x,\end{equation}
which is its main term (coming from (\ref{eq:glenkin})).
Instead of having an $\ell_2$ norm as in (\ref{eq:frainf}), we have the
square-root of a product of two squares of $\ell_2$ norms (by Cauchy-Schwarz),
namely, $\int_\mathfrak{M} |S_{\eta_+}^*(\alpha)|^2 d\alpha$ and
\begin{equation}\label{eq:thaddeus}\begin{aligned}
\mathop{\sum_{q\leq r}}_{\text{$q$ odd}} 
&\frac{\mu^2(q)}{\phi(q)^2} \int_{-\frac{\delta_0 r}{2 q x}}^{
\frac{\delta_0 r}{2 q x}} \left|\widehat{\eta_*}(- \alpha x) x\right|^2
d \alpha +
\mathop{\sum_{q\leq 2 r}}_{\text{$q$ even}} 
\frac{\mu^2(q)}{\phi(q)^2} \int_{-\frac{\delta_0 r}{q x}}^{
\frac{\delta_0 r}{q x}} \left|\widehat{\eta_*}(- \alpha x) x\right|^2
d \alpha\\ 
&\leq 
x |\widehat{\eta_*}|_2^2 \cdot \sum_{q} \frac{\mu^2(q)}{\phi(q)^2} .
\end{aligned}\end{equation}
By (\ref{eq:massacre}), the sum over $q$ is at most $2.82643$.

As for the contribution of the error part of $S_{\eta_1}(\alpha,x)$, we
bound it in the same way, using solely the $\ell_2$ norm in 
(\ref{eq:thaddeus}) (and replacing both $S_{\eta_2}(\alpha,x)$ 
and $S_{\eta_3}(\alpha,x)$ by expressions as in (\ref{eq:massac})).

The total of the error terms is thus
\begin{equation}\label{eq:teresa}\begin{aligned}
&x \cdot
\mathop{\mathop{\max_{\chi \mo q}}_{q\leq r\cdot \gcd(q,2)}}_{|\delta|\leq 
\gcd(q,2) \delta_0 r/q}
\sqrt{q} \cdot |\err_{\eta_*,\chi^*}(\delta,x)| \cdot A\\ +\; &x \cdot 
\mathop{\mathop{\max_{\chi \mo q}}_{q\leq r\cdot \gcd(q,2)}}_{|\delta|\leq 
\gcd(q,2) \delta_0 r/q}
\sqrt{q} \cdot |\err_{\eta_+,\chi^*}(\delta,x)| 
(\sqrt{A} + \sqrt{B_+}) \sqrt{B_*},
\end{aligned}\end{equation}
where $A = (1/x) \int_{\mathfrak{M}} |S_{\eta_+}(\alpha,x)|^2  d\alpha$ (bounded as in
(\ref{eq:bfpink})) and
\begin{equation}
B_* = 2.82643 |\eta_*|_2^2,
 \;\;\;\;\;\;\;\; B_+ = 2.82643 |\eta_+|_2^2.
\end{equation} 

In conclusion, we have proven
\begin{prop}\label{prop:nefumo}
Let $x\geq 1$. Let $\eta_+, \eta_*:\lbrack 0,\infty) \to \mathbb{R}$.
Assume $\eta_+ \in C^2$, $\eta_+''\in L_2$ and $\eta_+, \eta_* \in L^1 \cap
L^2$. Let $\eta_\circ:\lbrack 0,\infty)\to \mathbb{R}$ be thrice differentiable outside finitely
many points. Assume $\eta_\circ^{(3)}\in L_1$ and
$|\eta_+-\eta_\circ|_2<\epsilon_0 |\eta_\circ|_2$, where $\epsilon_0\geq 0$.

Let $S_{\eta}(\alpha,x) = \sum_{n} \Lambda(n) e(\alpha n) \eta(n/x)$.
Let $\err_{\eta,\chi}$, $\chi$ primitive, be given as in (\ref{eq:glenkin})
and (\ref{eq:brahms}). Let  $\delta_0>0$, $r\geq 1$.
Let $\mathfrak{M}=\mathfrak{M}_{\delta_0,r}$ be as in (\ref{eq:majdef}).

Then, for any $N\geq 0$,
\[
\int_{\mathfrak{M}} 
S_{\eta_+}(\alpha,x)^2 S_{\eta_*}(\alpha,x)
 e(-N \alpha) d\alpha\]
 equals
\begin{equation}\label{eq:opus111}\begin{aligned}
C_0 C_{\eta_\circ,\eta_*} x^2 &+ 
\left(
2.82643 |\eta_\circ|_2^2 (2+\epsilon_0)\cdot \epsilon_0
+ \frac{4.31004  |\eta_\circ|_2^2 +
0.0012 \frac{|\eta_\circ^{(3)}|_1^2}{\delta_0^5}}{r} \right) |\eta_*|_1 x^2\\
&+O^*(E_{\eta_*,r,\delta_0} A_{\eta_+} + 
E_{\eta_+,r,\delta_0} \cdot 
1.6812 (\sqrt{A_{\eta_+}} + 1.6812 |\eta_+|_2) |\eta_*|_2)
\cdot x^2\\
\\&+ O^*\left(2 Z_{\eta_+^2,2}(x) LS_{\eta_*}(x,r) \cdot x + 
4 \sqrt{Z_{\eta_+^2,2}(x) Z_{\eta_*^2,2}(x)}  LS_{\eta_+}(x,r)\cdot x\right),
\end{aligned}\end{equation}
where
\begin{equation}\label{eq:vulgo}\begin{aligned}
C_0 &= \prod_{p|N} \left(1 - \frac{1}{(p-1)^2}\right)
\cdot \prod_{p\nmid N} \left(1 + \frac{1}{(p-1)^3}\right),\\
C_{\eta_\circ,\eta_*} &=
 \int_0^\infty \int_0^\infty \eta_\circ(t_1) \eta_\circ(t_2) 
\eta_*\left(\frac{N}{x}-(t_1+t_2)\right) dt_1 dt_2 
,\end{aligned}\end{equation}
\begin{equation}\label{eq:vulgato}\begin{aligned}
E_{\eta,r,\delta_0} &= 
\mathop{\mathop{\max_{\chi \mo q}}_{q\leq \gcd(q,2) \cdot r}}_{
|\delta|\leq \gcd(q,2) \delta_0 r/ 2 q}
\sqrt{q}\cdot |\err_{\eta,\chi^*}(\delta,x)|,\;\;\;\;\;\;\;
ET_{\eta,s} = \max_{|\delta|\leq s/q} |\err_{\eta,\chi_T}(\delta,x)|,\\
A_{\eta} &= \frac{1}{x} \int_{\mathfrak{M}} \left|S_{\eta_+}(\alpha,x)\right|^2 d\alpha,
\;\;\;\;\;\;\;\;\;
L_{\eta,r,\delta_0} \leq 2 |\eta|_2^2 \mathop{\sum_{q\leq r}}_{\text{$q$
    odd}}  \frac{\mu^2(q)}{\phi(q)},\\
K_{r,2} &= (1+\sqrt{2 r}) (\log x)^2 |\eta|_\infty
(2 Z_{\eta,1}(x)/x + (1+\sqrt{2 r}) (\log x)^2 |\eta|_\infty/x),\\
Z_{\eta,k}(x) &= \frac{1}{x} \sum_n \Lambda^k(n) \eta(n/x),\;\;\;\;\;\;\;
LS_{\eta}(x,r) = \log r \cdot \max_{p\leq r}
 \sum_{\alpha\geq 1} \eta\left(\frac{p^\alpha}{
x}\right)
,\end{aligned}\end{equation}
and $\err_{\eta,\chi}$ is as in (\ref{eq:glenkin}) and (\ref{eq:brahms}).
\end{prop}
Here is how to read these expressions. The error term in the first line of
(\ref{eq:opus111}) will be small provided that $\epsilon_0$ is small
and $r$ is large.
The third line of
(\ref{eq:opus111}) will be negligible, as will be the term 
$2 \delta_0 r (\log e r) K_{r,2}$
in the definition of $A_\eta$.
(Clearly,
$Z_{\eta,k}(x) \ll_\eta (\log x)^{k-1}$ and 
$LS_{\eta}(x,q) \ll_\eta \tau(q) \log x$
for any $\eta$ of rapid decay.)

It remains to estimate the second line of (\ref{eq:opus111}).
This includes estimating $A_{\eta}$ -- a task that was already
accomplished in Lemma \ref{lem:drujal}.
We see that we will have to give very good bounds for $E_{\eta,r,\delta_0}$
 when $\eta=\eta_+$ or $\eta=\eta_*$.
We also see that we want to make $C_0 C_{\eta_+,\eta_*} x^2$ as large as possible;
it will be competing not just with the error terms here, but, more
importantly, with the bounds from the minor arcs, which will be proportional
to $|\eta_+|_2^2 |\eta_*|_1$.

\section{Optimizing and coordinating smoothing functions}\label{sec:rossini}

One of our goals is to maximize the quantity $C_{\eta_\circ,\eta_*}$ in 
(\ref{eq:vulgo}) 
relative to $|\eta_\circ|_2^2 |\eta_*|_1$. One way to do this is to ensure that
(a) $\eta_*$ is concentrated on a very short\footnote{This is an idea
due to Bourgain in a related context
\cite{MR1726234}.} interval $\lbrack 0,\epsilon)$,
(b) $\eta_\circ$ is supported on the interval $\lbrack 0,2\rbrack$, and is
symmetric around $t=1$, meaning that $\eta_\circ(t) \sim \eta_\circ(2-t)$.
Then, for $x \sim N/2$, the integral
\[
\int_0^\infty \int_0^\infty \eta_\circ(t_1) \eta_\circ(t_2) \eta_\ast\left(
 \frac{N}{x} - (t_1 + t_2)\right) dt_1 dt_2\]
in (\ref{eq:vulgo}) should be approximately equal to
\begin{equation}\label{eq:mana}|\eta_*|_1 \cdot 
\int_0^\infty \eta_\circ(t) \eta_\circ\left(\frac{N}{x} - t\right) dt =
 |\eta_*|_1 \cdot  \int_0^\infty \eta_\circ(t)^2 dt = |\eta_*|_1 \cdot |\eta_\circ|_2^2,
\end{equation}
provided that $\eta_0(t)\geq 0$ for all $t$.
It is easy to check (using Cauchy-Schwarz in the second step)
that this is essentially optimal. (We will redo this rigorously in a little
while.)

At the same time, the fact is that major-arc estimates are best for smoothing
functions $\eta$ of a particular form, and we have minor-arc estimates
from \cite{Helf} for
a different specific smoothing $\eta_2$. The issue, then, is how do we choose $\eta_\circ$
and $\eta_*$ as above so that we can 
\begin{itemize}
\item $\eta_*$ is concentrated on $\lbrack 0,\epsilon)$,
\item $\eta_\circ$ is supported on $\lbrack 0,2\rbrack$ and symmetric
around $t=1$,
\item we can give 
 minor-arc and major-arc estimates for $\eta_*$, 
\item we can give
major-arc estimates for a function $\eta_+$ close
to $\eta_\circ$ in $\ell_2$ norm?
\end{itemize} 
\subsection{The symmetric smoothing function $\eta_\circ$}\label{subs:charme}
We will later work with a smoothing function $\eta_\heartsuit$ 
whose Mellin transform
decreases very rapidly. Because of this rapid decay, we will be able to give
strong results based on an explicit formula for $\eta_\heartsuit$.
The issue is how to define $\eta_\circ$,
given $\eta_\heartsuit$, so that $\eta_\circ$ is symmetric around $t=1$
(i.e., $\eta_\circ(2 - x) \sim \eta_\circ(x)$) and is very small for $x>2$.

We will later set $\eta_{\heartsuit}(t) = e^{-t^2/2}$.
 Let
\begin{equation}\label{eq:clog}
h:t\mapsto \begin{cases}
t^3 (2-t)^3 e^{t-1/2} &\text{if $t\in \lbrack 0,2\rbrack$,}\\ 0 &\text{otherwise}\end{cases}
\end{equation}
We define $\eta_\circ:\mathbb{R}\to \mathbb{R}$ by
\begin{equation}\label{eq:cleo}\eta_\circ(t) = h(t) \eta_{\heartsuit}(t) = 
\begin{cases}
 t^3 (2-t)^3 e^{-(t-1)^2/2} &\text{if $t\in \lbrack 0,2\rbrack$,}\\ 
0 &\text{otherwise.}\end{cases}
\end{equation}
It is clear that $\eta_\circ$
is symmetric around $t=1$ for $t\in \lbrack 0,2\rbrack$.

\subsubsection{The product $\eta_\circ(t) \eta_\circ(\rho-t)$.}
We now should go back and redo rigorously what we discussed informally
around (\ref{eq:mana}). More precisely, we wish to estimate
\begin{equation}\label{eq:weor}
\eta_\circ(\rho) = 
\int_{-\infty}^\infty \eta_\circ(t) \eta_\circ(\rho-t) dt
= \int_{-\infty}^{\infty} \eta_\circ(t) \eta_\circ(2-\rho+t) dt
\end{equation}
for $\rho\leq 2$ close to $2$. In this, it will be useful that the
Cauchy-Schwarz inequality degrades slowly, in the following sense.
\begin{lem}\label{lem:gosor}
Let $V$ be a real vector space with an inner product 
$\langle \cdot,\cdot\rangle$. Then, for any $v,w\in V$ with $|w-v|_2 \leq
|v|_2/2$,
\[\langle v,w\rangle = |v|_2 |w|_2 + O^*(2.71 |v-w|_2^2).\]
\end{lem}
\begin{proof}
By a truncated Taylor expansion,
\[\begin{aligned}
\sqrt{1+x} &= 1 + \frac{x}{2} + \frac{x^2}{2} \max_{0\leq t\leq 1}
 \frac{1}{4 (1-(tx)^2)^{3/2}}\\
 &= 1 + \frac{x}{2} + O^*\left(\frac{x^2}{2^{3/2}}\right)
\end{aligned}\]
for $|x|\leq 1/2$. Hence, for $\delta = |w-v|_2/|v|_2$,
\[\begin{aligned}
\frac{|w|_2}{|v|_2} &= \sqrt{1 + \frac{2 \langle w-v,v\rangle
+ |w-v|_2^2}{|v|_2^2}}
= 1 + \frac{2 \frac{\langle w-v,v\rangle}{|v|_2^2} + \delta^2}{2} + O^*\left(\frac{(2\delta+\delta^2)^2}{2^{3/2}}\right)\\ &= 1 + \delta + O^*\left(\left(
\frac{1}{2} + \frac{(5/2)^2}{2^{3/2}}\right) \delta^2\right)
= 1 + \frac{\langle w-v,v\rangle}{|v|_2^2} + O^*\left(2.71 \frac{|w-v|_2^2}{|v|_2^2}\right).
\end{aligned}\]
Multiplying by $|v|_2^2$, we obtain that
\[|v|_2 |w|_2 = |v|_2^2 + \langle w-v,v\rangle + O^*\left(2.71 |w-v|_2^2\right)
= \langle v,w\rangle + O^*\left(2.71 |w-v|_2^2\right).\]
\end{proof}
Applying Lemma \ref{lem:gosor} to (\ref{eq:weor}), we obtain that
\begin{equation}\label{eq:espri}\begin{aligned}
(\eta_\circ \ast \eta_\circ)(\rho) &=
\int_{-\infty}^{\infty} \eta_\circ(t) \eta_\circ((2-\rho)+t) dt \\ &=
\sqrt{\int_{-\infty}^\infty |\eta_\circ(t)|^2 dt} \sqrt{\int_{-\infty}^\infty
 |\eta_\circ((2-\rho)+t)|^2 dt}
\\ &+ O^*\left(2.71 \int_{-\infty}^\infty \left|\eta_\circ(t)-
\eta_\circ((2-\rho)+t)\right|^2 dt\right)\\
&= |\eta_\circ|_2^2 + O^*\left(2.71 \int_{-\infty}^{\infty}
\left(\int_0^{2-\rho} \left|\eta_\circ'(r+t)\right|
dr\right)^2 dt\right)\\
&= |\eta_\circ|_2^2 + O^*\left(2.71 (2-\rho)
\int_0^{2-\rho}
 \int_{-\infty}^{\infty} \left|\eta_\circ'(r+t)\right|^2
dt dr\right)\\
&= |\eta_\circ|_2^2 + O^*(2.71 (2-\rho)^2 |\eta_\circ'|_2^2).
\end{aligned}\end{equation}

We will be working with $\eta_*$ supported on the non-negative
reals; we recall that $\eta_{\circ}$ is supported on $\lbrack 0,2\rbrack$.
Hence \begin{equation}\label{eq:jaram}\begin{aligned}
\int_0^\infty &\int_0^\infty
\eta_\circ(t_1) \eta_\circ(t_2) \eta_*\left(\frac{N}{x} - (t_1+t_2)\right)
dt_1 dt_2 
= \int_0^{\frac{N}{x}}  (\eta_\circ \ast \eta_\circ)(\rho)
\eta_*\left(\frac{N}{x} - \rho\right) d\rho\\
&= \int_0^{\frac{N}{x}} (|\eta_\circ|_2^2 + O^*(2.71 (2-\rho)^2 |\eta_\circ'|_2^2))
\cdot \eta_*\left(\frac{N}{x} - \rho\right) d\rho \\ &= 
|\eta_\circ|_2^2 \int_0^{\frac{N}{x}} \eta_*(\rho) d\rho + 
2.71 |\eta_\circ'|_2^2 \cdot O^*\left(
\int_0^{\frac{N}{x}} ((2-N/x)+\rho)^2 \eta_*(\rho) d\rho\right)
,\end{aligned}\end{equation}
provided that $N/x\geq 2$. 
We see that it will be wise to set $N/x$ very
slightly larger than $2$. As we said before, $\eta_*$ will be scaled so that it is concentrated
on a small interval $\lbrack 0,\epsilon)$.
\subsection{The smoothing function $\eta_*$: adapting minor-arc bounds}\label{subs:reddo}
Here the challenge is to define a smoothing function $\eta_*$ that is good
both for minor-arc estimates and for major-arc estimates. The two regimes
tend to favor different kinds of smoothing function. For minor-arc
estimates, both \cite{Tao} and \cite{Helf} use
\begin{equation}\label{eq:eta2}
\eta_2(t) = 4 \max(\log 2 - |\log 2 t|,0) = 
((2 I_{\lbrack 1/2,1\rbrack}) \ast_M (2 I_{\lbrack 1/2,1\rbrack}))(t),
\end{equation}
where $I_{\lbrack 1/2,1\rbrack}(t)$ is $1$ if $t\in \lbrack 1/2,1\rbrack$ and
$0$ otherwise. For major-arc estimates, we will use a function based on
\[\eta_{\heartsuit} = e^{-t^2/2}.\]
We will actually use here 
the function $t^2 e^{-t^2/2}$, whose Mellin transform is
$M\eta_{\heartsuit}(s+2)$ (by, e.g., \cite[Table 11.1]{Mellin}).)

We will follow the simple expedient of convolving the two smoothing functions,
one good for minor arcs, the other one for major arcs.
 In general, let $\varphi_1,\varphi_2:\lbrack 0,\infty)\to \mathbb{C}$.
It is easy to use bounds on sums of the form
\begin{equation}\label{eq:kostor}
S_{f,\varphi_1}(x) = \sum_n f(n) \varphi_1(n/x)\end{equation}
to bound sums of the form $S_{f,\varphi_1 \ast_M \varphi_2}$:
\begin{equation}\label{eq:chemdames}\begin{aligned}
S_{f,\varphi_1 \ast_M \varphi_2} &=
\sum_n f(n) (\varphi_1 \ast_M \varphi_2)\left(\frac{n}{x}\right) \\
&= 
 \int_0^{\infty} \sum_n f(n) \varphi_1\left(\frac{n}{w x}\right) \varphi_2(w)
   \frac{dw}{w}
= \int_{0}^{\infty} S_{f,\varphi_1}(w x) 
\varphi_2(w) \frac{dw}{w}.\end{aligned}\end{equation}
The same holds, of course, if $\varphi_1$ and $\varphi_2$ are switched, since
$\varphi_1 \ast_M \varphi_2 = \varphi_2 \ast_M \varphi_1$.
The only objection is that the bounds on (\ref{eq:kostor}) that we input
might not be valid, or non-trivial, when the argument $wx$ of $S_{f,\varphi_1}(wx)$ is very small.  Because of this, it is important that the functions
$\varphi_1$, $\varphi_2$ vanish at $0$,  and desirable that their first
 derivatives do so as well.


Let us see how this works out in practice for $\varphi_1 = \eta_2$. Here
$\eta_2:\lbrack 0,\infty)\to
\mathbb{R}$ is given by
\begin{equation}\label{eq:meichu}
\eta_2 = \eta_1 \ast_M \eta_1 = 4 \max(\log 2 - |\log 2 t|,0),\end{equation}
where $\eta_1 = 2 \cdot I_{\lbrack 1/2,1\rbrack}$. Bounding the
 sums $S_{\eta_2}(\alpha,x)$ on the minor arcs was the main subject of 
\cite{Helf}.

Before we use \cite[Main Thm.]{Helf}, we need an easy lemma so as to simplify
its statement.

\begin{lem}\label{lem:merkel}
For any $q\geq 1$ and any $r\geq \max(3,q)$,
\[\frac{q}{\phi(q)} < \digamma(r),\]
where
\begin{equation}\label{eq:koop}\begin{aligned}
\digamma(r) &= e^{\gamma} \log \log r + \frac{2.50637}{\log \log r}.
\end{aligned}\end{equation}
\end{lem}
\begin{proof}
Since $\digamma(r)$ is increasing for $r\geq 27$, the statement follows
immediately for $q\geq 27$ by \cite[Thm. 15]{MR0137689}:
\[\frac{q}{\phi(q)} < \digamma(q) \leq \digamma(r).\]
For $r<27$, it is clear that $q/\phi(q) \leq 2\cdot 3/(1\cdot 2) = 3$; it
is also easy to see that $\digamma(r)>e^\gamma\cdot 2.50637>3$ for all $r>e$.
\end{proof}

It is time to quote the main theorem in \cite{Helf}.
Let $x\geq x_0$, $x_0 = 2.16\cdot 10^{20}$. 
Let $2 \alpha = a/q +\delta/x$,
$q\leq Q$, $\gcd(a,q)=1$, $|\delta/x|\leq 1/q Q$, where
$Q = (3/4) x^{2/3}$. Then, if $3\leq q\leq x^{1/3}/6$, \cite[Main Thm.]{Helf}
gives us that
\begin{equation}\label{eq:monoro}
|S_{\eta_2}(\alpha,x)| \leq g_{x}\left(\max\left(1, \frac{|\delta|}{8}
\right) \cdot q\right) x,\end{equation}
where 
\begin{equation}\label{eq:syryza}
g_x(r) = \frac{(R_{x,2r} \log 2r + 0.5) \sqrt{\digamma(r)} + 2.5}{
\sqrt{2 r}} + \frac{L_r}{r} + 3.2 x^{-1/6},\end{equation}
with
\begin{equation}\label{eq:veror}\begin{aligned}
R_{x,t} &=  0.27125 \log 
\left(1 + \frac{\log 4 t}{2 \log \frac{9 x^{1/3}}{2.004 t}}\right)
 + 0.41415 \\
L_{t} &=  \digamma(t) \left(\log 2^{\frac{7}{4}} t^{\frac{13}{4}}
+ \frac{80}{9}\right) +
\log 2^{\frac{16}{9}} t^{\frac{80}{9}} + \frac{111}{5},
\end{aligned}\end{equation}
(We are using Lemma \ref{lem:merkel}
to bound all terms $1/\phi(q)$ appearing in \cite[Main Thm.]{Helf};
we are also using the obvious fact that, for $\delta_0 q$ fixed and $0<a<b$,
$\delta_0^a q^b$ is maximal when $\delta_0$ is minimal.)
If $q > x^{1/3}/6$, then, again by \cite[Main Thm.]{Helf},
\begin{equation}\label{eq:horm}
|S_{\eta_2}(\alpha,x)| \leq h(x) x,
\end{equation}
where
\begin{equation}\label{eq:flou}
h(x) = 0.2727 x^{-1/6} (\log x)^{3/2} + 1218 x^{-1/3} \log x .
\end{equation}

We will work with $x$ varying within a range, and so we must pay some
attention to the dependence of (\ref{eq:monoro}) and (\ref{eq:horm}) on $x$.
Let us prove two auxiliary lemmas on this.
\begin{lem}\label{lem:convet}
Let $g_x(r)$ be as in (\ref{eq:syryza}) and
$h(x)$ as in (\ref{eq:flou}). Then 
\[x\mapsto \begin{cases} h(x) &\text{if $x < (6 r)^3$}\\
g_{x}(r) &\text{if $x\geq (6 r)^3$}\end{cases}\]
is a decreasing function of $x$ for $r\geq 3$ fixed and $x\geq 21$.
\end{lem}
\begin{proof}
It is clear from the definitions 
that $x\mapsto h(x)$ (for $x\geq 21$) 
and $x\mapsto g_{x,0}(r)$ are both decreasing.
Thus, we simply have to show that $h(x_1) \geq g_{x_1,0}(r)$ for $x_1=
(6 r)^3$. Since $x_1\geq (6 \cdot 11)^3 > e^{12.5}$,
\[\begin{aligned}R_{x_1,2r} &\leq 0.27125 \log( 0.065 \log x_1 + 1.056) +0.41415\\
&\leq 0.27125 \log((0.065 + 0.0845) \log x_1) + 0.41415 \leq
0.27215 \log \log x_1.\end{aligned}\]
Hence \[\begin{aligned}
R_{x_1,2r} \log 2r + 0.5 &\leq 0.27215 \log \log x_1 \log x_1^{1/3} -
0.27215 \log 12.5 \log 3 + 0.5 \\ &\leq 0.09072 \log \log x_1 \log x_1
- 0.255.\end{aligned}\]
At the same time,
\begin{equation}\label{eq:pust}\begin{aligned}
\digamma(r) &= e^{\gamma} \log \log \frac{x_1^{1/3}}{6} + \frac{2.50637}{\log
  \log r}\leq 
e^{\gamma} \log \log x_1 - e^{\gamma} \log 3 + 1.9521 \\
&\leq e^{\gamma} \log \log x_1\end{aligned}\end{equation}
for $r\geq 37$, and we also get $\digamma(r)\leq e^{\gamma} \log \log x_1$
for $r\in \lbrack 11, 37\rbrack$ by the bisection method with $10$ iterations. 
Hence
\[\begin{aligned}
(R_{x_1,2r} &\log 2r + 0.5) \sqrt{\digamma(r)} +2.5 \\&\leq
(0.09072 \log \log x_1 \log x_1 - 0.255) \sqrt{e^\gamma \log \log x_1}
+2.5\\
&\leq 0.1211 \log x_1 (\log \log x_1)^{3/2} + 2,\end{aligned}\]
and so
\[\frac{(R_{x_1,2r} \log 2r + 0.5) \sqrt{\digamma(r)} +2.5}{\sqrt{2 r}}
\leq (0.21 \log x_1 (\log \log x_1)^{3/2} + 3.47) x_1^{-1/6}.\]

Now, by (\ref{eq:pust}),
\[\begin{aligned}
L_{r} &\leq e^{\gamma} \log \log x_1 \cdot \left(
\log 2^{\frac{7}{4}} (x_1^{1/3}/6)^{13/4} + \frac{80}{9}\right) + 
\log 2^{\frac{16}{9}} (x_1^{1/3}/6)^{\frac{80}{9}} + \frac{111}{5}\\
&\leq e^{\gamma} \log \log x_1 \cdot \left(\frac{13}{12} \log x_1 + 
4.28\right) + \frac{80}{27} \log x + 7.51.\end{aligned}\]
It is clear that
\[\frac{4.28 e^{\gamma} \log \log x_1  + \frac{80}{27} \log x_1 + 7.51}{
x_1^{1/3}/6} < 1218 x_1^{-1/3} \log x_1.\]
for $x_1\geq e$.

It remains to show that
\begin{equation}\label{eq:humo}
0.21 \log x_1 (\log \log x_1)^{3/2} + 3.47 + 3.2 + \frac{13}{12} e^{\gamma}
 x_1^{-1/6} \log x_1 \log \log x_1 
\end{equation}
is less than $0.2727 (\log x_1)^{3/2}$ for $x_1$ large enough.
Since $t\mapsto (\log t)^{3/2}/t^{1/2}$ is decreasing for $t>e^3$, we see
that 
\[\frac{0.21 \log x_1 (\log \log x_1)^{3/2} + 6.67
+ \frac{13}{12} e^{\gamma} x_1^{-1/6} \log x_1 \log \log x_1}{0.2727 (\log x_1)^{3/2}} < 1\]
for all $x_1\geq e^{34}$, simply because it is true for $x=e^{34}>e^{e^3}$.

We conclude that $h(x_1) \geq g_{x_1,0}(r) = g_{x_1,0}(x_1^{1/3}/6)$ for 
$x_1\geq e^{34}$. We check that 
$h(x_1)\geq g_{x_1,0}(x_1^{1/3}/6)$ for all $x_1\in \lbrack 5832,
e^{34}\rbrack$ as well by the bisection method (applied
to $\lbrack 5832,583200\rbrack$ and to $\lbrack 583200,e^{34}\rbrack$
with $30$ iterations -- in the latter interval, with $20$ initial iterations).
\end{proof}


\begin{lem}\label{lem:convog}
Let $R_{x,r}$ be as in (\ref{eq:syryza}). Then
$t\to R_{e^t,r}(r)$ is convex-up for $t\geq 3 \log 6 r$.
\end{lem}
\begin{proof}
Since $t\to e^{-t/6}$ and $t\to t$ are clearly convex-up, all we have to do
is to show that $t\to R_{e^t,r}$ is convex-up. In general,
since \[(\log f)'' = \left(\frac{f'}{f}\right)' = \frac{f'' f - (f')^2}{f^2},\]
a function of the form $(\log f)$ is convex-up exactly when $f'' f - (f')^2\geq 0$.
If $f(t) = 1 + a/(t-b)$, we have $f'' f - (f')^2\geq 0$
whenever
\[(t+a-b) \cdot (2 a) \geq  a^2,\]
i.e., $a^2 + 2 a t \geq 2 a b$,
and that certainly happens when $t\geq b$. In our case,
$b = 3 \log (2.004 r/9)$, and so $t\geq 3 \log 6 r$ implies $t\geq b$.
\end{proof}

Now we come to the point where we prove bounds on the exponential sums
$S_{\eta_*}(\alpha,x)$ (that is, sums based on the smoothing $\eta_*$)
based on our bounds (from \cite{Helf}) on the exponential sums 
$S_{\eta_2}(\alpha,x)$. This is straightforward, as promised.
\begin{prop}\label{prop:gorsh}
Let $x\geq K x_0$, $x_0=2.16\cdot 10^{20}$, $K\geq 1$.
Let $S_\eta(\alpha,x)$ be as in (\ref{eq:fellok}). 
Let
$\eta_* = \eta_2 \ast_M \varphi$, where $\eta_2$ is as in (\ref{eq:meichu})
and $\varphi: \lbrack 0,\infty)\to \lbrack 0, \infty)$ is continuous and in $L^1$.

Let $2\alpha = a/q + \delta/x$, $q\leq Q$, $\gcd(a,q)=1$,
$|\delta/x|\leq 1/qQ$, where $Q = (3/4) x^{2/3}$. If 
$q\leq (x/K)^{1/3}/6$, then
\begin{equation}\label{eq:kroz}
S_{\eta_*}(\alpha,x) \leq g_{x,\varphi}\left(\max\left(1,\frac{|\delta|}{8}
\right) q \right) \cdot |\varphi|_1 x,
\end{equation}
where
\begin{equation}\label{eq:basia}\begin{aligned}
g_{x,\varphi}(r) &=
\frac{(R_{x,K,\varphi,2 r} \log 2 r + 0.5) \sqrt{\digamma(r)} + 2.5}{\sqrt{2
    r}}  + 
\frac{L_r}{r} +3.2 K^{1/6} x^{-1/6},\\
R_{x,K,\varphi,t} &= R_{x,t} + (R_{x/K,t} - R_{x,t})
\frac{C_{\varphi,2,K}/|\varphi|_1}{\log K}
\end{aligned}\end{equation}
with $R_{x,t}$ and $L_{r}$ are as in (\ref{eq:veror}),
and
\begin{equation}\label{eq:cecidad}
C_{\varphi,2,K} = - \int_{1/K}^1 \varphi(w) \log w\; dw.\end{equation}

If $q>(x/K)^{1/3}/6$, then
\[|S_{\eta_*}(\alpha,x)| \leq h_\varphi(x/K)\cdot |\varphi|_1 x,\]
where 
\begin{equation}\label{eq:midin}\begin{aligned}
h_{\varphi}(x) &= h(x) + C_{\varphi,0,K}/|\varphi|_1,\\
C_{\varphi,0,K} &= 1.04488 \int_0^{1/K} |\varphi(w)| dw 
\end{aligned}\end{equation}
and $h(x)$ is as in (\ref{eq:flou}).
\end{prop}
\begin{proof}
By (\ref{eq:chemdames}),
\[\begin{aligned}
S_{\eta_*}(\alpha,x) &= \int_0^{1/K} S_{\eta_2}(\alpha,w x) 
\varphi(w) \frac{dw}{w} +
\int_{1/K}^\infty S_{\eta_2}(\alpha,w x) \varphi(w) \frac{dw}{w}.
\end{aligned}\]
We bound the first integral by the trivial estimate 
$|S_{\eta_2}(\alpha,w x)|\leq |S_{\eta_2}(0,w x)|$ and 
Cor.~\ref{cor:austeria}:
\[\begin{aligned}\int_0^{1/K} |S_{\eta_2}(0,w x)| 
\varphi(x) \frac{dw}{w} &\leq 1.04488 \int_0^{1/K} w x \varphi(w)
\frac{dw}{w} \\ &= 1.04488 x \cdot \int_0^{1/K} \varphi(w)
dw.\end{aligned}\]

If $w\geq 1/K$, then $w x\geq x_0$, and we can use 
(\ref{eq:monoro}) or (\ref{eq:horm}). 
If $q>(x/K)^{1/3}/6$, then $|S_{\eta_2}(\alpha,w x)|\leq h(x/K) w x$
by (\ref{eq:horm}); moreover, $|S_{\eta_2}(\alpha,y)|\leq h(y) y$ for
$x/K \leq y < (6 q)^3$ (by (\ref{eq:horm})) and
$|S_{\eta_2}(\alpha,y)|\leq g_{y,1}(r)$ for $y\geq (6 q)^3$ 
(by (\ref{eq:monoro})). Thus, Lemma \ref{lem:convet} gives us that
\[\begin{aligned}\int_{1/K}^\infty |S_{\eta_2}(\alpha,w x)| 
\varphi(w) \frac{dw}{w} &\leq 
\int_{1/K}^\infty h(x/K) w x \cdot \varphi(w) \frac{dw}{w} \\ &= 
h(x/K) x \int_{1/K}^\infty \varphi(w) dw \leq h(x/K) |\varphi|_1 \cdot x.
\end{aligned}\]

If $q\leq (x/K)^{1/3}/6$, we always use (\ref{eq:monoro}).
We can use the coarse bound
\[\int_{1/K}^\infty 3.2 x^{-1/6} \cdot w x \cdot \varphi(w) \frac{dw}{w}
\leq  3.2 K^{1/6} |\varphi|_1 x^{5/6}\]
Since $L_{r}$ does not depend on $x$,
\[
\int_{1/K}^{\infty} \frac{L_{r}}{r} \cdot w x\cdot 
\varphi(w) \frac{dw}{w} \leq 
\frac{L_{r}}{r} |\varphi|_1 x.\]

By Lemma \ref{lem:convog} and $q\leq (x/K)^{1/3}/6$, 
$y\mapsto R_{e^y,t}$ is convex-up and decreasing for $y\in \lbrack \log(x/K),\infty)$.
Hence
\[R_{w x,t} \leq \begin{cases} \frac{\log w}{\log \frac{1}{K}} R_{x/K,t} 
+ \left(1 - \frac{\log w}{\log \frac{1}{K}}\right) R_{x,t}
&\text{if $w<1$,}\\ R_{x,t} &\text{if $w\geq 1$.}\end{cases}\]
Therefore
\[\begin{aligned}
\int_{1/K}^{\infty} &R_{w x,t} \cdot w x \cdot \varphi(w) \frac{dw}{w} \\ &\leq
\int_{1/K}^1 \left(\frac{\log w}{\log \frac{1}{K}} R_{x/K,t} 
+ \left(1 - \frac{\log w}{\log \frac{1}{K}}\right) R_{x,t}\right) x \varphi(w)
dw  + \int_1^{\infty} R_{x,t} \varphi(w) x dw\\
&\leq R_{x,t} x \cdot \int_{1/K}^\infty \varphi(w) dw + 
(R_{x/K,t} - R_{x,t}) \frac{x}{\log K} \int_{1/K}^1  \varphi(w) \log w dw\\
&\leq \left(R_{x,t} |\varphi|_1 +
(R_{x/K,t} - R_{x,t}) \frac{C_{\varphi,2}}{\log K}\right) \cdot x,
\end{aligned}\]
where
\[C_{\varphi,2,K} = - \int_{1/K}^1 \varphi(w) \log w\; dw.\]
\end{proof}

We finish by proving a couple more lemmas.

\begin{lem}\label{lem:vinc}
Let $x> K\cdot  (6 e)^3$, $K>1$. 
Let $\eta_* = \eta_2 \ast_M \varphi$, where $\eta_2$ is as in (\ref{eq:meichu})
and $\varphi: \lbrack 0,\infty)\to \lbrack 0, \infty)$ is continuous and in
$L^1$. Let $g_{x,\varphi}$ be as in (\ref{eq:basia}).

Then $g_{x,\varphi}(r)$ is a decreasing function of $r$ for $r\geq 175$.
\end{lem}
\begin{proof}
Taking derivatives, we can easily see that
\begin{equation}\label{eq:hopo}r\mapsto \frac{\log \log r}{r},\;\;\;
r\mapsto \frac{\log r}{r},\;\;\; 
r\mapsto \frac{(\log r)^2 \log \log r}{r}\end{equation}
are decreasing for $r\geq 20$. The same is true if 
$\log \log r$ is replaced by $\digamma(r)$, since $\digamma(r)/\log \log r$
is a decreasing function for $r\geq e$.
Since $(C_{\varphi,2}/|\phi|_1)/\log K \leq 1$, we see that it is enough
to prove that $r\mapsto R_{y,t} \log 2r \sqrt{\log \log r}/\sqrt{2 r}$ is
decreasing on $r$ for $y=x$ and $y=x/K$ (under the assumption that $r\geq 175$).

Looking at (\ref{eq:veror}) and at (\ref{eq:hopo}), it remains only to check that
\begin{equation}\label{eq:horko}
r\mapsto \log \left(1 + \frac{\log 8 r}{2 \log \frac{9 x^{1/3}}{4.008 r}}
\right) \sqrt{\frac{\log \log r}{r}}\end{equation}
is decreasing on $r$ for $r\geq 175$.
Taking logarithms, and then derivatives, we see that we have to show that
\[\frac{\frac{\frac{1}{r} \ell +
\frac{\log 8 r}{r}}{2\ell^2}}{\left(1 + \frac{\log 8 r}{2 \ell}\right)
\log \left(1 + \frac{\log 8 r}{2 \ell}\right)}
+ \frac{1}{2 r \log r \log \log r} 
< \frac{1}{2 r},\]
where $\ell = \log \frac{9 x^{1/3}}{4.008 r}$. Since $r\leq x^{1/3}/6$,
$\ell \geq \log 54/4.008 > 2.6$. Thus, it is enough to ensure that
\begin{equation}\label{eq:hut}
\frac{2/2.6}{\left(1 + \frac{\log 8 r}{2 \ell}\right)
\log \left(1 + \frac{\log 8 r}{2 \ell}\right)}
  + \frac{1}{\log r \log \log r} < 1.\end{equation}
Since this is true for $r = 175$ and the left side is decreasing on $r$,
the inequality is true for all $r\geq 175$.

\end{proof}

\begin{lem}\label{lem:gosia}
Let $x\geq 10^{25}$. Let $\phi:\lbrack 0,\infty)\to \lbrack 0,\infty)$
be continuous and in $L^1$. Let $g_{x,\phi}(r)$ and $h(x)$ be as in (\ref{eq:basia}) and (\ref{eq:flou}), respectively. Then 
\[g_{x,\phi}\left(\frac{3}{8} 
x^{4/15}\right)\geq h(2 x/\log x).\]
\end{lem}
\begin{proof}
We can bound $g_{x,\phi}(r)$ from below by
\[gm_x(r) = \frac{(R_{x,r} \log 2 r + 0.5) \sqrt{\digamma(r)} + 2.5}{
\sqrt{2 r}}.\]
Let $r = (3/8) x^{4/15}$.
Using the assumption that $x\geq 10^{25}$, we see that
\[\begin{aligned}
R_{x,r} &= 0.27125 \log \left(1 + \frac{\log\left(\frac{3 x^{4/15}}{2}
\right)}{2 \log \left(\frac{9}{2.004 \cdot \frac{3}{8}}
\cdot x^{\frac{1}{3} - 4/15}\right)}
\right) + 0.41415\geq 0.63368.\end{aligned}\]
Using $x\geq 10^{25}$ again, we get that
\[\digamma(r) = e^\gamma \log \log r + \frac{2.50637}{\log \log r} \geq
5.68721.\]
Since $\log 2 r = (4/15) \log x + \log(3/4)$,
we conclude that
\[gm_x(r) \geq
\frac{ 0.40298 \log x + 3.25765}{\sqrt{3/4}\cdot x^{2/15}}
.\]
Recall that
\[h(x) = \frac{0.2727 (\log x)^{3/2}}{x^{1/6}} + \frac{1218 \log x}{x^{1/3}} .\]
A simple derivative test gives us that
\[x\mapsto
 \frac{(\log x + 3)/x^{2/15}}{
(\log(x/\log x))^{3/2}/(x/\log x)^{1/6}}\]
is increasing for $x\geq 10^{25}$ (and indeed for 
$x\geq e^{28}$, or even well before then)
and that $(1/x^{2/15})/((\log(x/\log x))/(x/\log x)^{1/3})$ is increasing for $x\geq e^{7}$.
Since
\[\begin{aligned}\frac{0.40298 (\log x+3)}{\sqrt{3/4} \cdot x^{2/15}} 
&\geq \frac{0.2727 (\log(2 x/\log x))^{3/2}}{(2 x/\log x)^{1/6}},\\
\frac{3.25765 - 3\cdot 0.40298}{\sqrt{3/4}\cdot x^{2/15}} &\geq
\frac{1218 \log (2 x/\log(x)) }{(2 x/\log(x))^{1/3}}\end{aligned}\]
for $x\geq 10^{25}$, we are done.
%
\end{proof}

\section{The $\ell_2$ norm and the large sieve}\label{sec:intri}

Our aim here is to give a bound on the $\ell_2$ norm of an
exponential sum over the minor arcs. While we care
about an exponential sum in particular, we will prove a result valid
for all exponential sums $S(\alpha,x) = \sum_n a_n e(\alpha n)$ with 
$a_n$ of prime support.

We start by adapting ideas from Ramar\'e's version of the large sieve
for primes to estimate $\ell_2$ norms over parts of
the circle (\S \ref{subs:ramar}). 
We are left with the task of giving an
explicit bound on the factor in Ramar\'e's work; this we do in
\S \ref{subs:boquo}. As a side effect, this finally gives a fully explicit
large sieve for primes that is asymptotically optimal, meaning a sieve that
does not have a spurious factor of $e^\gamma$ in front; this was
an arguably important gap in the literature.

\subsection{The $\ell_2$ norm over arcs: variations on the large
sieve for primes}\label{subs:ramar}

We are trying to estimate an integral 
$\int_{\mathbb{R}/\mathbb{Z}} |S(\alpha)|^3 d\alpha$. Rather than bound it by
$|S|_\infty |S|_2^2$, we can use the fact that large (``major'') values
of $S(\alpha)$ have to be multiplied only by 
$\int_\mathfrak{M} |S(\alpha)|^2 d\alpha$,
where $\mathfrak{M}$ is a union (small in measure) of minor arcs. Now,
can we give an upper bound for $\int_\mathfrak{M} |S(\alpha)|^2 d\alpha$ 
better than
$|S|_2^2 = \int_{\mathbb{R}/\mathbb{Z}} |S(\alpha)|^2 d\alpha$?

The first version of \cite{Helf} gave an estimate on that integral
using a technique due to Heath-Brown, which in turn rests
on an inequality of Montgomery's (\cite[(3.9)]{MR0337847}; see also, e.g.,
\cite[Lem.~7.15]{MR2061214}).
 The technique was 
communicated by Heath-Brown to the present author, who communicated it to Tao 
 (\cite[Lem.~4.6]{Tao} and adjoining comments).
 We will be able to do better than that estimate here.

The role played by Montgomery's inequality in Heath-Brown's method is played
here by a result of Ramar\'e's (\cite[Thm. 2.1]{MR2493924}; see also
\cite[Thm. 5.2]{MR2493924}). The following proposition is based on Ramar\'e's 
result, or rather on one possible proof of it. Instead of using
the result as stated in \cite{MR2493924}, we will actually be using elements of
the proof of \cite[Thm. 7A]{MR0371840}, credited to Selberg. Simply integrating
Ramar\'e's inequality would give a non-trivial if slightly worse bound.
\begin{prop}\label{prop:ramar}
Let $\{a_n\}_{n=1}^\infty$, $a_n\in \mathbb{C}$, be supported on the primes.
Assume that $\{a_n\}$ is in $\ell_1\cap \ell_2$ and that $a_n=0$ for $n\leq \sqrt{x}$.
Let $Q_0\geq 1$, $\delta_0\geq 1$ be such that $\delta_0 Q_0^2 \leq x/2$; set
$Q = \sqrt{x/2 \delta_0} \geq Q_0$. Let
\begin{equation}\label{eq:jokor}
\mathfrak{M} = \bigcup_{q\leq Q_0} \mathop{\bigcup_{a \mo q}}_{(a,q)=1}
 \left(\frac{a}{q} - \frac{\delta_0 r}{q x}, 
\frac{a}{q} + \frac{\delta_0 r}{q x}\right)
.\end{equation}

Let $S(\alpha) = \sum_n a_n e(\alpha n)$ for $\alpha \in \mathbb{R}/\mathbb{Z}$.
Then
\[\int_{\mathfrak{M}} \left|S(\alpha)\right|^2 d\alpha \leq
\left(\max_{q\leq Q_0} \max_{s\leq Q_0/q} \frac{G_q(Q_0/sq)}{G_q(Q/sq)}\right)
\sum_n |a_n|^2,\]
where
\begin{equation}\label{eq:malbo}
G_q(R) = \mathop{\sum_{r\leq R}}_{(r,q)=1} \frac{\mu^2(r)}{\phi(r)}.\end{equation}
\end{prop}
\begin{proof}
By (\ref{eq:jokor}),
\begin{equation}\label{eq:malkr}\int_{\mathfrak{M}} \left|S(\alpha)\right|^2 d\alpha = 
\sum_{q\leq Q_0} 
\int_{-\frac{\delta_0 Q_0}{q x}}^{\frac{\delta_0 Q_0}{q x}}
\mathop{\sum_{a \mo q}}_{(a,q)=1} 
\left|S\left(\frac{a}{q}+\alpha\right)\right|^2 d\alpha.\end{equation}
Thanks to the last equations of \cite[p. 24]{MR0371840} and
\cite[p. 25]{MR0371840},
\[\mathop{\sum_{a \mo q}}_{(a,q)=1} \left|S\left(\frac{a}{q}\right)\right|^2 =
\frac{1}{\phi(q)} \mathop{\mathop{\sum_{q^*|q}}_{(q^*,q/q^*)=1}}_{\mu^2(q/q^*) = 1}
q^* \cdot \sume_{\chi \mo q^*} \left|\sum_n a_n \chi(n)\right|^2\]
for every $q\leq \sqrt{x}$, where we use the assumption that 
$n$ is prime and $>\sqrt{x}$ (and thus coprime to $q$) when $a_n\ne 0$.
Hence
\[\begin{aligned}
&\int_{\mathfrak{M}} \left|S(\alpha)\right|^2 d\alpha = 
\sum_{q\leq Q_0} 
 \mathop{\mathop{\sum_{q^*|q}}_{(q^*,q/q^*)=1}}_{\mu^2(q/q^*) = 1}
q^* \int_{-\frac{\delta_0 Q_0}{q x}}^{\frac{\delta_0 Q_0}{q x}} \frac{1}{\phi(q)}
\left|\sum_n a_n e(\alpha n) \chi(n)\right|^2 d\alpha
\\
&= \sum_{q^*\leq Q_0} \frac{q^*}{\phi(q^*)} 
\mathop{\sum_{r\leq Q_0/q^*}}_{(r,q^*)=1}
\frac{\mu^2(r)}{\phi(r)} 
\int_{-\frac{\delta_0 Q_0}{q^* r x}}^{\frac{\delta_0 Q_0}{q^* r x}}
\sume_{\chi \mo q^*}
\left|\sum_n a_n e(\alpha n) \chi(n)\right|^2 d\alpha\\
&= \sum_{q^*\leq Q_0} \frac{q^*}{\phi(q^*)} 
\int_{-\frac{\delta_0 Q_0}{q^* x}}^{\frac{\delta_0 Q_0}{q^* x}}
\mathop{\sum_{r\leq \frac{Q_0}{q^*} \min\left(1,\frac{\delta_0}{|\alpha| x}\right)}}_{(r,q^*)=1}
\frac{\mu^2(r)}{\phi(r)} 
\sume_{\chi \mo q^*}
\left|\sum_n a_n e(\alpha n) \chi(n)\right|^2 d\alpha
\end{aligned}\]
Here $|\alpha|\leq \delta_0 Q_0/q^* x$ implies $(Q_0/q) \delta_0/|\alpha| x
\geq 1$. Therefore,
\begin{equation}\label{eq:ail}\begin{aligned}
\int_{\mathfrak{M}} &\left|S(\alpha)\right|^2 d\alpha \leq
\left(\max_{q^*\leq Q_0} \max_{s\leq Q_0/q^*} \frac{G_{q^*}(Q_0/sq^*)}{G_{q^*}(Q/sq^*)}\right)
\cdot \Sigma,\end{aligned}\end{equation}
where \[\begin{aligned}
\Sigma &= \sum_{q^*\leq Q_0} \frac{q^*}{\phi(q^*)} 
\int_{-\frac{\delta_0 Q_0}{q^* x}}^{\frac{\delta_0 Q_0}{q^* x}}
\mathop{\sum_{r\leq \frac{Q}{q^*} \min\left(1,\frac{\delta_0}{|\alpha| x}\right)}}_{(r,q^*)=1}
\frac{\mu^2(r)}{\phi(r)} 
\sume_{\chi \mo q^*}
\left|\sum_n a_n e(\alpha n) \chi(n)\right|^2 d\alpha\\
&\leq 
\sum_{q\leq Q} \frac{q}{\phi(q)} 
\mathop{\sum_{r\leq Q/q}}_{(r,q)=1}
\frac{\mu^2(r)}{\phi(r)} 
\int_{-\frac{\delta_0 Q}{q r x}}^{\frac{\delta_0 Q}{q r x}}
\sume_{\chi \mo q}
\left|\sum_n a_n e(\alpha n) \chi(n)\right|^2 d\alpha.
\end{aligned}\]
As stated in the proof of \cite[Thm. 7A]{MR0371840},
\[
\overline{\chi}(r) \chi(n) \tau(\overline{\chi}) c_r(n) =
\mathop{\sum_{b=1}^{qr}}_{(b,qr)=1} \overline{\chi}(b) e^{2\pi i n \frac{b}{q r}}
\]
for $\chi$ primitive of modulus $q$. Here $c_r(n)$ stands for the
Ramanujan sum \[c_r(n) = \mathop{\sum_{u \mo r}}_{(u,r)=1} e^{2\pi n u/r}.\] 
For $n$ coprime to $r$, $c_r(n) = \mu(r)$. Since $\chi$ is primitive,
$|\tau(\overline{\chi})| = \sqrt{q}$.
Hence, for $r\leq \sqrt{x}$ coprime to $q$,
\[q \left|\sum_n a_n e(\alpha n) \chi(n)\right|^2  =
\left| \mathop{\sum_{b=1}^{qr}}_{(b,qr)=1} \overline{\chi}(b) S\left(\frac{b}{q r}
+\alpha\right)\right|^2 .
\]
Thus,
\[\begin{aligned} \Sigma &=
\sum_{q\leq Q} \mathop{\sum_{r\leq Q/q}}_{(r,q)=1}
\frac{\mu^2(r)}{\phi(r q)} 
\int_{-\frac{\delta_0 Q}{q r x}}^{\frac{\delta_0 Q}{q r x}}
\sume_{\chi \mo q} 
\left| \mathop{\sum_{b=1}^{qr}}_{(b,qr)=1} \overline{\chi}(b) S\left(\frac{b}{q r}
+\alpha\right)\right|^2 d\alpha\\
&\leq 
\sum_{q\leq Q} \frac{1}{\phi(q)}
\int_{-\frac{\delta_0 Q}{q x}}^{\frac{\delta_0 Q}{q x}}
\sum_{\chi \mo q} 
\left| \mathop{\sum_{b=1}^{q}}_{(b,q)=1} \overline{\chi}(b) S\left(\frac{b}{q}
+\alpha\right)\right|^2 d\alpha\\
&= \sum_{q\leq Q} 
\int_{-\frac{\delta_0 Q}{q x}}^{\frac{\delta_0 Q}{q x}}
\mathop{\sum_{b=1}^{q}}_{(b,q)=1} \left|S\left(\frac{b}{q}+\alpha\right)\right|^2 
d\alpha.
\end{aligned}\]
Let us now check that the intervals $(b/q-\delta_0 Q/qx,b/q+\delta_0 Q/q x)$
do not overlap. Since $Q = \sqrt{x/2\delta_0}$, we see that
$\delta_0 Q/q x = 1/2q Q$.
The difference between two distinct fractions $b/q$,
$b'/q'$ is at least $1/q q'$. For $q,q'\leq Q$, $1/q q'\geq 1/2 q Q + 
1/2 Q q'$. Hence the intervals around $b/q$ and $b'/q'$ do not overlap.
We conclude that
\[\Sigma \leq \int_{\mathbb{R}/\mathbb{Z}} \left|S(\alpha)\right|^2 
= \sum_n |a_n|^2,\]
and so, by (\ref{eq:ail}), we are done. 
\end{proof}

We will actually use Prop.~\ref{prop:ramar} in the slightly modified form
given by the following statement.
\begin{prop}\label{prop:bellen}
Let $\{a_n\}_{n=1}^\infty$, $a_n\in \mathbb{C}$, be supported on the primes.
Assume that $\{a_n\}$ is in $\ell_1\cap \ell_2$ and that $a_n=0$ for $n\leq \sqrt{x}$.
Let $Q_0\geq 1$, $\delta_0\geq 1$ be such that $\delta_0 Q_0^2 \leq x/2$; set
$Q = \sqrt{x/2 \delta_0} \geq Q_0$. Let $\mathfrak{M}=\mathfrak{M}_{\delta_0,Q_0}$ be
as in (\ref{eq:majdef}).

Let $S(\alpha) = \sum_n a_n e(\alpha n)$ for $\alpha \in \mathbb{R}/\mathbb{Z}$.
Then
\[\int_{\mathfrak{M}_{\delta_0,Q_0}} \left|S(\alpha)\right|^2 d\alpha \leq
\left(\mathop{\max_{q\leq 2 Q_0}}_{\text{$q$ even}} \max_{s\leq 2 Q_0/q} \frac{G_{q}(2 Q_0/sq)}{G_{q}(2 Q/sq)}\right)
\sum_n |a_n|^2,\]
where
\begin{equation}\label{eq:malar}
G_q(R) = \mathop{\sum_{r\leq R}}_{(r,q)=1} \frac{\mu^2(r)}{\phi(r)}.\end{equation}
\end{prop}
\begin{proof}
By (\ref{eq:majdef}),
\[\begin{aligned}
\int_{\mathfrak{M}} \left|S(\alpha)\right|^2 d\alpha &= 
\mathop{\sum_{q\leq Q_0}}_{\text{$q$ odd}} 
\int_{-\frac{\delta_0 Q_0}{2 q x}}^{\frac{\delta_0 Q_0}{2 q x}}
\mathop{\sum_{a \mo q}}_{(a,q)=1} 
\left|S\left(\frac{a}{q}+\alpha\right)\right|^2 d\alpha\\
&+ \mathop{\sum_{q\leq Q_0}}_{\text{$q$ even}} 
\int_{-\frac{\delta_0 Q_0}{q x}}^{\frac{\delta_0 Q_0}{q x}}
\mathop{\sum_{a \mo q}}_{(a,q)=1} 
\left|S\left(\frac{a}{q}+\alpha\right)\right|^2 d\alpha .
\end{aligned}\]
We proceed as in the proof of Prop.~\ref{prop:ramar}.
We still have (\ref{eq:malkr}). Hence
$\int_{\mathfrak{M}} \left|S(\alpha)\right|^2 d\alpha$ equals
\[\begin{aligned}
&\mathop{\sum_{q^*\leq Q_0}}_{\text{$q^*$ odd}} \frac{q^*}{\phi(q^*)} 
\int_{-\frac{\delta_0 Q_0}{2 q^* x}}^{\frac{\delta_0 Q_0}{2 q^* x}}
\mathop{\sum_{r\leq \frac{Q_0}{q^*} \min\left(1,\frac{\delta_0}{2 |\alpha| x}\right)}}_{(r,2 
q^*)=1}
\frac{\mu^2(r)}{\phi(r)} 
\sume_{\chi \mo q^*}
\left|\sum_n a_n e(\alpha n) \chi(n)\right|^2 d\alpha\\
+
&\mathop{\sum_{q^*\leq 2 Q_0}}_{\text{$q^*$ even}} \frac{q^*}{\phi(q^*)} 
\int_{-\frac{\delta_0 Q_0}{q^* x}}^{\frac{\delta_0 Q_0}{q^* x}}
\mathop{\sum_{r\leq \frac{2 Q_0}{q^*} \min\left(1,\frac{\delta_0}{2 |\alpha| x}\right)}}_{(r,
q^*)=1}
\frac{\mu^2(r)}{\phi(r)} 
\sume_{\chi \mo q^*}
\left|\sum_n a_n e(\alpha n) \chi(n)\right|^2 d\alpha.
\end{aligned}\]
(The sum with $q$ odd and $r$ even is equal to the first sum; 
hence the factor of $2$ in front.)
Therefore,
\begin{equation}\label{eq:nimaduro}\begin{aligned}
\int_{\mathfrak{M}} \left|S(\alpha)\right|^2 d\alpha
&\leq
\left(\mathop{\max_{q^*\leq Q_0}}_{\text{$q^*$ odd}} \max_{s\leq Q_0/q^*} \frac{G_{2 q^*}(Q_0/sq^*)}{G_{2 q^*}(Q/sq^*)}\right)
\cdot 2 \Sigma_1\\
&+
\left(\mathop{\max_{q^*\leq 2 Q_0}}_{\text{$q^*$ even}} \max_{s\leq 2 Q_0/q^*} 
\frac{G_{q^*}(2Q_0/sq^*)}{G_{q^*}(2Q/sq^*)}\right)
\cdot \Sigma_2,\end{aligned}\end{equation}
where
\[\begin{aligned}
\Sigma_1 &= \mathop{\sum_{q\leq Q}}_{\text{$q$ odd}} \frac{q}{\phi(q)} 
\mathop{\sum_{r\leq Q/q}}_{(r,2 q)=1}
\frac{\mu^2(r)}{\phi(r)} 
\int_{-\frac{\delta_0 Q}{2 q r x}}^{\frac{\delta_0 Q}{2 q r x}}
\sume_{\chi \mo q}
\left|\sum_n a_n e(\alpha n) \chi(n)\right|^2 d\alpha\\
&= \mathop{\sum_{q\leq Q}}_{\text{$q$ odd}} \frac{q}{\phi(q)} 
\mathop{\mathop{\sum_{r\leq 2 Q/q}}_{(r,q)=1}}_{\text{$r$ even}}
 \frac{\mu^2(r)}{\phi(r)} 
\int_{-\frac{\delta_0 Q}{q r x}}^{\frac{\delta_0 Q}{q r x}}
\sume_{\chi \mo q}
\left|\sum_n a_n e(\alpha n) \chi(n)\right|^2 d\alpha.
\\
\Sigma_2 &=
\mathop{\sum_{q\leq 2 Q}}_{\text{$q$ even}} \frac{q}{\phi(q)} 
\mathop{\sum_{r\leq 2 Q/q}}_{(r,q)=1}
\frac{\mu^2(r)}{\phi(r)} 
\int_{-\frac{\delta_0 Q}{q r x}}^{\frac{\delta_0 Q}{q r x}}
\sume_{\chi \mo q}
\left|\sum_n a_n e(\alpha n) \chi(n)\right|^2 d\alpha.
\end{aligned}\]
The two expressions within parentheses in (\ref{eq:nimaduro}) are actually equal.

Much as before, using \cite[Thm.~7A]{MR0371840}, we obtain that
\[\begin{aligned}
\Sigma_1 &\leq \mathop{\sum_{q\leq Q}}_{\text{$q$ odd}} \frac{1}{\phi(q)}
\int_{-\frac{\delta_0 Q}{2 q x}}^{\frac{\delta_0 Q}{2 q x}}
\mathop{\sum_{b=1}^{q}}_{(b,q)=1} \left|S\left(\frac{b}{q} + \alpha\right)\right|^2 
d\alpha,\\
\Sigma_1 + \Sigma_2 &\leq \mathop{\sum_{q\leq 2 Q}}_{\text{$q$ even}} \frac{1}{\phi(q)}
\int_{-\frac{\delta_0 Q}{q x}}^{\frac{\delta_0 Q}{q x}}
\mathop{\sum_{b=1}^{q}}_{(b,q)=1} \left|S\left(\frac{b}{q} + \alpha\right)\right|^2 
d\alpha.
\end{aligned}\]
Let us now check that the intervals of integration 
$(b/q-\delta_0 Q/2 q x,b/q+\delta_0 Q/2 q x)$ (for $q$ odd),
$(b/q-\delta_0 Q/q x,b/q+\delta_0 Q/q x)$ (for $q$ even)
do not overlap. Recall that $\delta_0 Q/q x = 1/2 q Q$.
The absolute value of the difference between two distinct
 fractions $b/q$, $b'/q'$ is at least $1/q q'$.
For $q, q'\leq Q$ odd, this is larger than
$1/4 q Q + 1/4 Q q'$, and so the intervals do not overlap. For $q\leq Q$ odd
and $q'\leq 2 Q$ even (or vice versa), $1/q q'\geq 1/4 q Q + 1/2 Q q'$,
and so, again the intervals do not overlap. If $q\leq Q$ and $q'\leq Q$
are both even, then $|b/q - b'/q'|$ is actually $\geq 2/q q'$. Clearly,
$2/q q' \geq 1/2 q Q + 1/2 Q q'$, and so again there is no overlap.
We conclude that
\[2 \Sigma_1 + \Sigma_2 \leq \int_{\mathbb{R}/\mathbb{Z}} \left|S(\alpha)\right|^2 
= \sum_n |a_n|^2.\]
\end{proof}

\subsection{Bounding the quotient in the large sieve for primes}\label{subs:boquo}

The estimate given by Proposition \ref{prop:ramar} involves the quotient
\begin{equation}\label{eq:gusto}
\max_{q\leq Q_0} \max_{s\leq Q_0/q}
\frac{G_q(Q_0/sq)}{G_q(Q/sq)},\end{equation}
where $G_q$ is as in (\ref{eq:malbo}).
The appearance of such a quotient (at least for $s=1$) is typical of
Ramar\'e's version of the large sieve for primes; see, e.g., \cite{MR2493924}.
We will see how to bound such a quotient in a way that is essentially
optimal, not just asymptotically, but also in the ranges that are most
relevant to us. (This includes, for example, $Q_0\sim 10^6$, $Q\sim
10^{15}$.)

As the present work shows, Ramar\'e's work gives bounds that are, 
in some contexts, better than those of other large sieves for primes by a 
constant factor (approaching $e^\gamma = 1.78107\dotsc$).
Thus, giving a fully explicit
and nearly optimal bound for (\ref{eq:gusto}) is a task of clear
general relevance, besides being needed for our main goal.

We will obtain bounds for $G_q(Q_0/sq)/G_q(Q/sq)$ when $Q_0\leq 2\cdot 10^{10}$,
$Q\geq Q_0^2$. As we shall see, our bounds will be best when $s=q=1$
-- or, sometimes, when $s=1$ and $q=2$ instead.

Write $G(R)$ for $G_1(R) = \sum_{r\leq R} \mu^2(r)/\phi(r)$.
We will need several estimates for $G_q(R)$ and $G(R)$.
As stated in \cite[Lemma 3.4]{MR1375315},
\begin{equation}\label{eq:cante}
G(R) \leq \log R + 1.4709 \end{equation}
for $R\geq 1$.
By \cite[Lem.~7]{MR0374060}, 
\begin{equation}\label{eq:jondo}
G(R)\geq \log R + 1.07
\end{equation}
for $R\geq 6$. 
There is also the trivial bound
\begin{equation}\label{eq:coro}
\begin{aligned}G(R) &= \sum_{r\leq R} \frac{\mu^2(r)}{\phi(r)} =
\sum_{r\leq R} \frac{\mu^2(r)}{r} \prod_{p|r} \left(1
  -\frac{1}{p}\right)^{-1}\\
&= \sum_{r\leq R} \frac{\mu^2(r)}{r} \prod_{p|r} \sum_{j\geq 1}
\frac{1}{p^j} \geq \sum_{r\leq R} \frac{1}{r} > \log R.\end{aligned}
\end{equation}
The following bound, also well-known and easy,
\begin{equation}\label{eq:hosmo}
G(R)\leq \frac{q}{\phi(q)} G_q(R)\leq G(R q),
\end{equation}
can be obtained by multiplying $G_q(R) = \sum_{r\leq R: (r,q)=1}
\mu^2(r)/\phi(r)$ term-by-term by $q/\phi(q) = \prod_{p|q} (1 +
1/\phi(p))$. 

We will also use
Ramar\'e's estimate from \cite[Lem.~3.4]{MR1375315}:
\begin{equation}\label{eq:malito}
G_d(R) = \frac{\phi(d)}{d} \left(\log R + c_E + \sum_{p|d}
  \frac{\log p}{p}\right) + O^*\left(7.284 R^{-1/3} f_1(d)\right)
\end{equation}
for all $d\in \mathbb{Z}^+$ and all $R\geq 1$, where
\begin{equation}\label{eq:assur}
f_1(d) = \prod_{p|d} (1+p^{-2/3}) \left(1+\frac{p^{1/3}+p^{2/3}}{p
(p-1)}\right)^{-1}\end{equation}
and
\begin{equation}\label{eq:garno}
c_E = \gamma + \sum_{p\geq 2} \frac{\log p}{p (p-1)} = 1.3325822\dotsc
\end{equation}
by \cite[(2.11)]{MR0137689}. 

If $R\geq 182$, then
\begin{equation}\label{eq:charpy}
\log R + 1.312 \leq G(R) \leq \log R + 1.354,\end{equation}
where the upper bound is valid for $R\geq 120$.
This is true by (\ref{eq:malito}) for $R\geq 4\cdot 10^7$; we check 
(\ref{eq:charpy}) for $120\leq R\leq 4\cdot 10^7$ by a numerical 
computation.\footnote{Using D. Platt's implementation \cite{Platt}
of double-precision interval arithmetic based on Lambov's \cite{Lamb} ideas.}
Similarly, for $R\geq 200$,
\begin{equation}\label{eq:charpas}
\frac{\log R + 1.661}{2} \leq G_2(R) \leq \frac{\log R + 1.698}{2}
\end{equation}
by (\ref{eq:malito}) for $R\geq 1.6\cdot 10^8$,
and by a numerical computation for $200\leq R\leq 1.6\cdot 10^8$.

Write $\rho = (\log Q_0)/(\log Q)\leq 1$. 
We obtain immediately from (\ref{eq:charpy}) and 
(\ref{eq:charpas}) that
\begin{equation}\label{eq:rabatt}\begin{aligned}
\frac{G(Q_0)}{G(Q)} &\leq \frac{\log Q_0 + 1.354}{\log Q + 1.312} \\
\frac{G_2(Q_0)}{G_2(Q)} 
&\leq \frac{\log Q_0 + 1.698}{\log Q + 1.661}\end{aligned}\end{equation}
for $Q,Q_0\geq 200$.
What is hard is to approximate $G_q(Q_0)/G_q(Q)$ for $q$ large and $Q_0$ small.



Let us start by giving an easy bound, off from the truth by a factor of
about $e^\gamma$. (Specialists will recognize this as a factor that appears 
often in first attempts at estimates based on either large or small sieves.)
First, we need a simple explicit lemma.
\begin{lem}\label{lem:suspiro}
Let $m\geq 1$, $q\geq 1$. Then
\begin{equation}\label{eq:ostor}
\prod_{p|q \vee p\leq m} \frac{p}{p-1} \leq e^{\gamma} (\log (m+\log q) + 0.65771).\end{equation}
\end{lem}
\begin{proof}
Let $\mathscr{P}=\prod_{p\leq m \vee p|q} p$. Then, by
\cite[(5.1)]{MR0457373},
\[\mathscr{P}\leq q \prod_{p\leq m} p =
q e^{\sum_{p\leq m} \log p} \leq 
q e^{(1+\epsilon_0) m},\]
where $\epsilon_0 = 0.001102$. Now, by \cite[(3.42)]{MR0137689},
\[ \frac{n}{\phi(n)} \leq e^\gamma \log \log n + \frac{2.50637}{\log \log n}
\leq e^\gamma \log \log x + \frac{2.50637}{\log \log x}\]
for all $x\geq n\geq 27$ (since, given $a,b>0$, the function
 $t\mapsto a+b/t$ is increasing on $t$ 
for $t\geq \sqrt{b/a}$). Hence, if $q e^m \geq 27$,
\[\begin{aligned}
\frac{\mathscr{P}}{\phi(\mathscr{P})} &\leq e^\gamma
\log ((1+\epsilon_0) m + \log q) + \frac{2.50637}{\log (m+\log q)}\\
&\leq e^{\gamma} \left(\log (m+\log q) + \epsilon_0 +
\frac{2.50637/e^{\gamma}}{\log (m+\log q)}\right).
\end{aligned}\]
Thus (\ref{eq:ostor}) holds when $m+\log q \geq 8.53$, since then
$\epsilon_0 + (2.50637/e^{\gamma})/\log (m+\log q) \leq 0.65771$.
We verify all choices of $m,q\geq 1$ with $m+\log q\leq 8.53$ 
computationally; the worst case is that of $m=1$, $q=6$, which give the value
$0.65771$ in (\ref{eq:ostor}).
\end{proof}

Here is the promised easy bound.
\begin{lem}\label{lem:trivo}
Let $Q_0\geq 1$, $Q\geq 182 Q_0$. Let $q\leq Q_0$, $s\leq Q_0/q$, $q$ an integer.
Then \[\frac{G_q(Q_0/sq)}{G_q(Q/sq)} \leq
\frac{e^\gamma \log \left(\frac{Q_0}{sq} + \log q\right) + 1.172}
{\log \frac{Q}{Q_0} + 1.312} \leq
\frac{e^\gamma \log Q_0 + 1.172}
{\log \frac{Q}{Q_0} + 1.312} .\]
\end{lem}
\begin{proof}
Let $\mathscr{P}=\prod_{p\leq Q_0/sq \vee p|q} p$. Then
\[ G_q(Q_0/sq) G_{\mathscr{P}}(Q/Q_0) \leq G_q(Q/sq)\]
and so
\begin{equation}\label{eq:bete}
\frac{G_q(Q_0/sq)}{G_q(Q/sq)} \leq \frac{1}{G_{\mathscr{P}}(Q/Q_0)}.
\end{equation}

Now the lower bound in (\ref{eq:hosmo}) gives us that,
for $d=\mathscr{P}$, $R=Q/Q_0$, 
\[
G_{\mathscr{P}}(Q/Q_0) \geq \frac{\phi(\mathscr{P})}{\mathscr{P}}
G(Q/Q_0) .\]
By Lem.~\ref{lem:suspiro},
\[\frac{\mathscr{P}}{\phi(\mathscr{P})} \leq e^{\gamma} \left(\log \left(
\frac{Q_0}{s q} + \log q\right) + 0.658\right).\]
Hence, using (\ref{eq:charpy}), we get that
\begin{equation}\label{eq:mogan}
\begin{aligned}\frac{G_q(Q_0/sq)}{G_q(Q/sq)} &\leq 
\frac{\mathscr{P}/\phi(\mathscr{P})}{G(Q/Q_0)} \leq
\frac{e^\gamma \log \left(\frac{Q_0}{sq} + \log q\right) + 1.172}
{\log \frac{Q}{Q_0} + 1.312},
\end{aligned}\end{equation}
since $Q/Q_0\geq 184$. Since
\[\left(\frac{Q_0}{sq} + \log q\right)' = - \frac{Q_0}{sq^2} + \frac{1}{q}
= \frac{1}{q} \left(1 - \frac{Q_0}{sq}\right) \leq 0,\]
the rightmost expression of (\ref{eq:mogan}) is maximal for $q=1$.
\end{proof}

Lemma \ref{lem:trivo} will play a crucial role in reducing 
to a finite computation the problem
of bounding $G_q(Q_0/sq)/G_q(Q/sq)$.
As we will now
see, we can use Lemma \ref{lem:trivo} to 
obtain a bound that is useful when $sq$ is large compared to $Q_0$ --
precisely the case in which asymptotic estimates such as (\ref{eq:malito})
are relatively weak.
\begin{lem}\label{lem:paniz}
Let $Q_0\geq 1$, $Q\geq 200 Q_0$. Let $q\leq Q_0$, $s\leq Q_0/q$.
Let $\rho = (\log Q_0)/\log Q\leq 2/3$. 
Then, for any $\sigma\geq 1.312 \rho$, 
\begin{equation}\label{eq:martinu}
\frac{G_q(Q_0/sq)}{G_q(Q/sq)} \leq \frac{\log Q_0 + \sigma}{
\log Q + 1.312}\end{equation}
holds provided that
\[\frac{Q_0}{sq} \leq c(\sigma) \cdot Q_0^{(1-\rho) e^{-\gamma}} - \log q,\]
where $c(\sigma) = \exp(\exp(-\gamma)\cdot 
(\sigma - \sigma^2/5.248 - 1.172))$.
\end{lem}
\begin{proof}
By Lemma \ref{lem:trivo}, we see that
 (\ref{eq:martinu}) will hold provided that
\begin{equation}\label{eq:murey}\begin{aligned}
e^\gamma &\log \left(\frac{Q_0}{sq} + \log q\right) + 1.172
\leq
\frac{\log \frac{Q}{Q_0} + 1.312}{\log Q + 1.312} \cdot (\log Q_0 + \sigma).
\end{aligned}\end{equation}
The expression on the right of (\ref{eq:murey}) equals
\[\begin{aligned}
\log Q_0 + \sigma &- \frac{(\log Q_0 + \sigma) \log Q_0}{
\log Q + 1.312}\\
&=
(1-\rho) (\log Q_0 + \sigma) + \frac{1.312 \rho (\log Q_0 + \sigma)}{
\log Q + 1.312}\\ &\geq (1-\rho) (\log Q_0 + \sigma) +
1.312\rho^2\end{aligned}\]
and so
(\ref{eq:murey}) will
 hold provided that
\[\begin{aligned}
e^\gamma &\log \left(\frac{Q_0}{sq} + \log q\right) +1.172\leq
 (1-\rho) (\log Q_0) + (1-\rho) \sigma +
1.312\rho^2 .\end{aligned}\]
Taking derivatives, we see that
\[\begin{aligned}
(1-\rho) \sigma + 1.312 \rho^2 - 1.172 &\geq 
\left(1 - \frac{\sigma}{2.624}\right) \sigma + 1.312
\left(\frac{\sigma}{2.624}\right)^2 - 1.172\\
&= \sigma - \frac{\sigma^2}{4\cdot 1.312} - 1.172 .
\end{aligned}\]
Hence it is enough that
\[\frac{Q_0}{sq} + \log q \leq e^{e^{-\gamma} 
\left((1-\rho) \log Q_0 + \sigma - \frac{\sigma^2}{4\cdot 1.312} - 1.172\right)} =
c(\sigma)\cdot Q_0^{(1-\rho) e^{-\gamma}},\]
where $c(\sigma) = \exp(\exp(-\gamma)\cdot 
(\sigma - \sigma^2/5.248 - 1.172))$.
\end{proof}

\begin{prop}\label{prop:espagn}
Let $Q\geq 20000 Q_0$, $Q_0\geq Q_{0,\min}$, where $Q_{0,\min} = 10^5$. 
Let $\rho = (\log Q_0)/\log Q$. Assume $\rho\leq 0.6$. Then, for every 
$1\leq q\leq Q_0$ and every $s\in \lbrack 1, Q_0/q\rbrack$,
\begin{equation}\label{eq:martinos}
\frac{G_q(Q_0/sq)}{G_q(Q/sq)} \leq \frac{\log Q_0 + c_+}{
\log Q + c_E},\end{equation}
where $c_E$ is as in (\ref{eq:garno}) and $c_+ = 1.36$.
\end{prop}

An ideal result would have $c_+$ instead of $c_E$, but this is not actually
possible: error terms do exist, even if they are in reality smaller than
the bound given in (\ref{eq:malito}); this means that a bound such as (\ref{eq:martinos}) with $c_+$ instead of $c_E$ would be false for $q=1$, $s=1$.

There is nothing special about the assumptions $Q\geq 20000 Q_0$, 
$Q_0\geq 10^5$, $(\log Q_0)/(\log Q)\leq 0.6$: they can all be relaxed
at the cost of an increase in $c_+$.

\begin{proof}
Define $\err_{q,R}$ so that
\begin{equation}\label{eq:mero}
G_q(R) = \frac{\phi(q)}{q} \left(\log R + c_E + \sum_{p|q} \frac{\log p}{p} \right)+ \err_{q,R}.\end{equation}
Then (\ref{eq:martinos}) will hold if
\begin{equation}\label{eq:elsyn}\begin{aligned}
\log \frac{Q_0}{s q} &+ c_E + \sum_{p|q} \frac{\log p}{p} +
\frac{q}{\phi(q)} \err_{q,\frac{Q_0}{sq}}\\ &\leq
\left(\log \frac{Q}{s q} + c_E + \sum_{p|q} \frac{\log p}{p} +
\frac{q}{\phi(q)} \err_{q,\frac{Q}{sq}}\right) \frac{\log Q_0 + c_+}{\log Q + c_E}.\end{aligned}\end{equation}
This, in turn, happens if 
\[\begin{aligned}
\left(\log sq - \sum_{p|q} \frac{\log p}{p}\right)
&\left(1 - \frac{\log Q_0 + c_+}{\log Q + c_E}\right)
+ c_+ - c_E
\\
\geq &\frac{q}{\phi(q)} \left(\err_{q,\frac{Q_0}{sq}} 
- \frac{\log Q_0 + c_+}{\log Q + c_E} \err_{q,\frac{Q}{sq}}\right).
\end{aligned}\]
Define
\[\omega(\rho) = \frac{\log Q_{0,\min} + c_+}{\frac{1}{\rho}
\log Q_{0,\min} + c_E} = \rho + \frac{c_+ - \rho c_E}{
\frac{1}{\rho} \log Q_{0,\min} + c_E}.\]
Then $\rho \leq (\log Q_0+c_+)/(\log Q + c_E) \leq \omega(\rho)$
(because $c_+ \geq \rho c_E$).
We conclude that (\ref{eq:elsyn}) (and hence (\ref{eq:martinos}))
holds provided that
\begin{equation}\label{eq:karka}
(1-\omega(\rho)) \left(\log s q - \sum_{p|q} \frac{\log p}{p}\right)
+ c_\Delta \geq \frac{q}{\phi(q)}\left(\err_{q,\frac{Q_0}{sq}} 
+ \omega(\rho) \max\left(0, - \err_{q,\frac{Q}{sq}}\right)\right),\end{equation}
where $c_\Delta = c_+ - c_E$.
Note that $1-\omega(\rho)>0$.

First, let us give some easy bounds on the error terms; these bounds
will yield upper bounds for $s$. By (\ref{eq:cante}) and (\ref{eq:hosmo}),
\[
\err_{q,R} \leq \frac{\phi(q)}{q} \left(\log q -
 \sum_{p|q} \frac{\log p}{p} + (1.4709 - c_E)\right)\]
for $R\geq 1$; by (\ref{eq:charpy}) and (\ref{eq:hosmo}),
\[\err_{q,R} \geq - \frac{\phi(q)}{q} \left(
\sum_{p|q} \frac{\log p}{p} + (c_E - 1.312)\right)\]
for $R\geq 182$. Therefore, the right side of (\ref{eq:karka})
is at most
\[\log q -
 (1-\omega(\rho)) \sum_{p|q} \frac{\log p}{p} + 
((1.4709 - c_E) + \omega(\rho) (c_E - 1.312)),\]
and so (\ref{eq:karka}) holds provided that
\begin{equation}\label{eq:miasmar}
(1-\omega(\rho)) \log s q\geq \log q +
(1.4709 - c_E) + \omega(\rho) (c_E - 1.312) - c_\Delta.\end{equation}
We will thus be able to assume from now on that (\ref{eq:miasmar})
does not hold, or, what is the same, that
\begin{equation}\label{eq:koklo}
sq < \left(c_{\rho,2} q\right)^{\frac{1}{1-\omega(\rho)}}\end{equation}
holds, where 
$c_{\rho,2} = \exp((1.4709 - c_E) + \omega(\rho) (c_E - 1.312) - c_\Delta)$.

What values of $R = Q_0/sq$ must we consider for $q$ given? First,
by (\ref{eq:koklo}), 
we can assume $R>Q_{0,\min}/(c_{\rho,2} q)^{1/(1-\omega(\rho))}$.
We can also assume
\begin{equation}\label{eq:joho}
R >
c(c_+) \cdot \max(R q,Q_{0,\min})^{(1-\rho) e^{-\gamma}} - \log q 
\end{equation}
for $c(c_+)$ is as in Lemma \ref{lem:paniz}, since
all smaller $R$ are covered by that Lemma.
Clearly, (\ref{eq:joho}) implies that
\[R^{1-\tau} > c(c_+) \cdot q^{\tau} -
\frac{\log q}{R^{\tau}} > 
c(c_+) q^{\tau} -\log q,
\]
where $\tau = (1-\rho) e^{-\gamma}$,
and also that $R> c(c_+) Q_{0,\min}^{(1-\rho) e^{-\gamma}} - \log q$.
Iterating, we obtain that we can assume that
$R>\varpi(q)$, where
\begin{equation}\label{eq:armor}
\varpi(q) = \max\left(\varpi_0(q),
c(c_+) Q_{0,\min}^{\tau} - \log q,
\frac{Q_{0,\min}}{(c_{\rho,2} q)^{\frac{1}{1-\omega(\rho)}}}\right)\end{equation}
and
\[\varpi_0(q) = 
\begin{cases}
\left(c(c_+) q^{\tau} - \frac{\log q}{
(c(c_+) q^{\tau} -\log q)^{\frac{\tau}{1- \tau}}}
\right)^{\frac{1}{1-\tau}} &\text{if 
$c(c_+) q^{\tau} > \log q + 1$},\\
0 
&\text{otherwise.}\end{cases}\]

Looking at (\ref{eq:karka}), we see that
it will be enough to show that, for all $R$ satisfying $R>\varpi(q)$, we have
\begin{equation}\label{eq:luce}
\err_{q,R} + \omega(\rho) \max\left(0,-\err_{q,tR}\right)
 \leq \frac{\phi(q)}{q} \kappa(q)
\end{equation} for all $t\geq 20000$, where
\[\kappa(q) = 
(1 - \omega(\rho)) \left(\log q - \sum_{p|q} \frac{\log p}{p}\right) + 
c_\Delta.\]

Ramar\'e's bound (\ref{eq:malito}) implies that
\begin{equation}\label{eq:agammen}
|\err_{q,R}| \leq 7.284 R^{-1/3} f_1(q),\end{equation}
with $f_1(q)$ as in (\ref{eq:assur}), and so
\[\err_{q,R} + \omega(\rho) \max\left(0,-\err_{q,tR}\right) \leq
(1 + \beta_\rho) \cdot 7.284 R^{-1/3} f_1(q),\]
where $\beta_\rho = \omega(\rho)/20000^{1/3}$.
This is enough when 
\begin{equation}\label{eq:sosor}
R\geq \lambda(q) = 
 \left(\frac{q}{\phi(q)} \frac{7.284 (1+\beta_\rho) f_1(q)}{\kappa(q)}\right)^3.
\end{equation}

It remains to do two things. First, we have to compute how large $q$
has to be for $\varpi(q)$ to be guaranteed to be greater than 
$\lambda(q)$. (For such $q$, there is no checking to be done.) Then,
we check the inequality (\ref{eq:luce}) for all smaller $q$, letting
$R$ range through the integers in $\lbrack \varpi(q),\lambda(q)\rbrack$.
We bound $\err_{q,t R}$ using (\ref{eq:agammen}), but we compute
$\err_{q,R}$ directly.

{\em How large must $q$ be for $\varpi(q)>\lambda(q)$ to hold?} 
We claim that $\varpi(q)>\lambda(q)$ whenever
 $q\geq 2.2\cdot 10^{10}$.
Let us show this. 

It is easy to see that $(p/(p-1)) \cdot f_1(p)$ and
$p\to (\log p)/p$ are decreasing functions of $p$ for
$p\geq 3$; moreover, for both functions, the value at 
$p\geq 7$ is smaller than for $p=2$. 
Hence, we have that, for $q < \prod_{p\leq p_0} p$, $p_0$ a prime,
\begin{equation}\label{eq:modo}
\kappa(q) \geq (1-\omega(\rho)) \left(\log q - 
\sum_{p<p_0} \frac{\log p}{p}\right) + c_\Delta
\end{equation}
and
\begin{equation}\label{eq:hipo}
\lambda(q) \leq \left(\prod_{p<p_0} \frac{p}{p-1} \cdot
\frac{7.284 (1+\beta_\rho) \prod_{p<p_0} f_1(p)}{
(1-\omega(\rho)) \left(\log q - 
\sum_{p<p_0} \frac{\log p}{p}\right) + c_\Delta}\right)^3.\end{equation}
If we also assume that $2\cdot 3\cdot 5\cdot 7 \nmid q$, we obtain
\begin{equation}\label{eq:modowo}
\kappa(q) \geq (1-\omega(\rho)) \left(\log q - 
\mathop{\sum_{p<p_0}}_{p\ne 7} \frac{\log p}{p}\right) + c_\Delta
\end{equation}
and
\begin{equation}\label{eq:hipowo}
\lambda(q) \leq \left(\mathop{\prod_{p<p_0}}_{p\ne 7} \frac{p}{p-1} \cdot
\frac{7.284 (1+\beta_\rho) \prod_{p<p_0, p \ne 7} f_1(p)}{
(1-\omega(\rho)) \left(\log q - 
\sum_{p<p_0, p\ne 7} \frac{\log p}{p}\right) + c_\Delta}\right)^3\end{equation}
for $q< \prod_{p\leq p_0}$.
(We are taking out $7$ because it is the ``least helpful'' prime to omit
among all primes from $2$ to $7$, again by the fact that
 $(p/(p-1)) \cdot f_1(p)$ and
$p\to (\log p)/p$ are decreasing functions for $p\geq 3$.)

We know how to give upper bounds for the expression on the right
of (\ref{eq:hipo}).
The task is in essence simple: we can base our bounds on the classic
explicit work in \cite{MR0137689}, except that we also have to optimize matters
so that they are close to tight for $p_1=29$, $p_1=31$ and other low $p_1$.

By \cite[(3.30)]{MR0137689} and a numerical computation for $29\leq p_1\leq 43$,
\[\prod_{p\leq p_1} \frac{p}{p-1} <  
1.90516 \log p_1\]
for $p_1\geq 29$.
Since $\omega(\rho)$ is increasing on $\rho$ and we are assuming
$\rho\leq 0.6$, $Q_{0,\min} = 100000$,
\[\omega(\rho)\leq 0.627312,\;\;\;\;\;
\beta_{\rho} \leq 0.023111.\]
For $x>a$, where $a>1$ is any constant, we obviously have
\[\sum_{a< p\leq x} \log \left(1 + p^{-2/3}\right) \leq
\sum_{a< p\leq x} (\log p) \frac{p^{-2/3}}{\log a}.
\]
by Abel summation (see (\ref{eq:jokors})) and the estimates
\cite[(3.32)]{MR0137689} for $\theta(x) = \sum_{p\leq x} \log p$,
\[\begin{aligned}
\sum_{a<p\leq x}& (\log p) p^{-2/3} = (\theta(x) - \theta(a)) x^{- \frac{2}{3}} -
\int_a^x (\theta(u) -\theta(a)) \left(-\frac{2}{3} u^{-\frac{5}{3}}\right) du \\
&\leq (1.01624 x- \theta(a)) x^{-\frac{2}{3}} +
\frac{2}{3} \int_a^x \left(1.01624 u - \theta(a)\right) u^{-\frac{5}{3}} du\\
&= (1.01624 x- \theta(a)) x^{-\frac{2}{3}} +
2\cdot 1.01624 (x^{1/3}-a^{1/3}) + \theta(a) (x^{-2/3} - a^{-2/3})\\
&= 3\cdot 1.01624 \cdot x^{1/3}
- (2.03248 a^{1/3} + \theta(a) a^{-2/3}).
\end{aligned}\]
We conclude that
$\sum_{10^4<p\leq x} \log (1 + p^{-2/3}) \leq 0.33102 x^{1/3} - 7.06909$
for $x>10^4$. Since $\sum_{p\leq 10^4} \log p \leq 10.09062$, this means that
\[\sum_{p\leq x} \log (1 + p^{-2/3}) \leq \left(0.33102 +
\frac{10.09062 - 7.06909}{10^{4/3}}\right) x^{1/3} 
\leq 0.47126 x^{1/3}\] for $x>10^4$;
a direct computation for all $x$ prime between $29$ and $10^4$ then 
confirms that \[\sum_{p\leq x} \log (1 + p^{-2/3}) \leq 
0.74914 x^{1/3}\]
for all $x\geq 29$. Thus,
\[\prod_{p\leq x} f_1(p) \leq \frac{e^{\sum_{p\leq x} \log (1+ p^{-2/3})}}{
\prod_{p\leq 29} \left(1 + \frac{p^{1/3} + p^{2/3}}{p (p-1)}\right)}
\leq \frac{e^{0.74914 x^{1/3}}}{6.62365} 
\]
for $x\geq 29$.
Finally, by \cite[(3.24)]{MR0137689},
$\sum_{p\leq p_1} \frac{\log p}{p} < \log p_1$. 

We conclude that, for $q<\prod_{p\leq p_0} p_0$, $p_0$ a prime, and 
$p_1$ the prime immediately preceding $p_0$,
\begin{equation}\label{eq:victo}\begin{aligned}
\lambda(q) &\leq \left(
1.90516 \log p_1 \cdot \frac{7.45235 \cdot \left(
 \frac{e^{0.74914 p_1^{1/3}}}{6.62365}\right)}{
0.37268 (\log q - \log p_1) + 0.02741}
\right)^3\\
&\leq \frac{190.272 (\log p_1)^3 e^{2.24742 p_1^{1/3}}}{
(\log q - \log p_1 + 0.07354)^3}.
\end{aligned}\end{equation}

It is clear from (\ref{eq:armor})  
that $\varpi(q)$ is increasing 
as soon as \[q\geq \max(Q_{0,\min},Q_{0,\min}^{1-\omega(\rho)}/c_{\rho,2})\] 
and $c(c_+) q^\tau > \log q + 1$, since then $\varpi_0(q)$ is increasing
and $\varpi(q) = \varpi_0(q)$. Here it is useful to recall that
$c_{\rho,2}\geq \exp(1.4709-c_+)$, and to note that
$c(c_+) q^{\tau} - (\log q+1)$ is increasing for 
$q\geq 1/(\tau\cdot  c(c_+))^{1/\tau}$; we see also that 
$1/(\tau \cdot
 c(c_+))^{1/\tau} \leq 1/((1-0.6) e^{-\gamma} c(c_+))^{1/((1-0.6) e^{-\gamma})}$
for $\rho\leq 0.6$. A quick computation for our value of $c_+$ makes
us conclude that $q>1.12 Q_{0,\min}=112000$ is a sufficient condition for
$\varpi(q)$ to be equal to $\varpi_0(q)$ and for $\varpi_0(q)$ to be increasing.

Since (\ref{eq:victo}) is decreasing on $q$ for $p_1$ fixed, and
$\varpi_0(q)$ is decreasing on $\rho$ and increasing on $q$, we set 
$\rho = 0.6$ and check that then
\[\varpi_0\left(2.2\cdot 10^{10}\right) \geq 846.765,\]
whereas, by (\ref{eq:victo}),
\[\lambda(2.2\cdot 10^{10}) \leq 838.227 < 846.765;\]
this is enough to ensure that $\lambda(q)<\varpi_0(q)$ for 
$2.2\cdot 10^{10} \leq q < \prod_{p\leq 31} p$.

Let us now give some rough bounds that will be enough to cover the
case $q\geq \prod_{p\leq 31} p$. First, as we already discussed,
$\varpi(q)= \varpi_0(q)$ and, since $c(c_+) q^\tau > \log q + 1$,
\begin{equation}\label{eq:drolo}
\varpi_0(q) \geq (c(c_+) q^\tau - \log q)^{\frac{1}{1-\tau}} \geq
(0.911 q^{0.224} - \log q)^{1.289} \geq q^{0.2797}\end{equation}
by $q\geq \prod_{p\leq 31} p$. We are in the range $\prod_{p\leq p_1} p \leq
q\leq \prod_{p\leq p_0} p$, where $p_1<p_0$ are two consecutive primes with
$p_1\geq 31$. By \cite[(3.16)]{MR0137689} and a computation for 
$31\leq q< 200$, we know that $\log q \geq \prod_{p\leq p_1} \log p \geq
0.8009 p_1$. By (\ref{eq:victo}) and (\ref{eq:drolo}), 
it follows that we just have to show
that
\[e^{0.224 t} > \frac{190.272 (\log t)^3 e^{2.24742 t^{1/3}}}{
(0.8009 t - \log t + 0.07354)^3}\]
for $t\geq 31$. Now, $t\geq 31$ implies 
$0.8009 t - \log t + 0.07354 \geq 0.6924 t$, and so, taking logarithms
we see that we just have to verify
\begin{equation}\label{eq:mutuso}
0.224 t -  2.24742 t^{1/3} > 3 \log \log t - 3 \log t + 6.3513\end{equation}
for $t\geq 31$, and, since the left side is increasing and the right
side is decreasing for $t\geq 31$, this is trivial to check.

We conclude that $\varpi(q)>\lambda(q)$ whenever $q\geq 2.2\cdot 10^{10}$.

It remains to see how we can relax this assumption if we assume
that $2\cdot 3\cdot 5\cdot 7\nmid q$.  We repeat the same
analysis as before, using (\ref{eq:modowo}) and (\ref{eq:hipowo}) instead
of (\ref{eq:modo}) and (\ref{eq:hipo}). For $p_1\geq 29$,
\[\mathop{\prod_{p\leq p_1}}_{p\ne 7}
 \frac{p}{p-1} < 1.633 \log p_1,\;\;\;\;
\mathop{\prod_{p\leq p_1}}_{p\ne 7}
f_1(p) \leq \frac{e^{0.74914 x^{1/3} - \log(1+7^{-2/3})}}{5.8478} \leq
\frac{e^{0.74914 x^{1/3}}}{7.44586}\]
and $\sum_{p\leq p_1: p\ne 7} (\log p)/p < \log p_1 - (\log 7)/7$. 
So, for $q< \prod_{p\leq p_0: p\ne 7} p$, and $p_1\geq 29$ 
the prime immediately preceding $p_0$,
\[\begin{aligned}
\lambda(q) 
&\leq \left(
1.633 \log p_1 \cdot \frac{7.45235 \cdot \left(
 \frac{e^{0.74914 p_1^{1/3}}}{7.44586}\right)}{
0.37268 \left(\log q - \log p_1 + \frac{\log 7}{7}\right) + 0.02741}
\right)^3\\
&\leq \frac{84.351 (\log p_1)^3 e^{2.24742 p_1^{1/3}}}{
(\log q - \log p_1 + 0.35152)^3}.
\end{aligned}\] Thus we obtain, just like before, that
\[\varpi_0(3.3\cdot 10^9)\geq 477.465,\;\;\;\;\;\;\;
\lambda(3.3\cdot 10^9) \leq 475.513 < 477.465.\]
We also check that $\varpi_0(q_0)\geq 916.322$ is greater than
$\lambda(q_0)\leq 429.731$ for $q_0= \prod_{p\leq 31: p\ne 7} p$.
The analysis for $q\geq \prod_{p\leq 37: p\ne 7} p$ is also just like before:
since $\log q \geq 0.8009 p_1 - \log 7$, we have to show that
\[\frac{e^{0.224 t}}{7} > \frac{84.351 (\log t)^3 e^{2.24742 t^{1/3}}}{(0.8009 t - \log t + 0.07354)^3}\]
for $t\geq 37$, and that, in turn, follows from
\[0.224 t - 2.24742 t^{1/3} > 3 \log \log t - 3 \log t + 6.74849,\]
which we check for $t\geq 37$ just as we checked (\ref{eq:mutuso}).

We conclude that 
$\varpi(q)>\lambda(q)$ if $q\geq 3.3\cdot 10^{9}$ and $210\nmid q$.

{\em Computation.}
Now, for $q < 3.3 \cdot 10^9$ (and also for $3.3\cdot 10^9 \leq q <
2.2\cdot 10^{10}$, $210|q$),
 we need to check that
the maximum $m_{q,R,1}$ of $\err_{q,R}$ over all $\varpi(q)\leq
R<\lambda(q)$ satisfies (\ref{eq:luce}).
Note that there is a term $\err_{q,tR}$ in (\ref{eq:luce}); we bound
it using (\ref{eq:agammen}).

Since $\log R$ is increasing on $R$
and $G_q(R)$ depends only on $\lfloor R\rfloor$, we can tell from 
(\ref{eq:mero}) that, since we are taking the maximum of $\err_{q,R}$, 
it is enough to check integer values of $R$. We check all integers $R$
in $\lbrack \varpi(q),\lambda(q))$ for all $q<3.3\cdot 10^9$
(and all  $3.3\cdot 10^9 \leq q <
2.2\cdot 10^{10}$, $210|q$) by an explicit computation.\footnote{This is
by far the heaviest computation in the present paper, though it is still
rather minor (about two weeks of computing on a single core of a fairly new
(2010) desktop computer carrying out other tasks as well; this is next
to nothing compared to the computations in \cite{Plattfresh}, or even
those in \cite{HelPlat}). For the applications in the present paper,
we could have assumed $\rho\leq 8/15$, and that would have
reduced computation time drastically; the lighter assumption $\rho\leq 0.6$
was made with views to general applicability in the future.
As elsewhere in this section, 
numerical computations were carried out by the author 
in C; all floating-point operations used 
D. Platt's interval arithmetic package.}
\end{proof}

Finally, we have the trivial bound
\begin{equation}\label{eq:triwia}
\frac{G_q(Q_0/sq)}{G_q(Q/sq)} \leq 1,\end{equation}
which we shall use for $Q_0$ close to $Q$.


\begin{cor}\label{cor:coeur}
Let $\{a_n\}_{n=1}^\infty$, $a_n\in \mathbb{C}$, be supported on the primes.
Assume that $\{a_n\}$ is in $\ell_1\cap \ell_2$ and that $a_n=0$ for $n\leq \sqrt{x}$.
Let $Q_0\geq 10^5$, $\delta_0\geq 1$ be such that $(20000 Q_0)^2 \leq 
x/2 \delta_0$; set
$Q = \sqrt{x/2 \delta_0}$. 

Let $S(\alpha) = \sum_n a_n e(\alpha n)$ for $\alpha \in \mathbb{R}/\mathbb{Z}$.
Let $\mathfrak{M}$ as in (\ref{eq:jokor}). Then, if
$Q_0\leq Q^{0.6}$,
\[\int_{\mathfrak{M}} \left|S(\alpha)\right|^2 d\alpha \leq
\frac{\log Q_0 + c_+}{\log Q + c_E} \sum_n |a_n|^2,\]
where $c_+ = 1.36$ and $c_E = \gamma + \sum_{p\geq 2} (\log p)/(p(p-1)) = 
1.3325822\dotsc$.

Let $\mathfrak{M}_{\delta_0,Q_0}$ as in (\ref{eq:majdef}). Then, if
$(2 Q_0)\leq (2 Q)^{0.6}$,
\[\int_{\mathfrak{M}_{\delta_0,Q_0}} \left|S(\alpha)\right|^2 d\alpha \leq
\frac{\log 2 Q_0 + c_+}{\log 2 Q + c_E} \sum_n |a_n|^2.\]
\end{cor}
Here, of course, $\int_{\mathbb{R}/\mathbb{Z}} \left|S(\alpha)\right|^2 d\alpha
= \sum_n |a_n|^2$ (Plancherel). If $Q_0>Q^{0.6}$, we will use the trivial bound
\begin{equation}\label{eq:trivo}
\int_{\mathfrak{M}_{\delta_0,r}} \left|S(\alpha)\right|^2 d\alpha \leq
\int_{\mathbb{R}/\mathbb{Z}} \left|S(\alpha)\right|^2 d\alpha =
 \sum_n |a_n|^2.
\end{equation} 
\begin{proof}
Immediate from Prop.~\ref{prop:ramar}, Prop.~\ref{prop:bellen} and
Prop.~\ref{prop:espagn}.
\end{proof}
Obviously, one can also give a statement derived from
Prop.~\ref{prop:ramar}; the resulting bound is
\[\int_{\mathfrak{M}} |S(\alpha)|^2 d\alpha \leq 
\frac{\log Q_0 + c_+}{\log Q + c_E} \sum_n |a_n|^2,\]
where $\mathfrak{M}$ is as in (\ref{eq:jokor}).

We also record the large-sieve form of the result.
\begin{cor}\label{cor:carnatio}
Let $N\geq 1$.
Let $\{a_n\}_{n=1}^\infty$, $a_n\in \mathbb{C}$, be supported on the integers
$n\leq N$.
Let $Q_0\geq 10^5$, $Q\geq 20000 Q_0$. Assume that $a_n=0$ for every $n$ for
which there is a $p\leq Q$ dividing $n$.

Let $S(\alpha) = \sum_n a_n e(\alpha n)$ for $\alpha \in \mathbb{R}/\mathbb{Z}$.
Then, if $Q_0\leq Q^{0.6}$,
\[\sum_{q\leq Q_0} \mathop{\sum_{a \mod q}}_{(a,q)=1} 
       \left|S(a/q)\right|^2 d\alpha \leq
\frac{\log Q_0 + c_+}{\log Q + c_E} \cdot (N + Q^2) \sum_n |a_n|^2
,\]
where $c_+ = 1.36$ and $c_E = \gamma + \sum_{p\geq 2} (\log p)/(p(p-1)) = 
1.3325822\dotsc$.
\end{cor}
\begin{proof}
Proceed as Ramar\'e does in the proof of \cite[Thm. 5.2]{MR2493924},
with $\mathscr{K}_q = \{a\in \mathbb{Z}/q\mathbb{Z}: (a,q)=1\}$ and $u_n = a_n$); in particular,
apply \cite[Thm. 2.1]{MR2493924}.
The proof of \cite[Thm. 5.2]{MR2493924} shows that
\[\sum_{q\leq Q_0} \mathop{\sum_{a \mod q}}_{(a,q)=1} 
       \left|S(a/q)\right|^2 d\alpha \leq
\max_{q\leq Q_0} \frac{G_q(Q_0)}{G_q(Q)} \cdot
\sum_{q\leq Q_0} \mathop{\sum_{a \mod q}}_{(a,q)=1} 
       \left|S(a/q)\right|^2 d\alpha.\]
Now, instead of using the easy inequality $G_q(Q_0)/G_q(Q)\leq G_1(Q_0)/G_1(Q/Q_0)$,
use Prop.~\ref{prop:espagn}.
\end{proof}

\begin{center}
* * *
\end{center}

It would seem desirable to prove a result such as Prop.~\ref{prop:espagn}
(or Cor.~\ref{cor:coeur}, or Cor.~\ref{cor:carnatio}).
without computations and with conditions that are as weak as possible.
Since, as we said, we cannot make $c_+$ equal to $c_E$, and since 
$c_+$ does have to increase when the conditions are weakened (as is shown by
computations; this is not an artifact of our method of proof)
 the right goal might be to show that 
the maximum of $G_q(Q_0/sq)/G_q(Q/sq)$ is reached when $s=q=1$.

However, this is also untrue without conditions. For instance, for
$Q_0=2$ and $Q$ large, the value of $G_q(Q_0/q)/G_q(Q/q)$ at $q=2$ is larger
than at $q=1$: by (\ref{eq:malito}),
\[\frac{G_2\left(\frac{Q_0}{2}\right)}{G_2\left(\frac{Q}{2}\right)} \sim
\frac{1}{\frac{1}{2} \left(\log \frac{Q}{2} + c_E + \frac{\log 2}{2}\right)}
= \frac{2}{\log Q + c_E - \frac{\log 2}{2}} > \frac{2}{\log Q + c_E}
\sim \frac{G(Q_0)}{G(Q)}.\]
The same holds for $Q_0=3$, $Q_0=5$ or $Q_0 = 30$, say, since in all these
cases $Q_0/\phi(Q_0) > \log Q_0$. Thus, it is clear that, at the very least,
a lower bound on $Q_0$ is needed as a condition. This also dims the hopes
somewhat for a combinatorial proof of $G_q(Q_0/q) G(Q) \leq G_q(Q/q)
G(Q_0)$; at any rate, while such a proof would be welcome, it could not be
extremely straightforward, since there are terms in $G_q(Q_0/q) G(Q)$
that do not appear in $G_q(Q/q) G(Q_0)$.


\section{The integral over the minor arcs}

The time has come to bound the part of our triple-product integral
(\ref{eq:osto}) that comes from the minor arcs $\mathfrak{m}\subset
\mathbb{R}/\mathbb{Z}$. We have an $\ell_\infty$ estimate (from
Prop.~\ref{prop:gorsh}, based on \cite{Helf}) and an $\ell_2$ estimate
(from \S \ref{sec:intri}). Now we must put them together.

There are two ways in which we must be careful. A trivial bound
of the form $\ell_3^3 = \int |S(\alpha)|^3 d\alpha \leq \ell_2^2 \cdot \ell_\infty$
would introduce a fatal factor of $\log x$ coming from $\ell_2$. We avoid
this by using the fact that we have $\ell_2$ estimates over 
$\mathfrak{M}_{\delta_0,Q_0}$ for varying $Q_0$. 

We must also remember
to substract the major-arc contribution from our estimate for
 $\mathfrak{M}_{\delta_0,Q_0}$; this is why we were careful to give
a lower bound in Lem.~\ref{lem:drujal}, as opposed to just the upper
bound (\ref{eq:mardi}).

\subsection{Putting together $\ell_2$ bounds over arcs and
$\ell_\infty$ bounds}

Let us start with a simple lemma -- essentially a way to obtain upper bounds
by means of summation by parts.
\begin{lem}\label{lem:jardinbota}
Let $f,g:\{a,a+1,\dotsc,b\}\to \mathbb{R}_0^+$, where $a,b\in
\mathbb{Z}^+$. Assume that, for all $x\in \lbrack a,b\rbrack$,
\begin{equation}\label{eq:gorto}
\sum_{a\leq n\leq x} f(n) \leq F(x),
\end{equation}
where $F:\lbrack a,b\rbrack\to \mathbb{R}$ is continuous,
piecewise differentiable and non-decreasing. Then
\[
\sum_{n=a}^b f(n) \cdot g(n) \leq (\max_{n\geq a} g(n))\cdot F(a) 
+ \int_a^{b} (\max_{n\geq u} g(n)) \cdot F'(u) du .
\]
\end{lem}
\begin{proof}
Let $S(n) = \sum_{m=a}^n f(m)$. Then, by partial summation,
\begin{equation}\label{eq:marshti}
\sum_{n=a}^b f(n) \cdot g(n) \leq S(b) g(b) + \sum_{n=a}^{b-1} S(n) (g(n) -
g(n+1)) .
\end{equation}
Let $h(x) = \max_{x\leq n\leq b} g(n)$. Then $h$ is non-increasing. Hence
(\ref{eq:gorto}) and (\ref{eq:marshti}) imply that
 \[\begin{aligned}
\sum_{n=a}^b f(n) g(n) &\leq \sum_{n=a}^b f(n) h(n)\\
&\leq S(b) h(b) + \sum_{n=a}^{b-1} S(n) (h(n) - h(n+1)) \\
&\leq F(b) h(b) + \sum_{n=a}^{b-1} F(n) (h(n) - h(n+1)) .
\end{aligned}\]
In general, for $\alpha_n\in \mathbb{C}$,  $A(x)=\sum_{a\leq n\leq x} \alpha_n$
and $F$ continuous and piecewise differentiable on $\lbrack a,x\rbrack$,
\begin{equation}\label{eq:jokors}
\sum_{a\leq n\leq x} \alpha_n F(x) = A(x) F(x) - \int_a^x A(u) F'(u) du .
\;\;\;\;\;\;\;\text{({\em Abel summation})}\end{equation}
Applying this with $\alpha_n = h(n) - h(n+1)$ and $A(x) = \sum_{a\leq n\leq
x} \alpha_n  = h(a) - h(\lfloor x\rfloor + 1)$, we obtain
\[\begin{aligned}
\sum_{n=a}^{b-1} &F(n) (h(n)-h(n+1)) \\&= (h(a)- h(b)) F(b-1) - 
\int_a^{b-1} (h(a) - h(\lfloor u\rfloor +1)) F'(u) du\\
&= h(a) F(a) - h(b) F(b-1) + \int_a^{b-1} h(\lfloor u\rfloor+1) F'(u) du\\
&= h(a) F(a) - h(b) F(b-1) + \int_a^{b-1} h(u) F'(u) du\\
&= h(a) F(a) - h(b) F(b) + \int_a^{b} h(u) F'(u) du,\end{aligned}\]
since $h(\lfloor u\rfloor + 1) = h(u)$ for $u\notin \mathbb{Z}$.
Hence
\[\sum_{n=a}^b f(n) g(n) \leq h(a) F(a) + \int_a^{b} h(u) F'(u) du .\] 
\end{proof}

We will now see our main application of Lemma \ref{lem:jardinbota}.
We have to bound an integral of the form
$\int_{\mathfrak{M}_{\delta_0,r}} |S_1(\alpha)|^2 |S_2(\alpha)|  d\alpha$,
where $\mathfrak{M}_{\delta_0,r}$ is 
a union of arcs defined as in (\ref{eq:majdef}). Our inputs are (a) a bound
on integrals of the form $\int_{\mathfrak{M}_{\delta_0,r}} |S_1(\alpha)|^2 d\alpha$,
(b) a bound on $|S_2(\alpha)|$ for $\alpha\in 
(\mathbb{R}/\mathbb{Z})\setminus \mathfrak{M}_{\delta_0,r}$. The input of type
(a) is what we derived in \S \ref{subs:ramar} and \S \ref{subs:boquo}; the
input of type (b) is a minor-arcs bound, and as such is the main subject
of \cite{Helf}.  

\begin{prop}\label{prop:palan}
Let $S_1(\alpha) = \sum_n a_n e(\alpha n)$, $a_n \in \mathbb{C}$,
$\{a_n\}$ in $L^1$. 
Let $S_2:\mathbb{R}/\mathbb{Z}\to \mathbb{C}$ be continuous.
Define $\mathfrak{M}_{\delta_0,r}$ as in (\ref{eq:majdef}).

Let $r_0$ be a positive integer not greater than $r_1$. 
Let $H:\lbrack r_0,r_1\rbrack
 \to \mathbb{R}^+$ be a continuous, piecewise differentiable, non-decreasing
function such that
\begin{equation}\label{eq:qewer}
\frac{1}{\sum |a_n|^2} \int_{\mathfrak{M}_{\delta_0,r+1}} |S_1(\alpha)|^2 d\alpha \leq 
H(r)
\end{equation}
for some $\delta_0 \leq x/2 r_1^2$ and
all $r\in \lbrack r_0,r_1\rbrack$. Assume, moreover, that $H(r_1)=1$.
Let $g:\lbrack r_0,
r_1\rbrack \to \mathbb{R}^+$ be a non-increasing function such that
\begin{equation}\label{eq:rien}
\max_{\alpha \in (\mathbb{R}/\mathbb{Z})\setminus \mathfrak{M}_{\delta_0,r}} 
|S_2(\alpha)| \leq g(r)
\end{equation}
for all $r\in \lbrack r_0,r_1\rbrack$ and $\delta_0$ as above. 

Then
\begin{equation}\label{eq:malmu}\begin{aligned}
\frac{1}{\sum_n |a_n|^2} 
&\int_{(\mathbb{R}/\mathbb{Z})\setminus 
\mathfrak{M}_{\delta_0,r_0}} |S_1(\alpha)|^2 |S_2(\alpha)|  d\alpha 
\\ &\leq
g(r_0) \cdot (H(r_0) - I_0) + \int_{r_0}^{r_1} g(r) H'(r) dr ,
\end{aligned}\end{equation}
where
\begin{equation}\label{eq:pentimento}
I_0 = \frac{1}{\sum_n |a_n|^2} \int_{\mathfrak{M}_{\delta_0,r_0}} |S_1(\alpha)|^2 d\alpha.\\
\end{equation}
\end{prop}
The condition $\delta_0\leq x/2 r_1^2$ is there just to ensure that the
arcs in the definition of $\mathfrak{M}_{\delta_0,r}$ do not overlap for $r\leq r_1$.
\begin{proof}
For $r_0\leq r< r_1$, let
\[f(r) = \frac{1}{\sum_n |a_n|^2}  \int_{\mathfrak{M}_{\delta_0,r+1}\setminus \mathfrak{M}_{\delta_0,r}} |S_1(\alpha)|^2 d\alpha.\]
Let
\[f(r_1) = \frac{1}{\sum_n |a_n|^2}  \int_{(\mathbb{R}/\mathbb{Z})\setminus
  \mathfrak{M}_{\delta_0,r_1}} |S_1(\alpha)|^2 d\alpha.\]
Then, by (\ref{eq:rien}),
\[\frac{1}{\sum_n |a_n|^2} \int_{(\mathbb{R}/\mathbb{Z})\setminus \mathfrak{M}_{\delta_0,r_0}} |S_1(\alpha)|^2 |S_2(\alpha)|
d\alpha \leq \sum_{r=r_0}^{r_1} f(r) g(r).\]

By (\ref{eq:qewer}),
\begin{equation}\label{eq:gogol}\begin{aligned}
\sum_{r_0\leq r\leq x} f(r) &= \frac{1}{\sum_n |a_n|^2} 
\int_{\mathfrak{M}_{\delta_0,x+1}\setminus \mathfrak{M}_{\delta_0,r_0}} |S_1(\alpha)|^2 d\alpha \\ &= 
\left(
\frac{1}{\sum_n |a_n|^2}  \int_{\mathfrak{M}_{\delta_0,x+1}} |S_1(\alpha)|^2 d\alpha
\right) - I_0 \leq H(x) - I_0
\end{aligned}\end{equation}
for $x\in \lbrack r_0,r_1)$. Moreover,
\[\begin{aligned}\sum_{r_0\leq r\leq r_1} f(r) &= 
\frac{1}{\sum_n |a_n|^2}  \int_{(\mathbb{R}/\mathbb{Z})\setminus \mathfrak{M}_{\delta_0,r_0}} |S_1(\alpha)|^2\\
&= \left(
\frac{1}{\sum_n |a_n|^2}  \int_{\mathbb{R}/\mathbb{Z}} |S_1(\alpha)|^2\right) - I_0
= 1 - I_0 = H(r_1) - I_0.\end{aligned}\]

We let $F(x) = H(x) - I_0$
and apply Lemma \ref{lem:jardinbota} with $a=r_0$, $b=r_1$. We obtain that
\[\begin{aligned}
\sum_{r=r_0}^{r_1} f(r) g(r) &\leq (\max_{r\geq r_0} g(r)) F(r_0) +
\int_{r_0}^{r_1} (\max_{r\geq u} g(r))  F'(u)\; du\\
&\leq g(r_0) (H(r_0)- I_0) + 
\int_{r_0}^{r_1} g(u) H'(u)\; du.\end{aligned}\]
\end{proof}

\subsection{The minor-arc total}

We now apply Prop.~\ref{prop:palan}.
Inevitably, the main statement
involves some integrals that will have to be evaluated
at the end of the section.

\begin{thm}\label{thm:ostop}
Let $x\geq 10^{25}\cdot \varkappa$, where
$\varkappa\geq 1$.
Let
\begin{equation}\label{eq:lalaz}
S_\eta(\alpha,x) = \sum_n \Lambda(n) e(\alpha n) \eta(n/x).\end{equation}
Let $\eta_*(t) = (\eta_2 \ast_M \varphi)(\varkappa t)$, where $\eta_2$ is as in (\ref{eq:meichu})
and $\varphi: \lbrack 0,\infty)\to \lbrack 0, \infty)$ is continuous and in $\ell^1$.
Let $\eta_+:\lbrack 0,\infty)\to \lbrack 0,\infty)$ 
be a bounded, piecewise differentiable function with $\lim_{t\to \infty}
\eta_+(t)=0$.
Let $\mathfrak{M}_{\delta_0,r}$ be as in (\ref{eq:majdef}) with $\delta_0=8$.
Let $10^5 \leq r_0 < r_1$, where $r_1 = (3/8) (x/\varkappa)^{4/15}$.

Let 
\[Z_{r_0} = \int_{(\mathbb{R}/\mathbb{Z})\setminus \mathfrak{M}_{8,r_0}} 
|S_{\eta_*}(\alpha,x)| |S_{\eta_+}(\alpha,x)|^2 d\alpha.\]
Then
\[Z_{r_0} \leq \left(\sqrt{\frac{|\varphi|_1 x}{\varkappa} (M + T) } + 
\sqrt{S_{\eta_*}(0,x) \cdot E}\right)^2,\]
where
\begin{equation}\label{eq:georgic}\begin{aligned}
S &= \sum_{p>\sqrt{x}} (\log p)^2 \eta_+^2(n/x),\\
T &= C_{\varphi,3}(\log x) \cdot
(S - (\sqrt{J} - \sqrt{E})^2),\\
J&= \int_{\mathfrak{M}_{8,r_0}} |S_{\eta_+}(\alpha,x)|^2\; d\alpha,\\
E &= 
\left((C_{\eta_+,0} + C_{\eta_+,2}) \log x + (2 C_{\eta_+,0} + C_{\eta_+,1})\right)
\cdot x^{1/2},
\end{aligned}\end{equation}
\begin{equation}\label{eq:malus}\begin{aligned}
C_{\eta_+,0} &= 0.7131 \int_0^\infty \frac{1}{\sqrt{t}}
(\sup_{r\geq t} \eta_+(r))^2 dt,\\
C_{\eta_+,1} &= 0.7131 \int_1^\infty \frac{\log t}{\sqrt{t}}
(\sup_{r\geq t} \eta_+(r))^2 dt,\\
C_{\eta_+,2} &= 0.51942 |\eta_+|_\infty^2,\\
C_{\varphi,3}(K) &= \frac{1.04488}{|\varphi|_1} \int_0^{1/K} |\varphi(w)| dw
\end{aligned}\end{equation}
and 
\begin{equation}\label{eq:gypo}\begin{aligned}
M &= 
g(r_0) \cdot 
\left(\frac{\log (r_0+1)+c^+}{\log \sqrt{x} + c^-} \cdot S -
(\sqrt{J}-\sqrt{E})^2 \right)
\\ &+  \left(\frac{2}{\log x + 2 c^-}
\int_{r_0}^{r_1} \frac{g(r)}{r}  dr + \left(
\frac{7}{15}
+ \frac{- 2.14938 + \frac{8}{15} \log \varkappa}{\log x + 2 c^-}\right) g(r_1)\right)
\cdot S 
\end{aligned}\end{equation}
where $g(r)=g_{x/\varkappa,\varphi}(r)$ with $K = \log(x/\varkappa)/2$
 (see (\ref{eq:basia})), $c_+ = 2.3912$ and
$c_- = 0.6294$.
\end{thm}
\begin{proof}
Let $y = x/\varkappa$.
 Let $Q = (3/4) y^{2/3}$, as in \cite[Main Thm.]{Helf}
(applied with $y$ instead of $x$). 
Let $\alpha\in (\mathbb{R}/\mathbb{Z})\setminus \mathfrak{M}_{8,r}$,
where $r\geq r_0$ and $y$ is used instead of $x$ to define
$\mathfrak{M}_{8,r}$ (see (\ref{eq:majdef})). There
exists an approximation $2 \alpha =a/q + \delta/y$ with $q\leq Q$,
$|\delta|/y\leq 1/q Q$. Thus, $\alpha = a'/q' +\delta/2 y$, where
either $a'/q' = a/2q$ or $a'/q' = (a+q)/2q$ holds.
(In particular, if $q'$ is odd, then $q'=q$; if $q'$ is even, then 
$q'$ may be $q$ or $2q$.)

There are three cases:
\begin{enumerate}
\item $q\leq r$. Then either (a) $q'$ is odd and $q'\leq r$ or (b)
$q'$ is even and $q'\leq 2 r$.
Since $\alpha$ is not in $\mathfrak{M}_{8,r}$, then, by
definition (\ref{eq:majdef}), $|\delta|/2 y \geq
\delta_0 r/2 q y$, and so $|\delta|\geq  
\delta_0 r/q = 8 r/q$. In particular, $|\delta|\geq 8$.

Thus, by Prop.~\ref{prop:gorsh},
\begin{equation}\label{eq:gropa}
|S_{\eta_*}(\alpha,x)| = 
|S_{\eta_2\ast_M \phi}(\alpha,y)| 
\leq g_{y,\varphi}\left(\frac{|\delta|}{8} q\right)
\cdot |\varphi|_1 y \leq g_{y,\varphi}(r) \cdot |\varphi|_1 y,
\end{equation}
where we use the fact that $g(r)$ is a decreasing function 
(Lemma \ref{lem:vinc}).
\item $r < q \leq y^{1/3}/6$. Then, 
 by Prop.~\ref{prop:gorsh} and Lemma \ref{lem:vinc},
\begin{equation}\label{eq:grope}
|S_{\eta_*}(\alpha,x)| = 
|S_{\eta_2\ast_M \phi}(\alpha,y)| 
\leq g_{y,\varphi}\left(\max\left(\frac{|\delta|}{8},1\right) q
\right) \cdot |\varphi|_1 y\leq g_{y,\varphi}(r) \cdot |\varphi|_1 y.\end{equation}
\item $q> y^{1/3}/6$. Again by Prop.~\ref{prop:gorsh},
\begin{equation}\label{eq:gropi}
|S_{\eta_*}(\alpha,x)| = 
|S_{\eta_2\ast_M \phi}(\alpha,y)| 
\leq \left(h\left(\frac{y}{K}\right) + 
C_{\varphi,3}(K)\right) |\varphi|_1 y, 
\end{equation}
where $h(x)$ is as in (\ref{eq:flou}). (Note that $C_{\varphi,3}(K)$, as in
(\ref{eq:malus}), equals $C_{\varphi,0,K}/|\phi|_1$, where
$C_{\varphi,0,K}$ is as in (\ref{eq:midin}).)
 We set $K = (\log y)/2$. Since $y = 
x/\kappa\geq 10^{25}$, it follows that
$y/K = 2 y/\log y > 2.16 \cdot 10^{20}$.
\end{enumerate}
Let 
\[\begin{aligned}
r_1 = \frac{3}{8} y^{4/15},\;\;\;\;\;\;\;\;
g(r) = \begin{cases} g_{x,\varphi}(r) &\text{if $r\leq r_1$},\\
g_{x,\varphi}(r_1) &\text{if $r> r_1$.}\end{cases}\end{aligned}\]
By Lemma \ref{lem:vinc}, $g(r)$ is a decreasing function
for $r\geq 175$; moreover, by
Lemma \ref{lem:gosia}, $g_{y,\phi}(r_1) \geq h(2 y/\log y)$, where $h$
is as in (\ref{eq:flou}), and so
$g(r)\geq h(2 y/\log y)$ for all $r\geq r_0 > 175$. Thus, we
have shown that
\begin{equation}\label{eq:bertru}
|S_{\eta_*}(y,\alpha)|\leq 
\left(g(r) + C_{\varphi,3}\left(\frac{\log y}{2}\right)\right) \cdot |\varphi|_1 y \end{equation}
for all $\alpha\in (\mathbb{R}/\mathbb{Z})\setminus \mathfrak{M}_{8,r}$.

We first need to undertake the fairly
dull task of getting non-prime or small $n$ out of the sum defining
$S_{\eta_+}(\alpha,x)$. 
Write \[\begin{aligned}
S_{1,\eta_+}(\alpha,x) &= 
\sum_{p>\sqrt{x}} (\log p) e(\alpha p) \eta_+(p/x),\\
S_{2,\eta_+}(\alpha,x) &= 
\mathop{\sum_{\text{$n$ non-prime}}}_{n>\sqrt{x}} \Lambda(n) e(\alpha n) \eta_+(n/x) +
\sum_{n\leq \sqrt{x}} \Lambda(n) e(\alpha n)\eta_+(n/x).\end{aligned}\]
By the triangle inequality (with weights $|S_{\eta_+}(\alpha,x)|$),
\[\begin{aligned}&\sqrt{\int_{(\mathbb{R}/\mathbb{Z})\setminus \mathfrak{M}_{8,r_0}} 
|S_{\eta_*}(\alpha,x)| |S_{\eta_+}(\alpha,x)|^2 d\alpha}\\ &\leq \sum_{j=1}^2
\sqrt{\int_{(\mathbb{R}/\mathbb{Z})\setminus \mathfrak{M}_{8,r_0}} 
|S_{\eta_*}(\alpha,x)| |S_{j,\eta_+}(\alpha,x)|^2 d\alpha}.\end{aligned}\]
Clearly, \[\begin{aligned}
&\int_{(\mathbb{R}/\mathbb{Z)}\setminus \mathfrak{M}_{8,r_0}}
|S_{\eta_*}(\alpha,x)| |S_{2,\eta_+}(\alpha,x)|^2 d\alpha \\ &\leq
\max_{\alpha \in \mathbb{R}/\mathbb{Z}}
 \left|S_{\eta_*}(\alpha,x)\right| \cdot \int_{\mathbb{R}/\mathbb{Z}}
 |S_{2,\eta_+}(\alpha,x)|^2 d\alpha\\ &\leq
\sum_{n=1}^\infty \Lambda(n) \eta_*(n/x)
\cdot
\left(\sum_{\text{$n$ non-prime}} \Lambda(n)^2 \eta_+(n/x)^2 + 
\sum_{n\leq \sqrt{x}} \Lambda(n)^2 \eta_+(n/x)^2\right).\end{aligned}\]
Let $\overline{\eta_+}(z) = \sup_{t\geq z} \eta_+(t)$.
Since $\eta_+(t)$ tends to $0$ as $t\to \infty$, so does $\overline{\eta_+}$.
By  \cite[Thm. 13]{MR0137689}, partial summation
and integration by parts,
\[\begin{aligned}
\sum_{\text{$n$ non-prime}} &\Lambda(n)^2 \eta_+(n/x)^2 \leq
\sum_{\text{$n$ non-prime}} \Lambda(n)^2 \overline{\eta_+}(n/x)^2\\ &\leq
-\int_1^\infty \left(\mathop{\sum_{n\leq t}}_{\text{$n$ non-prime}} \Lambda(n)^2\right) 
\left( \overline{\eta_+}^2 (t/x) \right)' dt\\
&\leq
-\int_1^\infty (\log t) \cdot 1.4262 \sqrt{t} 
\left( \overline{\eta_+}^2 (t/x) \right)'
dt\\
&\leq  0.7131\int_1^{\infty} \frac{\log e^2 t}{\sqrt{t}} \cdot 
\overline{\eta_+}^2
\left(\frac{t}{x}\right) dt\\ &=
\left(0.7131 \int_{1/x}^\infty \frac{2 + \log t x}{\sqrt{t}} 
\overline{\eta_+}^2(t) dt \right) \sqrt{x},
\end{aligned}\]
while, by \cite[Thm. 12]{MR0137689},
\[\begin{aligned}
\sum_{n\leq \sqrt{x}} \Lambda(n)^2 \eta_+(n/x)^2 &\leq \frac{1}{2} 
|\eta_+|_\infty^2
(\log x) \sum_{n\leq \sqrt{x}} \Lambda(n)\\
&\leq 0.51942 |\eta_+|_\infty^2 \cdot \sqrt{x} \log x.
\end{aligned}\]
This shows that
\[\int_{(\mathbb{R}/\mathbb{Z)}\setminus \mathfrak{M}_{8,r_0}}
|S_{\eta_*}(\alpha,x)| |S_{2,\eta_+}(\alpha,x)|^2 d\alpha \leq 
\sum_{n=1}^\infty \Lambda(n) \eta_*(n/x) \cdot
E = S_{\eta_*}(0,x)\cdot E,\]
where $E$ is as in (\ref{eq:georgic}).

It remains to bound
\begin{equation}\label{eq:flashgo}
\int_{(\mathbb{R}/\mathbb{Z})\setminus \mathfrak{M}_{8,r_0}} 
|S_{\eta_*}(\alpha,x)| |S_{1,\eta_+}(\alpha,x)|^2 d\alpha .
\end{equation}
We wish to apply Prop.~\ref{prop:palan}.
Corollary \ref{cor:coeur} gives us an input of type (\ref{eq:qewer});
we have just derived a bound 
(\ref{eq:bertru}) that provides an input of type (\ref{eq:rien}).
More precisely, by Corollary \ref{cor:coeur},
(\ref{eq:qewer}) holds with
\[H(r) = \begin{cases}
\frac{\log (r+1) + c^+}{\log \sqrt{x} + c^-}
&\text{if $r< r_1$},\\ 1 &\text{if $r\geq r_1$,}\end{cases}\]
where $c^+ = 2.3912 > \log 2 + 1.698$ and $c^- = 0.6294
< \log(1/\sqrt{2\cdot 8}) + \log 2 + 1.3225822$.
(We can apply Corollary
\ref{cor:coeur} because $(2 (r_1+1))\leq  ((4/9) x^{4/15}+2)
\leq (2 \sqrt{x/16})^{0.6}$ for $x\geq 10^{25}$ 
(or even for $x\geq 1000$).)
Since $r_1 = (3/8) y^{4/15}$
and $x\geq 10^{25} \cdot
\varkappa$,
\[\begin{aligned}
\lim_{r\to r_1^+} H(r) &- \lim_{r\to r_1^-} H(r) = 1 - 
\frac{\log ((3/8) (x/\varkappa)^{4/15}+1)+c^+}{\log \sqrt{x} + c^-}\\
&\leq 1 - \left(\frac{4/15}{1/2} +
\frac{\log \frac{3}{8} + c^+ - \frac{4}{15} \log \varkappa - \frac{8}{15} c^-}{
\log \sqrt{x} + c^-}\right)\\
&\leq \frac{7}{15}
+ \frac{- 2.14938 + \frac{8}{15} \log \varkappa}{\log x + 2 c^-} .
\end{aligned}\]
We also have (\ref{eq:rien}) with \begin{equation}\label{eq:jorge}
\left(g(r) + C_{\varphi,3}\left(\frac{\log y}{2}\right)\right)
\cdot |\varphi|_1 y\end{equation}
instead of $g(r)$ (by (\ref{eq:bertru})). 
Here (\ref{eq:jorge}) is a decreasing function of $r$ because $g(r)$ is,
as we already checked.
Hence, Prop.~\ref{prop:palan} gives us that
 (\ref{eq:flashgo}) is at most
\begin{equation}\label{eq:lili}\begin{aligned}
g(r_0) \cdot &(H(r_0) - I_0) + (1-I_0) \cdot C_{\varphi,3}\left(\frac{\log y}{2}\right)
\\ &+  \frac{1}{\log \sqrt{x} + c^-}
\int_{r_0}^{r_1} \frac{g(r)}{r+1}  dr + 0.4156 g(r_1)
\end{aligned}\end{equation}
times $|\varphi|_1 y \cdot \sum_{p>\sqrt{x}} (\log p)^2
\eta_+^2(p/x)$, where
\begin{equation}
I_0 = \frac{1}{\sum_{p>\sqrt{x}} (\log p)^2 \eta_+^2(n/x)}
\int_{\mathfrak{M}_{8,r_0}} 
 |S_{1,\eta_+}(\alpha,x)|^2\; d\alpha.
\end{equation}
By the triangle inequality,
\[\begin{aligned}
 \sqrt{\int_{\mathfrak{M}_{8,r_0}} 
 |S_{1,\eta_+}(\alpha,x)|^2\; d\alpha} &=
\sqrt{\int_{\mathfrak{M}_{8,r_0}} 
 |S_{\eta_+}(\alpha,x) - S_{2,\eta_+}(\alpha,x)|^2\; d\alpha} \\ &\geq
 \sqrt{\int_{\mathfrak{M}_{8,r_0}} 
 |S_{\eta_+}(\alpha,x)|^2\; d\alpha} -
\sqrt{\int_{\mathfrak{M}_{8,r_0}} 
 |S_{2,\eta_+}(\alpha,x)|^2\; d\alpha} \\
&\geq  \sqrt{\int_{\mathfrak{M}_{8,r_0}} 
 |S_{\eta_+}(\alpha,x)|^2\; d\alpha} -
\sqrt{\int_{\mathbb{R}/\mathbb{Z}} 
 |S_{2,\eta_+}(\alpha,x)|^2\; d\alpha}. 
\end{aligned}\] 
As we already showed,
\[\int_{\mathbb{R}/\mathbb{Z}} 
 |S_{2,\eta_+}(\alpha,x)|^2\; d\alpha = 
\mathop{\sum_{\text{$n$ non-prime}}}_{\text{or $n\leq \sqrt{x}$}}
\Lambda(n)^2 \eta_+(n/x)^2\leq E.\]
Thus,
\[I_0\cdot S \geq (\sqrt{J}- \sqrt{E})^2,\]
and so we are done.


\end{proof}

We now should estimate the integral in (\ref{eq:gypo}).
It is easy to see that
\begin{equation}\label{eq:ostram}
\begin{aligned}
\int_{r_0}^\infty \frac{1}{r^{3/2}} dr &= \frac{2}{r_0^{1/2}},\;\;\;\;\;
\int_{r_0}^{\infty} \frac{\log r}{r^2} dr = \frac{\log e r_0}{r_0},\;\;\;\;\;
\int_{r_0}^{\infty} \frac{1}{r^2} dr = \frac{1}{r_0},\\
\int_{r_0}^{r_1} \frac{1}{r} dr = \log \frac{r_1}{r_0},\;\;\;\;\;
&\int_{r_0}^\infty \frac{\log r}{r^{3/2}} dr = \frac{2 \log e^2 r_0}{\sqrt{r_0}}
,\;\;\;\;\;
\int_{r_0}^\infty \frac{\log 2 r}{r^{3/2}} dr = 
\frac{2 \log 2 e^2 r_0}{\sqrt{r_0}},\\
\int_{r_0}^\infty \frac{(\log 2 r)^2}{r^{3/2}} dr = 
&\frac{2 P_2(\log 2 r_0)}{\sqrt{r_0}},\;\;\;\;\;\;
\int_{r_0}^\infty \frac{(\log 2 r)^3}{r^{3/2}} dr = 
\frac{2 P_3(\log 2 r_0)}{r_0^{1/2}},\end{aligned}\end{equation}
where 
\begin{equation}\label{eq:javich}
P_2(t) = t^2 + 4 t + 8,\;\;\;\;\;\;\;\;\;
P_3(t) = t^3 + 6 t^2 + 24 t + 48.\end{equation}
We also have
\begin{equation}\label{eq:maldicho}
\int_{r_0}^\infty \frac{dr}{r^2 \log r} = E_1(\log r_0)
\end{equation}
where $E_1$ is the {\em exponential integral}
\[E_1(z) = \int_z^\infty \frac{e^{-t}}{t} dt.\] 

We must also estimate the integrals
\begin{equation}\label{eq:kurica}
\int_{r_0}^{r_1} \frac{\sqrt{\digamma(r)}}{r^{3/2}} dr,\;\;\;\;\;
\int_{r_0}^{r_1} \frac{\digamma(r)}{r^2} dr,\;\;\;\;\;
\int_{r_0}^{r_1} \frac{\digamma(r) \log r}{r^2} dr,\;\;\;\;\;
\int_{r_0}^{r_1} \frac{\digamma(r)}{r^{3/2}} dr,
\end{equation}

Clearly, $\digamma(r) - e^\gamma \log \log r = 2.50637/\log \log r$
is decreasing on $r$. Hence, for $r\geq 10^5$, 
\[\digamma(r) \leq e^\gamma \log \log r + c_\gamma,\]
where $c_\gamma = 1.025742$. Let $F(t) = e^\gamma \log t + c_\gamma$.
Then $F''(t) = -e^\gamma/t^2 < 0$.
Hence
\[\frac{d^2 \sqrt{F(t)}}{dt^2} = \frac{F''(t)}{2 \sqrt{F(t)}} - 
\frac{(F'(t))^2}{4 (F(t))^{3/2}} < 0
\] 
for all $t>0$. In other words, 
$\sqrt{F(t)}$ is convex-down, and so we can bound $\sqrt{F(t)}$
from above by $\sqrt{F(t_0)} + \sqrt{F}'(t_0)\cdot (t-t_0)$, for any 
$t\geq t_0>0$.
Hence, for $r\geq r_0\geq 10^5$,
\[\begin{aligned}
\sqrt{\digamma(r)} &\leq \sqrt{F(\log r)} \leq
\sqrt{F(\log r_0)} + \frac{d \sqrt{F(t)}}{dt}|_{t= \log r_0} \cdot
\log \frac{r}{r_0}\\
&= \sqrt{F(\log r_0)} + \frac{e^\gamma}{\sqrt{F(\log r_0)}}
\cdot \frac{\log \frac{r}{r_0}}{2 \log r_0} .
\end{aligned}\]
Thus, by (\ref{eq:ostram}),
\begin{equation}\label{eq:jot}\begin{aligned}
\int_{r_0}^{\infty} \frac{\sqrt{\digamma(r)}}{r^{3/2}} dr &\leq
\sqrt{F(\log r_0)} \left(2 - \frac{e^\gamma}{F(\log r_0)}\right) 
\frac{1}{\sqrt{r_0}}+
\frac{e^\gamma}{\sqrt{F(\log r_0)} \log r_0}
\frac{\log e^2 r_0}{\sqrt{r_0}}
\\ &= \frac{2 \sqrt{F(\log r_0)}}{\sqrt{r_0}}
\left(1 + \frac{e^\gamma}{F(\log r_0) \log r_0}\right).
\end{aligned}\end{equation}

The other integrals in (\ref{eq:kurica}) are easier. Just as in
(\ref{eq:jot}), we extend
the range of integration to $\lbrack r_0, \infty\rbrack$. 
Using (\ref{eq:ostram}) and (\ref{eq:maldicho}), we obtain
\[\begin{aligned}
\int_{r_0}^\infty \frac{\digamma(r)}{r^2} dr &\leq
\int_{r_0}^{\infty} \frac{F(\log r)}{r^2} dr =
e^\gamma \left(\frac{\log \log r_0}{r_0} + E_1(\log r_0)\right) + \frac{c_\gamma}{r_0},\\
\int_{r_0}^{\infty} \frac{\digamma(r) \log r}{r^2} dr &\leq
e^\gamma \left(\frac{(1 + \log r_0) \log \log r_0 + 1}{r_0} +
E_1(\log r_0)\right)
+ \frac{ c_\gamma \log e r_0}{r_0},
\end{aligned}\]
By \cite[(6.8.2)]{MR2723248},
\[\begin{aligned}
\frac{1}{r (\log r + 1)} \leq E_1(\log r) &\leq \frac{1}{r \log r}.\\
\end{aligned}\] (The second inequality is obvious.) Hence
\[\begin{aligned}
\int_{r_0}^{\infty} \frac{\digamma(r)}{r^2} dr &\leq
\frac{e^\gamma (\log \log r_0 + 1/\log r_0) + c_\gamma}{r_0},\\
\int_{r_0}^{\infty} \frac{\digamma(r) \log r}{r^2} dr &\leq
\frac{e^\gamma \left(\log \log r_0 + \frac{1}{\log r_0}\right) +  
c_\gamma}{r_0} \cdot \log e r_0
.\end{aligned}\]
Finally,
\[\begin{aligned}
\int_{r_0}^{\infty} \frac{\digamma(r)}{r^{3/2}} &\leq
e^{\gamma}\left( \frac{2 \log \log r_0}{\sqrt{r_0}} + 2 E_1\left(\frac{\log
    r_0}{2}\right)\right) + 
\frac{2 c_\gamma}{\sqrt{r_0}}\\
&\leq \frac{2}{\sqrt{r_0}} \left(F(\log r_0) + \frac{2 e^\gamma}{\log r_0}
\right).
\end{aligned}\]


It is time to estimate
\begin{equation}\label{eq:caushwa}
\int_{r_0}^{r_1} \frac{R_{z,2r} \log 2r \sqrt{\digamma(r)}}{r^{3/2}} dr,
\end{equation}
where $z = y$ or $z=y/((\log y)/2)$ (and $y = x/\varkappa$, as before),
and where $R_{z,t}$ is as defined in (\ref{eq:veror}).
By Cauchy-Schwarz, (\ref{eq:caushwa}) is at most
\[
\sqrt{\int_{r_0}^{r_1} \frac{(R_{z,2r} \log 2r)^2}{r^{3/2}} dr} \cdot
\sqrt{\int_{r_0}^{r_1} \frac{\digamma(r)}{r^{3/2}} dr}.\]
We have already bounded the second integral. Let us look at the first one.
We can write $R_{z,t} = 0.27125 R_{z,t}^\circ + 0.41415$, where
\begin{equation}\label{eq:jamo}
R_{z,t}^\circ  = 
\log \left(1 + \frac{\log 4 t}{2 \log \frac{9 z^{1/3}}{2.004
      t}}\right).\end{equation}
Clearly,
\[R_{z,e^t/4}^\circ = \log \left(1 + \frac{t/2}{\log \frac{36
      z^{1/3}}{2.004}
- t}\right).\]
Now, for $f(t) = \log (c + a t/(b - t))$ and $t\in \lbrack 0,b)$,
\[
f'(t) = \frac{ab}{\left(c + \frac{a t}{b-t}\right) 
(b - t)^2},\;\;\;\;\;\;\;\;
f''(t) = \frac{- a b ((a-2 c) (b-2 t) - 2 c t)}{\left(c + \frac{a t}{b-t}\right)^2 
(b - t)^4}.\]
In our case, $a=1/2$, $c=1$ and $b = \log 36 z^{1/3} - \log(2.004)>0$.
Hence, for $t<b$,
\[-a b((a-2 c) (b-2 t)-2c t) = \frac{b}{2} \left(2 t + \frac{3}{2} (b - 2
  t)
\right)
= \frac{b}{2} \left(\frac{3}{2} b - t\right)>0,
\]
and so $f''(t)>0$. In other words, $t\to R_{z,e^t/4}^\circ$ is convex-up
for $t< b$, i.e., for $e^t/4 < 9 z^{1/3}/2.004$. 
It is easy to check that, since we are assuming $y\geq 10^{25}$,
\[2 r_1 = \frac{3}{16} y^{4/15} <
\frac{9}{2.004} \left(\frac{2 y}{\log y}\right)^{1/3} \leq \frac{9 z^{1/3}}{2.004}.\]
We conclude that $r\to R_{z,2 r}^\circ$ is convex-up on $\log 8 r$ 
for $r\leq r_1$, and hence so is $r\to R_{z,r}$, and so, in turn, is
$r\to R_{z,r}^2$.  Thus, 
for $r\in \lbrack r_0,r_1\rbrack$,
\begin{equation}\label{eq:strona}R_{z,2 r}^2 \leq R_{z,2 r_0}^2 
\cdot \frac{\log r_1/r}{\log r_1/r_0}
+ R_{z,2 r_1}^2 \cdot \frac{\log r/r_0}{\log r_1/r_0}
 .\end{equation}

Therefore, by (\ref{eq:ostram}),
\begin{equation}\label{eq:basmed}\begin{aligned}
&\int_{r_0}^{r_1} \frac{(R_{z,2r} \log 2r)^2}{r^{3/2}} dr
\leq \int_{r_0}^{r_1}  \left(R_{z,2 r_0}^2 
 \frac{\log r_1/r}{\log r_1/r_0}
+ R_{z,2 r_1}^2\frac{\log r/r_0}{\log r_1/r_0}\right) (\log 2 r)^2 
\frac{dr}{r^{3/2}}\\
= & \frac{2 R_{z,2 r_0}^2}{\log \frac{r_1}{r_0}} \left(\left(
\frac{P_2(\log 2 r_0)}{\sqrt{r_0}} - \frac{P_2(\log 2 r_1)}{\sqrt{r_1}}
\right) \log 2 r_1 - 
\left(\frac{P_3(\log 2 r_0)}{\sqrt{r_0}} - \frac{P_3(\log 2 r_1)}{\sqrt{r_1}}
\right)\right)\\
+ & \frac{2 R_{z,2 r_1}^2}{\log \frac{r_1}{r_0}} \left(
\left(\frac{P_3(\log 2 r_0)}{\sqrt{r_0}} - \frac{P_3(\log 2 r_1)}{\sqrt{r_1}}
\right) -
\left(
\frac{P_2(\log 2 r_0)}{\sqrt{r_0}} - \frac{P_2(\log 2 r_1)}{\sqrt{r_1}}
\right) \log 2 r_0
\right)\\
&= 
2 \left(R_{z,2 r_0}^2 - \frac{\log 2 r_0}{\log \frac{r_1}{r_0}}
(R_{z,2r_1}^2 - R_{z,2r_0}^2)\right)
\cdot 
 \left(\frac{P_2(\log 2 r_0)}{\sqrt{r_0}} - \frac{P_2(\log 2 r_1)}{\sqrt{r_1}}\right)\\
&+ 2 \frac{R_{z,2r_1}^2 - R_{z,2r_0}^2}{\log \frac{r_1}{r_0}}
 \left(\frac{P_3(\log 2 r_0)}{\sqrt{r_0}} - \frac{P_3(\log 2
     r_1)}{\sqrt{r_1}}\right)\\
&= 
2 R_{z,2 r_0}^2
\cdot 
 \left(\frac{P_2(\log 2 r_0)}{\sqrt{r_0}} - \frac{P_2(\log 2 r_1)}{\sqrt{r_1}}\right)\\
&+ 2 \frac{R_{z,2r_1}^2 - R_{z,2r_0}^2}{\log \frac{r_1}{r_0}}
 \left(\frac{P_2^-(\log 2 r_0)}{\sqrt{r_0}} - \frac{P_3(\log 2
     r_1)- (\log 2 r_0) P_2(\log 2 r_1)}{\sqrt{r_1}}\right) ,\end{aligned}
\end{equation}
where $P_2(t)$ and $P_3(t)$ are as in (\ref{eq:javich}),
and $P_2^-(t) = P_3(t) - t P_2(t) = 2 t^2 + 16 t + 48$.

Putting all terms together, we conclude that
\begin{equation}\label{eq:byrne}
\int_{r_0}^{r_1} \frac{g(r)}{r} dr \leq
f_0(r_0,y) + f_1(r_0) + f_2(r_0,y),\end{equation}
where 
\begin{equation}\label{eq:cymba}\begin{aligned}
f_0(r_0,y) &= 
\left(\left(1 - c_\varphi\right) \sqrt{I_{0,r_0,r_1,y}}
+ c_\varphi \sqrt{I_{0,r_0,r_1,\frac{2 y}{\log y}}} \right)
\sqrt{\frac{2}{\sqrt{r_0}}  I_{1,r_0}}\\
f_1(r_0) &= \frac{\sqrt{F(\log r_0)}}{\sqrt{2 r_0}} \left(1 
 + \frac{e^\gamma}{F(\log r_0) \log r_0}\right) + \frac{5}{\sqrt{2 r_0}}
\\ &+
\frac{1}{r_0} \left(\left(\frac{13}{4} \log e r_0 + 10.102\right) J_{r_0}
+ \frac{80}{9} \log er_0 + 23.433 
\right)\\
f_2(r_0,y) &= 3.2 \frac{((\log y)/2)^{1/6}}{y^{1/6}} \log \frac{r_1}{r_0},
\end{aligned}\end{equation}
where $F(t) = e^\gamma \log t + c_\gamma$, $c_\gamma = 1.025742$,
$y = x/\varkappa$ (as usual),
\begin{equation}\label{eq:waslight}
\begin{aligned}
I_{0,r_0,r_1,z} &= 
R_{z,2 r_0}^2
\cdot 
 \left(\frac{P_2(\log 2 r_0)}{\sqrt{r_0}} - \frac{P_2(\log 2 r_1)}{\sqrt{r_1}}\right)\\
&+ \frac{R_{z,2r_1}^2 - R_{z,2r_0}^2}{\log \frac{r_1}{r_0}}
 \left(\frac{P_2^-(\log 2 r_0)}{\sqrt{r_0}} - \frac{P_3(\log 2
     r_1)- (\log 2 r_0) P_2(\log 2 r_1)}{\sqrt{r_1}}\right) 
\\
J_{r} &= F(\log r) + \frac{e^\gamma}{\log r},\;\;\;\;\;\;
I_{1,r} = F(\log r) + \frac{2 e^{\gamma}}{\log r},\;\;\;\;\;\;\;
c_\varphi = \frac{C_{\varphi,2,\frac{\log y}{2}}/|\varphi|_1}{\log \frac{\log y}{2}}
\end{aligned}\end{equation}
and $C_{\varphi,2,K}$ is as in (\ref{eq:cecidad}). 

\section{Conclusion}

We now need to gather all results, using the smoothing functions
\[\eta_* = (\eta_2 \ast_M \varphi)(\varkappa t),
\]
where  $\varphi(t) = t^2 e^{-t^2/2}$, $\eta_2 = \eta_1\ast_M \eta_1$
and  $\eta_1 = 2\cdot I_{\lbrack -1/2,1/2\rbrack}$,
and
\[ \eta_+ = h_{200}(t) t e^{-t^2/2},\]
where 
\[\begin{aligned}
h_H(t) = \int_0^\infty h(t y^{-1}) F_H(y) \frac{dy}{y},&\\
h(t) = \begin{cases} t^2 (2-t)^3 e^{t-1/2} &\text{if $t\in \lbrack
    0,2\rbrack$,}\\ 0 &\text{otherwise,}\end{cases}
\;\;\;\;\;\;\;\;\;\;\;
&F_H(t) = \frac{\sin(H \log y)}{\pi \log y}.
\end{aligned}\]
Both $\eta_*$ and $\eta_+$ were studied in \cite{HelfMaj}. We also saw
 $\eta_*$ in Thm.~\ref{thm:ostop} (which actually works for
general $\varphi:\lbrack 0,\infty)\to \lbrack 0,\infty)$, as its statement says). We will set $\kappa$ soon.

We fix a value for $r$, namely, $r = 150000$. Our results will
have to be valid for any $x\geq x_+$, where $x_+$ is fixed. We set
$x_+ = 4.9\cdot 10^{26}$, since we want a result valid for $N\geq 10^{27}$,
and, as was discussed in (\ref{subs:charme}), we will work with $x_+$
slightly smaller than $N/2$.

\subsection{The $\ell_2$ norm over the major arcs: explicit version}

We apply Lemma \ref{lem:drujal} with $\eta=\eta_+$ and $\eta_\circ$
as in (\ref{eq:cleo}). Let us first work out the
error terms defined in (\ref{eq:sreda}).
 Recall that $\delta_0=8$. By \cite[Thm.~1.4]{HelfMaj},
\begin{equation}\label{eq:zakone1}
\begin{aligned}ET_{\eta_+,\delta_0 r/2} &=
\max_{|\delta|\leq \delta_0 r/2} |\err_{\eta,\chi_T}(\delta,x)|\\
&= 3.34\cdot 10^{-11} + \frac{251100}{\sqrt{x_+}}  
\leq 1.1377\cdot 10^{-8},
\end{aligned}\end{equation}
\begin{equation}\label{eq:zakone2}
\begin{aligned}E_{\eta_+,r,\delta_0} &=
\mathop{\mathop{\max_{\chi \mo q}}_{q\leq r\cdot \gcd(q,2)}}_{|\delta|\leq
  \gcd(q,2) \delta_0 r/2 q}
\sqrt{q} |\err_{\eta_+,\chi^*}(\delta,x)|\\ &\leq
6.18 \cdot 10^{-12} + \frac{1.14\cdot 10^{-10}}{\sqrt{2}} +
\frac{1}{\sqrt{x_+}} \left(499100 + 52 \sqrt{300000}\right)\\ &\leq
2.3921 \cdot 10^{-8},
\end{aligned}\end{equation}
where, in the latter case, we are using the fact that the stronger bound
for $q=1$ (namely, (\ref{eq:zakone1})) allows us to assume $q\geq 2$.

 We also need to bound
a few norms: by the estimates in \cite[App.~B.3 and B.5]{HelfMaj}
(applied with $H=200$),
\begin{equation}\label{eq:sazar}\begin{aligned}
\left|\eta_+\right|_1 &\leq 1.062319,\;\;\;\;\;\;\;
\left|\eta_+\right|_2 \leq 0.800129 + \frac{274.8569}{200^{7/2}} \leq 
0.800132\\
\left|\eta_+\right|_\infty &\leq 
1+ 2.06440727 \cdot \frac{1 + \frac{4}{\pi} \log H}{H} \leq
1.079955.
\end{aligned}\end{equation}
By (\ref{eq:glenkin}),
\[\begin{aligned}
S_{\eta_+}(0,x) &= \widehat{\eta_+}(0)\cdot x + O^*\left(\err_{\eta_+,\chi_T}(0,x)\right)\cdot x\\ &\leq (|\eta_+|_1 + ET_{\eta_+,\delta_0 r/2}) x
\leq 1.063 x.
\end{aligned}\]
This is far from optimal, but it will do, since all we wish to do with
this is to bound $K_{r,2}$ in (\ref{eq:sreda}): 
\[\begin{aligned}
K_{r,2} &= (1+ \sqrt{300000}) (\log x)^2 \cdot 1.079955\\ &\cdot 
(2\cdot 1.062319 + (1+ \sqrt{300000}) (\log x)^2 1.079955/x)\\
&\leq 1259.06 (\log x)^2 \leq 
9.71\cdot 10^{-21} x
\end{aligned}\]
for $x\geq x_+$. By (\ref{eq:zakone1}), we also have
\[5.19\delta_0 r \left(ET_{\eta_+,\frac{\delta_0 r}{2}}\cdot \left(|\eta_+|_1 + 
\frac{ET_{\eta_+,\frac{\delta_0 r}{2}}}{2}\right)\right)\leq 0.075272
\]
and
\[\delta_0 r (\log 2 e^2 r) 
\left(E_{\eta_+,r,\delta_0}^2
+ K_{r,2}/x\right)  \leq 1.0034\cdot 10^{-8}.\]

We know (see \cite[App.~B.2 and B.3]{HelfMaj})
that
\begin{equation}\label{eq:lopez}
0.8001287\leq |\eta_\circ|_2 \leq 0.8001288\end{equation} and
\begin{equation}\label{eq:sanchez}
|\eta_+- \eta_\circ|_2 \leq \frac{274.856893}{H^{7/2}} \leq 2.42942 \cdot
10^{-6}.
\end{equation}
Symbolic integration gives
\begin{equation}\label{eq:melancho}
|\eta_\circ'|_2^2 = 2.7375292\dotsc
\end{equation}
We bound $|\eta_\circ^{(3)}|_1$ using the fact that (as we can tell
by taking derivatives) $\eta_\circ^{(2)}(t)$ increases from $0$ at $t=0$
to a maximum within $\lbrack 0,1/2\rbrack$, and then decreases to 
$\eta_\circ^{(2)}(1) = -7$, only to increase to a maximum within
$\lbrack 3/2,2\rbrack$ (equal to that within $\lbrack 0,1/2\rbrack$)
and then decrease to $0$ at $t=2$:
\begin{equation}\label{eq:halr}\begin{aligned}
|\eta_\circ^{(3)}|_1 &= 2 \max_{t\in \lbrack 0,1/2\rbrack}
\eta_\circ^{(2)}(t) - 2 \eta_\circ^{(2)}(1) + 2 \max_{t\in \lbrack 3/2,2\rbrack}
\eta_\circ^{(2)}(t)\\
&= 4 \max_{t\in \lbrack 0,1/2\rbrack}
\eta_{\circ}^{(2)}(t) + 14 \leq 4\cdot 4.6255653 + 14 \leq 32.5023,
\end{aligned}\end{equation}
where we compute the maximum by the bisection method with $30$ iterations
(using interval arithmetic, as always).

We evaluate explicitly
\[\mathop{\sum_{q\leq r}}_{\text{$q$ odd}} \frac{\mu^2(q)}{\phi(q)} =
 6.798779\dotsc
\]
Looking at (\ref{eq:chetvyorg}) and 
(\ref{eq:mardi}), 
we conclude that
\[\begin{aligned}
L_{r,\delta_0} &\leq 2\cdot 6.798779\cdot 0.800132^2 \leq 8.70531
,\\
L_{r,\delta_0} &\geq 2\cdot 6.798779\cdot 0.8001287^2 
+ O^*((\log r + 1.7) \cdot (3.889 \cdot 10^{-6}+5.91\cdot 10^{-12}))\\
&+ O^*\left(1.342\cdot 10^{-5}\right)\cdot 
\left(0.64787 + \frac{\log r}{4 r} + \frac{0.425}{r}\right) \geq 8.70517
.\end{aligned}\]

Lemma \ref{lem:drujal} thus gives us that
\begin{equation}\label{eq:celine}\begin{aligned}
\int_{\mathfrak{M}_{8,r_0}} \left|S_{\eta_+}(\alpha,x)\right|^2 d\alpha
&= (8.70524+O^*(0.00007)) x + O^*(0.075272) x\\ &= (8.7052+O^*(0.0754)) x
\leq 8.7806 x.
\end{aligned}\end{equation}

\subsection{The total major-arc contribution}

First of all, we must bound from below
\begin{equation}\label{eq:ausbeuter}
C_0 = \prod_{p|N} \left(1 - \frac{1}{(p-1)^2}\right) \cdot
\prod_{p\nmid N} \left(1 + \frac{1}{(p-1)^3}\right).\end{equation}
The only prime that we know does not divide $N$ is $2$. Thus, we use the
bound
\begin{equation}\label{eq:arnar}
C_0 \geq 2 \prod_{p>2} \left(1 - \frac{1}{(p-1)^2}\right)
\geq 1.3203236 .\end{equation}

The other main constant is $C_{\eta_\circ,\eta_*}$, 
which we defined in (\ref{eq:vulgo})
and already started to estimate in (\ref{eq:jaram}):
\begin{equation}\label{eq:karlmarx}
C_{\eta_\circ,\eta_*} =
|\eta_\circ|_2^2 \int_0^{\frac{N}{x}} \eta_*(\rho) d\rho + 
2.71 |\eta_\circ'|_2^2 \cdot O^*\left(
\int_0^{\frac{N}{x}} ((2-N/x)+\rho)^2 \eta_*(\rho) d\rho
\right)\end{equation}
provided that $N\geq 2 x$. 
Recall that $\eta_* = (\eta_2 \ast_M \varphi)(\varkappa t)$, where
$\varphi(t) = t^2 e^{-t^2/2}$. Therefore,
\[\begin{aligned}
\int_0^{N/x} \eta_*(\rho) d\rho &= 
\int_0^{N/x} (\eta_2\ast \varphi)(\varkappa \rho) d\rho = 
\int_{1/4}^1 \eta_2(w) \int_0^{N/x} \varphi\left(\frac{\varkappa
    \rho}{w}\right) d\rho \frac{dw}{w}\\
&= \frac{|\eta_2|_1 |\varphi|_1}{\varkappa} - \frac{1}{\varkappa}
\int_{1/4}^1 \eta_2(w) \int_{\varkappa N/x w}^{\infty} \varphi(\rho) d\rho dw.
\end{aligned}\]
Now
\[\int_y^\infty \varphi(\rho) d\rho = y e^{-y^2/2} + 
\sqrt{2} \int_{y/\sqrt{2}}^\infty e^{-t^2} dt < \left(y +
  \frac{2}{y}\right)
e^{-y^2/2}\]
by \cite[(7.8.3)]{MR2723248}. Hence
\[\int_{\varkappa N/x w}^{\infty} \varphi(\rho) d\rho
\leq \int_{2 \varkappa}^{\infty} \varphi(\rho) d\rho < 
\left(2\varkappa +
  \frac{1}{\varkappa}\right)
e^{-2 \varkappa^2}\]
and so, since $|\eta_2|_1=1$,
\begin{equation}\label{eq:sasa}
\begin{aligned}
\int_0^{N/x} \eta_*(\rho) d\rho &\geq
\frac{|\varphi|_1}{\varkappa} - \int_{1/4}^1 \eta_2(w) dw \cdot \left(2 +
  \frac{1}{\varkappa^2}\right) e^{-2 \varkappa^2}\\ &\geq
\frac{|\varphi|_1}{\varkappa} - \left(2 +
  \frac{1}{\varkappa^2}\right) e^{-2 \varkappa^2}.\end{aligned}\end{equation}

Let us now focus on the second integral in (\ref{eq:karlmarx}).
Write $N/x = 2 + c_1/\varkappa$. Then the integral equals
\[\begin{aligned}
\int_0^{2+c_1/\varkappa} &(- c_1/\varkappa+ \rho)^2 \eta_*(\rho) d\rho \leq
\frac{1}{\varkappa^3} \int_0^\infty (u-c_1)^2\; (\eta_2\ast_M \varphi)(u)\;
du\\ &=
\frac{1}{\varkappa^3} \int_{1/4}^1 \eta_2(w)
\int_0^\infty (v w - c_1)^2 \varphi(v) dv dw\\
&= \frac{1}{\varkappa^3} \int_{1/4}^1 \eta_2(w)
\left(3 \sqrt{\frac{\pi}{2}} w^2 - 2\cdot 2 c_1 w 
+ c_1^2 \sqrt{\frac{\pi}{2}}\right) dw \\
&= \frac{1}{\varkappa^3}
\left(\frac{49}{48}\sqrt{\frac{\pi}{2}} - \frac{9}{4} c_1 +
\sqrt{\frac{\pi}{2}} c_1^2\right).
\end{aligned}\]
It is thus best to choose $c_1 = (9/4)/\sqrt{2 \pi} = 0.89762\dotsc$.
Looking up $|\eta_\circ'|_2^2$ in (\ref{eq:melancho}), we obtain
\[\begin{aligned}
2.71 |\eta_\circ'|_2^2 \cdot 
&\int_0^{\frac{N}{x}} ((2-N/x)+\rho)^2 \eta_*(\rho) d\rho \\ &\leq
7.4188 \cdot \frac{1}{\varkappa^3}
\left(\frac{49}{48} \sqrt{\frac{\pi}{2}} - 
\frac{(9/4)^2}{2 \sqrt{2 \pi}}\right)
\leq \frac{2.0002}{\varkappa^3}.\end{aligned}\]
We conclude that
\[C_{\eta_\circ,\eta_*} \geq \frac{1}{\varkappa} |\varphi|_1 |\eta_\circ|_2^2 
- |\eta_\circ|_2^2  \left(2 + \frac{1}{\varkappa^2}\right)
e^{-2 \varkappa^2} - \frac{2.0002}{\varkappa^3}.\]
Setting \[\varkappa = 49\] and using (\ref{eq:lopez}), we obtain
\begin{equation}\label{eq:barbar}
C_{\eta_\circ,\eta_*} \geq 
\frac{1}{\varkappa} (|\varphi|_1 |\eta_\circ|_2^2
- 0.000834).\end{equation}
Here it is useful to note that $|\varphi|_1 =  \sqrt{\frac{\pi}{2}}$, and so,
by (\ref{eq:lopez}), $|\varphi|_1 |\eta_\circ|_2^2 = 0.80237\dotsc$.

We have finally chosen $x$ in terms of $N$:
\begin{equation}\label{eq:warwar}
x = \frac{N}{2 + \frac{c_1}{\varkappa}}  = \frac{N}{2 + 
\frac{9/4}{\sqrt{2 \pi}} \frac{1}{49}}  = 0.495461\dotsc \cdot N.\end{equation}
Thus, we see that, since we are assuming $N\geq 10^{27}$,
we in fact have $x\geq 4.95461\dotsc \cdot 10^{26}$, and so, in
particular, 
\begin{equation}\label{eq:kalm}
x\geq 4.9\cdot 10^{26},\;\;\;\; \frac{x}{\varkappa} \geq 10^{25}.\end{equation}

Let us continue with our determination of the major-arcs total. 
We should compute the quantities in (\ref{eq:vulgato}). We already
have bounds for $E_{\eta_+,r,\delta_0}$, $A_{\eta_+}$ (see (\ref{eq:celine})),
$L_{\eta,r,\delta_0}$ and $K_{r,2}$. 
By \cite[Cor.~1.3]{HelfMaj},
we have
\begin{equation}\label{eq:fabienne}\begin{aligned}
E_{\eta_*,r,8} &\leq 
\mathop{\mathop{\max_{\chi \mo q}}_{q\leq r\cdot \gcd(q,2)}}_{|\delta|\leq
  \gcd(q,2) \delta_0 r/2 q}
\sqrt{q} |\err_{\eta_*,\chi^*}(\delta,x)|\\
&\leq \frac{1}{\varkappa} \left(
4.269\cdot 10^{-14} + \frac{1}{\sqrt{x_+}}\left(380600+76
\sqrt{300000}\right)\right)\\
&\leq \frac{1.9075 \cdot 10^{-8}}{\varkappa},
\end{aligned}\end{equation}
where the factor of $\varkappa$ comes from the scaling in 
$\eta_*(t) = (\eta_2 \ast_M \varphi)(\varkappa t)$ (which in effect
divides $x$ by $\varkappa$).
It remains only to bound the more harmless terms of type
$Z_{\eta,2}$ and $LS_\eta$.

Clearly, $Z_{\eta_+^2,2}\leq (1/x) \sum_n \Lambda(n) (\log n) \eta_+^2(n/x)$.
Now, by \cite[Prop.~1.5]{HelfMaj},
\begin{equation}\label{eq:malavita}
\begin{aligned}
\sum_{n=1}^\infty &\Lambda(n) (\log n) \eta^2(n/x)\\ &= 
\left(0.640206 + O^*\left(2\cdot 10^{-6} + 
\frac{310.84}{\sqrt{x}}\right)\right) x \log x - 0.021095 x\\
&\leq (0.640206 + O^*(3\cdot 10^{-6})) x \log x - 0.021095 x.
\end{aligned}\end{equation}
Thus,
\[Z_{\eta_+^2,2}\leq 0.640209 x \log x.\]
We will proceed a little more crudely for $Z_{\eta_*^2,2}$:
\begin{equation}\label{eq:bavette}\begin{aligned}
Z_{\eta^2_*,2} &= \frac{1}{x} \sum_n \Lambda^2(n) \eta_*^2(n/x) \leq
\frac{1}{x} \sum_n \Lambda(n) \eta_*(n/x) \cdot (\eta_*(n/x) \log n)\\
&\leq
(|\eta_*|_1 + |\err_{\eta_*,\chi_T}(0,x)|)
\cdot (|\eta_*(t) \cdot \log^+(\varkappa t)|_\infty + |\eta_*|_\infty \log
(x/\varkappa)),\end{aligned}\end{equation}
where $\log^+(t) := \max(0,\log t)$. 
It is easy to see that
\begin{equation}\label{eq:macadam}
|\eta_*|_\infty =  |\eta_2\ast_M \varphi|_\infty \leq 
\left|\frac{\eta_2(t)}{t}\right|_1 |\varphi|_\infty
\leq 4 (\log 2)^2 \cdot \frac{2}{e} \leq 1.414,
\end{equation}
and, since $\log^+$ is non-decreasing and $\eta_2$ is supported on
a subset of $\lbrack 0,1\rbrack$,
\[\begin{aligned}
|\eta_*(t) \cdot \log^+(\varkappa t)|_\infty &= |(\eta_2\ast_M \varphi) \cdot
\log^+|_\infty \leq |\eta_2 \ast_M (\varphi\cdot \log^+)|_\infty\\
&\leq \left|\frac{\eta_2(t)}{t}\right|_1 
 \cdot |\varphi \cdot \log^+|_\infty \leq 1.921813\cdot 0.381157 \leq
0.732513\end{aligned}\]
where we bound $|\varphi\cdot \log^+|_\infty$ by the bisection method with
$25$ iterations. We already know that 
\begin{equation}\label{eq:marldoro}
|\eta_*|_1 = \frac{|\eta_2|_1 |\varphi|_1}{\varkappa} = 
\frac{|\varphi|_1}{\varkappa} = \frac{\sqrt{\pi/2}}{\varkappa}.\end{equation}
By \cite[Cor.~1.3]{HelfMaj},
\[|\err_{\eta_*,\chi_T}(0,x)| \leq
4.269 \cdot 10^{-14} + \frac{1}{\sqrt{x_+}} (380600+76)
\leq 1.71973 \cdot 10^{-8}.
\]
We conclude that 
\begin{equation}\label{eq:julie}
Z_{\eta_*^2,2} \leq (\sqrt{\pi/2}/49 + 1.71973\cdot 10^{-8}) (0.732513 +
1.414 \log(x/49))\leq 0.0362 \log x.\end{equation}

We have bounds for $|\eta_*|_\infty$ and $|\eta_+|_\infty$. We can also bound
\[|\eta_* \cdot t|_\infty = \frac{|(\eta_2 \ast_M \varphi)\cdot t|_\infty}{\kappa}
\leq  \frac{|\eta_2|_1 \cdot |\varphi \cdot t|_\infty}{\kappa}
 \leq \frac{3^{3/2} e^{-3/2}}{\kappa};\]
we quote the bound 
\begin{equation}\label{eq:muthit}
|\eta_+\cdot t|_\infty = 1.064735+3.25312\cdot (1+(4/\pi) \log 200)/200 
\leq 1.19073\end{equation} from \cite[\S~B.5]{HelfMaj}.

We can now bound $LS_\eta(x,r)$ for $\eta = \eta_*,\eta_+$:
\[\begin{aligned}
LS_{\eta}(x,r) &= \log r \cdot \max_{p\leq r} \sum_{\alpha\geq 1} \eta\left(
\frac{p^\alpha}{x}\right) \leq
(\log r) \cdot \max_{p\leq r} \left( 
\frac{\log x}{\log p} |\eta|_\infty + \mathop{\sum_{\alpha\geq 1}}_{p^\alpha\geq x} \frac{|\eta \cdot t|_\infty}{p^\alpha/x}\right)\\
&\leq (\log r) \cdot \max_{p\leq r} \left(\frac{\log x}{\log p} |\eta|_\infty +
\frac{|\eta\cdot t|_\infty}{1-1/p}\right) \\ &\leq
\frac{(\log r) (\log x)}{\log 2} |\eta|_\infty + 2 (\log r)
|\eta\cdot t|_\infty,
\end{aligned}\]
and so
\begin{equation}\label{eq:alisa}\begin{aligned}
LS_{\eta_*}&\leq 
\left(\frac{1.414}{\log 2} \log x + 2\cdot \frac{(3/e)^{3/2}}{49}\right) \log r 
\leq 24.32 \log x + 0.57,\\
\\
LS_{\eta_+}&\leq \left(\frac{1.07996}{\log 2} \log x + 2\cdot 1.19073
\right) \log r\leq 18.57\log x + 28.39.\end{aligned}\end{equation}

We can now start to put together all terms in (\ref{eq:opus111}). 
Let $\epsilon_0 = |\eta_+-\eta_\circ|_2/|\eta_\circ|_2$. Then,
by (\ref{eq:sanchez}),
\[\epsilon_0 |\eta_\circ|_2 \leq
|\eta_+ - \eta_\circ|_2 \leq 2.43\cdot 10^{-6}. 
\]
Thus,
\[2.82643 |\eta_\circ|_2^2 (2+\epsilon_0)\cdot \epsilon_0
+ \frac{4.31004 |\eta_\circ|_2^2 +
0.0012 \frac{|\eta_\circ^{(3)}|_1^2}{\delta_0^5}}{r} \]
is at most
\[\begin{aligned}
&2.82643\cdot 2.43\cdot 10^{-6} \cdot (2\cdot 0.80013 + 
2.43\cdot 10^{-6}) \\ &+ 
\frac{4.3101\cdot 0.80013^2 + 0.0012 \cdot \frac{32.503^2}{8^5}}{150000}
\leq 2.9387\cdot 10^{-5}\end{aligned}\]
by (\ref{eq:lopez}), (\ref{eq:halr}), and (\ref{eq:marldoro}).

Since $\eta_* = (\eta_2 \ast_M \varphi)(\varkappa x)$ 
and $\eta_2$ is supported on $\lbrack 1/4,1\rbrack$,
\[\begin{aligned}
|\eta_*|_2^2 &= \frac{|\eta_2 \ast_M \varphi|_2^2}{\varkappa}
= \frac{1}{\varkappa} \int_0^\infty \left(\int_0^\infty \eta_2(t)
\varphi\left(\frac{w}{t}\right) \frac{dt}{t}\right)^2 dw\\
&\leq \frac{1}{\varkappa}
\int_0^\infty \left(1 - \frac{1}{4}\right)
 \int_0^\infty \eta_2^2(t)
\varphi^2\left(\frac{w}{t}\right) \frac{dt}{t^2} dw \\ &=  \frac{3}{4\varkappa}
\int_0^\infty \frac{\eta_2^2(t)}{t} \left(
\int_0^\infty 
\varphi^2\left(\frac{w}{t}\right) \frac{dw}{t}\right)
dt\\ &= \frac{3}{4 \varkappa}
|\eta_2(t)/\sqrt{t}|_2^2 \cdot |\varphi|_2^2
= \frac{3}{4 \varkappa}\cdot 
\frac{32}{3} (\log 2)^3 \cdot \frac{3}{8} \sqrt{\pi} 
\leq \frac{1.77082}{\varkappa}.
\end{aligned}\]

Recalling the bounds on $E_{\eta_*,r,\delta_0}$ and $E_{\eta_+,r,\delta_0}$
we obtained in (\ref{eq:zakone2}) and (\ref{eq:fabienne}), we conclude that
the second line of (\ref{eq:opus111}) is at most $x^2$ times
\[\begin{aligned}
\frac{1.9075\cdot 10^{-8}}{\varkappa} \cdot 8.7806 
 &+ 2.3921\cdot 10^{-8} \cdot 1.6812 \\&\cdot (\sqrt{8.7806} + 1.6812
\cdot 0.80014) \sqrt{\frac{1.77082}{\varkappa}}
&\leq \frac{1.7815 \cdot 10^{-6}}{\varkappa},\end{aligned}\]
where we are using the bound $A_{\eta_+}\leq 8.8013$ we obtained in
(\ref{eq:celine}). (We are also using the bounds on norms in (\ref{eq:sazar}).)

By the bounds (\ref{eq:bavette}), (\ref{eq:julie}) and (\ref{eq:alisa}), we see that the
 third line of (\ref{eq:opus111}) is at most
\[\begin{aligned}
& 2\cdot (0.640209 \log x)
\cdot (24.32\log x + 0.57)\cdot x\\
&+ 4 \sqrt{0.640209 \log x \cdot 0.0362 \log x} (18.57 \log x + 28.39)
x \leq 43 (\log x)^2 x,
\end{aligned}\]
where we just use the very weak assumption $x\geq 10^{15}$ to simplify, though we can by
now assume (\ref{eq:kalm}).

Using the assumption $x\geq x_+= 4.9\cdot 10^{26}$,
we conclude that, for $r=150000$, the integral over the major arcs
\[\int_{\mathfrak{M}_{8,r}} S_{\eta_+}(\alpha,x)^2 S_{\eta_*}(\alpha,x) e(-N \alpha) d\alpha\]
is 
\begin{equation}\label{eq:margot}\begin{aligned}
&C_0 \cdot C_{\eta_0,\eta_*} x^2
+ O^*\left(2.9387\cdot 10^{-5} \cdot \frac{\sqrt{\pi/2}}{\varkappa} x^2 + 
\frac{1.7815\cdot 10^{-6}}{\varkappa} x^2 + 43 (\log x)^2 x\right)\\
&= 
C_0 \cdot C_{\eta_0,\eta_*} x^2 + O^*\left(\frac{3.8613\cdot 10^{-5}\cdot x^2}{\varkappa}\right)
=
C_0 \cdot C_{\eta_0,\eta_*} x^2 + O^*(7.881\cdot 10^{-7} x^2),
\end{aligned}\end{equation}
where $C_0$ and $C_{\eta_0,\eta_*}$ are as in (\ref{eq:vulgo}).
Notice that $C_0 C_{\eta_0,\eta_*} x^2$ is the expected asymptotic
for the integral over all of $\mathbb{R}/\mathbb{Z}$.

Moreover, by (\ref{eq:arnar}), (\ref{eq:barbar}) and (\ref{eq:lopez}),
as well as $|\varphi|_1 = \sqrt{\pi/2}$,
\[\begin{aligned}
C_0 \cdot C_{\eta_0,\eta_*} &\geq 1.3203236 \left(\frac{|\varphi|_1 |\eta_\circ|_2^2}{\varkappa} - \frac{0.000834}{\varkappa}\right)\\ &\geq
\frac{1.0594001}{\varkappa} - \frac{0.001102}{\varkappa} 
\geq \frac{1.058298}{49}.\end{aligned}\]
Hence
\begin{equation}\label{eq:juventud}
\int_{\mathfrak{M}_{8,r}} S_{\eta_+}(\alpha,x)^2 S_{\eta_*}(\alpha,x) e(-N \alpha) 
d\alpha
\geq \frac{1.058259}{\varkappa} x^2,\end{equation}
where, as usual, $\varkappa=49$. This is our total major-arc bound.

\subsection{The minor-arc total: explicit version}

We need to estimate the quantities $E$, $S$, $T$, $J$, $M$
in Theorem \ref{thm:ostop}. Let us start by bounding the constants
in (\ref{eq:malus}). The constants $C_{\eta_+,j}$, $j=0,1,2$, will appear
only in the minor term $A_2$, and so crude bounds on them will do.

By (\ref{eq:sazar}) and (\ref{eq:muthit}),
\[\sup_{r\geq t} \eta_+(r) \leq
\min\left(1.07996,\frac{1.19073}{t}\right)\]
for all $t\geq 0$.
Thus,
\[\begin{aligned}
C_{\eta_+,0} &= 0.7131 \int_0^\infty \frac{1}{\sqrt{t}} 
\left(\sup_{r\geq t} \eta_+(r)\right)^2 dt\\ &\leq
0.7131 \left(\int_0^1  \frac{1.07996^2}{\sqrt{t}} dt  + \int_1^\infty
\frac{1.19073^2}{t^{5/2}} dt\right)
\leq 2.3375.\end{aligned}\]
Similarly,
\[\begin{aligned}
C_{\eta_+,1} &\leq 0.7131 \int_1^\infty \frac{\log t}{\sqrt{t}} 
\left(\sup_{r\geq t} \eta_+(r)\right)^2 dt\\ &\leq
0.7131 \int_1^\infty \frac{1.19073^2 \log t}{t^{5/2}} dt \leq
0.4494.\end{aligned}\]
Immediately from (\ref{eq:sazar}),
\[C_{\eta_+,2} = 0.51941 |\eta_+|_\infty^2 \leq 0.60579.\]

We get
\begin{equation}\label{eq:dubistdie}\begin{aligned}
E &\leq \left((2.3375+0.60579) \log x + (2\cdot 2.3375 + 0.4494)\right) 
\cdot x^{1/2}\\
&\leq (2.9433 \log x + 5.1244) \cdot x^{1/2}
\leq 8.4031\cdot 10^{-12} \cdot x
,\end{aligned}\end{equation}
where $E$ is defined as in (\ref{eq:georgic}), and where
we are using the assumption $x\geq x_+ = 4.9\cdot 10^{26}$.
Using (\ref{eq:fabienne}) and (\ref{eq:marldoro}), we see that
\[S_{\eta_*}(0,x) = (|\eta_*|_1 + O^*(ET_{\eta_*,0})) x =
\left(\sqrt{\pi/2} + O^*(1.9075\cdot 10^{-8})\right) \frac{x}{\varkappa}
.\]
Hence
\begin{equation}\label{eq:hexelor}
S_{\eta_*}(0,x) \cdot E  \leq 1.0532 \cdot 10^{-11}\cdot
\frac{x^2}{\varkappa}
.\end{equation}

We can bound
\begin{equation}\label{eq:felipa}
S \leq \sum_n \Lambda(n) (\log n) \eta_+^2(n/x) \leq
0.640209 x \log x - 0.021095 x\end{equation}
by (\ref{eq:malavita}). Let us now estimate $T$.
Recall that $\varphi(t) = t^2 e^{-t^2/2}$. Since
\[\int_0^u \varphi(t) dt = \int_0^u t^2 e^{-t^2/2} dt \leq \int_0^u t^2 dt =
\frac{u^3}{3},\]
we can bound
\[C_{\varphi,3}\left(\log \frac{x}{\varkappa}\right) 
= \frac{1.04488}{\sqrt{\pi/2}} \int_0^{\frac{2}{\log x/\varkappa}}
t^2 e^{-t^2/2} dt \leq \frac{0.2779}{((\log x/\varkappa)/2)^3}.\]

By (\ref{eq:celine}), we already know that
$J = (8.7052 + O^*(0.0754)) x$. Hence
\begin{equation}\label{eq:je}\begin{aligned}
(\sqrt{J}-\sqrt{E})^2 &=
(\sqrt{(8.7052 + O^*(0.0754)) x} - \sqrt{8.4031\cdot 10^{-12} \cdot x})^2\\
&\geq 8.6297 x , 
\end{aligned}\end{equation}
and so
\[\begin{aligned}
T &= C_{\varphi,3}\left(\frac{1}{2} \log \frac{x}{\varkappa}\right) 
\cdot (S - (\sqrt{J}-\sqrt{E})^2)\\
&\leq \frac{8\cdot 0.2779}{(\log x/\varkappa)^3} \cdot 
(0.640209 x \log x - 0.021095 x
- 8.6297 x) \\ &\leq 0.17792 \frac{8 x \log x}{(\log x/\varkappa)^3} - 2.40405 \frac{8 x}{(\log x/\varkappa)^3}
\\ &\leq 1.42336 \frac{x}{(\log x/\varkappa)^2} - 13.69288 \frac{x}{(\log x/\varkappa)^3}.
\end{aligned}\]
for $\varkappa=49$. Since $x/\varkappa \geq 10^{25}$, this implies that
\begin{equation}\label{eq:lamber}
T\leq 3.5776\cdot 10^{-4} \cdot x.
\end{equation}

It remains to estimate $M$. Let us first look at $g(r_0)$; here
$g=g_{x/\varkappa,\varphi}$, where $g_{x/\varkappa,\varphi}$ is defined as in 
(\ref{eq:basia}) and $\phi(t) = t^2 e^{-t^2/2}$, as usual. Write $y=x/\varkappa$.
 We must estimate the constant $C_{\varphi,2,K}$ defined
in (\ref{eq:midin}):
\[\begin{aligned}
C_{\varphi,2,K} &= - \int_{1/K}^1 \varphi(w) \log w\; dw \leq
- \int_{0}^1 \varphi(w) \log w\; dw \\ &\leq
-\int_0^1 w^2 e^{-w^2/2} \log w\; dw \leq 0.093426,\end{aligned}\]
where again we use VNODE-LP for rigorous numerical integration.
Since $|\varphi|_1= \sqrt{\pi/2}$ and $K = (\log y)/2$, this implies that
\begin{equation}\label{eq:dukas}
\frac{C_{\varphi,2,K}/|\varphi|_1}{\log K} \leq \frac{0.07455}{\log \frac{\log y}{2}}
\end{equation}
and so \begin{equation}\label{eq:calvino}
R_{y,K,\varphi,t} = \frac{0.07455}{\log \frac{\log y}{2}} R_{y/K,t} +
\left(1 - \frac{0.07455}{\log \frac{\log y}{2}}\right) R_{y,t}.\end{equation}

Let $t = 2 r_0 = 300000$; we recall that $K = (\log y)/2$. Recall from (\ref{eq:kalm}) that $y=x/\varkappa
\geq 10^{25}$; thus, $y/K\geq 3.47435\cdot 10^{23}$ and 
$\log((\log y)/2) \geq 3.35976$.
Going back to the definition of $R_{x,t}$ in (\ref{eq:veror}),
we see that
\begin{equation}\label{eq:mali1}
\begin{aligned}
R_{y,,2r_0} &\leq 0.27125 \log \left(1 + \frac{\log (8\cdot 150000)}{2\log \frac{9\cdot (10^{25})^{1/3}}{2.004\cdot 2\cdot 150000}}\right) + 0.41415
\leq 0.58341,
\end{aligned}\end{equation}
\begin{equation}\label{eq:mali2}\begin{aligned}
R_{y/K,2r_0} &\leq 0.27125 \log \left(1 + \frac{\log (8\cdot 150000)}{2\log \frac{9\cdot (3.47435\cdot 10^{23})^{1/3}}{2.004\cdot 2\cdot 150000}}\right) + 0.41415
\leq 0.60295,
\end{aligned}\end{equation}
and so
\[R_{y,K,\varphi,2 r_0} \leq \frac{0.07455}{3.35976} 0.60295 +
\left(1 - \frac{0.07455}{3.35976}\right) 0.58341 \leq 0.58385.\]

Using
\[\digamma(r) = e^{\gamma} \log \log r + \frac{2.50637}{\log \log r} \leq
5.42506,\]
we see from (\ref{eq:veror}) that
\[L_{r_0} = 5.42506\cdot \left(\log 2^{\frac{7}{4}} 150000^{\frac{13}{4}} + 
\frac{80}{9}\right) + \log 2^{\frac{16}{9}} 150000^{\frac{80}{9}} + \frac{111}{5}
\leq 394.316.\]
Going back to (\ref{eq:basia}), we sum up and obtain that
\[\begin{aligned} g(r_0) &= \frac{(0.58385\cdot \log 300000+0.5) \sqrt{5.42506} + 2.5}{
\sqrt{2\cdot 150000}} + \frac{394.316}{150000} + 3.2 \left(\frac{\log y}{2 y}\right)^{1/6}\\ &\leq 0.041014.\end{aligned}\]
Using again the bound $x\geq 4.9\cdot 10^{26}$, we
obtain
\[\begin{aligned}&\frac{\log(150000+1)+c^+}{\log \sqrt{x} + c^-} \cdot S -
(\sqrt{J} - \sqrt{E})^2 \\ &\leq 
\frac{13.6164}{\frac{1}{2} \log x + 0.6294} \cdot (0.640209 x \log x - 
0.021095 x) 
- 8.6297 x \\ &\leq 
17.4347 x - \frac{11.2606 x}{\frac{1}{2} \log x + 0.6294} 
- 8.6297 x \\ &\leq (17.4347-8.6297) x \leq 8.805 x,
\end{aligned}\]
where $c^+ = 2.3912$ and $c^- = 0.6294$.
Therefore,
\begin{equation}\label{eq:roussel}\begin{aligned}
g(r_0)\cdot \left(\frac{\log(150000+1)+c^+}{\log \sqrt{x} + c^-} \cdot S -
(\sqrt{J} - \sqrt{E})^2\right) &\leq 0.041061\cdot 8.805 x\\ &\leq	
0.36155 x.\end{aligned}\end{equation}
This is one of the main terms.

Let $r_1 = (3/8) y^{4/15}$, where, as usual, $y = x/\varkappa$ and
$\varkappa=49$. Then
\begin{equation}\label{eq:mali3}\begin{aligned}
R_{y,2r_1} &= 
0.27125 \log \left(1 + \frac{\log \left(8 \cdot
\frac{3}{8} y^{4/15}\right)}{2 \log \frac{9 y^{1/3}}{2.004\cdot 
\frac{3}{4} y^{4/15}}}\right) + 0.41415\\
&= 0.27125 \log \left(1 + \frac{\frac{4}{15} \log y + \log 3}{
2 \left(\frac{1}{3}-\frac{4}{15}\right) \log y + 2 \log \frac{9}{2.004\cdot
\frac{3}{4}}}\right)
+ 0.41415\\
&\leq 0.27125 \log\left(1+\frac{\frac{4}{15}}{2 \left(\frac{1}{3}-
\frac{4}{15}\right)}
\right) + 0.41415\leq 0.71215.
\end{aligned}\end{equation}
Similarly, for $K=(\log y)/2$ (as usual),
\begin{equation}\label{eq:kalda}\begin{aligned}
R_{y/K,2r_1} &= 
0.27125 \log \left(1 + \frac{\log \left(8 \cdot
\frac{3}{8} y^{4/15}\right)}{2 \log \frac{9 (y/K)^{1/3}}{2.004\cdot 
\frac{3}{4} y^{4/15}}}\right) + 0.41415\\
&= 0.27125 
\log \left(1 + \frac{\frac{4}{15} \log y + \log 3}{
\frac{2}{15} \log y - \frac{2}{3} \log \log y + 2 \log \frac{9 \cdot 2^{1/3}}{2.004\cdot \frac{3}{4}}}\right)
+0.41415\\ &=
 0.27125 
\log \left(3 + 
 \frac{\frac{4}{3} \log \log y - c}{
\frac{2}{15} \log y - \frac{2}{3} \log \log y + 2 \log 
\frac{12\cdot 2^{1/3}}{2.004}}
\right) + 0.41415,
\end{aligned}\end{equation}
where $c = 4 \log(12\cdot 2^{1/3}/2.004) - \log 3$.
Let
\[f(t) = \frac{\frac{4}{3} \log t - c
}{\frac{2}{15} t - \frac{2}{3} \log t + 2 \log \frac{12\cdot 2^{1/3}}{2.004}}.\]
The bisection method with $32$ iterations shows that
\begin{equation}\label{eq:golrof}
f(t) \leq 0.019562618\end{equation}
for $180\leq t\leq 30000$; since $f(t)<0$ for $0<t<180$
(by $(4/3) \log t - c <0$) and since, by $c>20/3$, we have 
$f(t)<(5/2) (\log t)/t$ as soon as $t> (\log t)^2$ (and so, in particular, for
$t>30000$), we see that (\ref{eq:golrof}) is valid for all $t>0$. Therefore,
\begin{equation}\label{eq:mali4}
R_{y/K,2 r_1} \leq 0.71392,\end{equation}
 and so, by (\ref{eq:calvino}), we conclude that
\[R_{y,K,\varphi,2 r_1} \leq \frac{0.07455}{3.35976} \cdot
 0.71392 +
\left(1 - \frac{0.07455}{3.35976}\right) \cdot 0.71215 \leq
0.71219.
\]

Since $r_1 = (3/8) y^{4/15}$ and $\digamma(r)$ is increasing for $r\geq 27$,
 we know that
\begin{equation}\label{eq:corple}\begin{aligned}
\digamma(r_1) &\leq \digamma(y^{4/15})
= e^\gamma \log \log y^{4/15} + \frac{2.50637}{\log \log y^{4/15}} \\
&= e^{\gamma} \log \log y + \frac{2.50637}{\log \log y - \log \frac{15}{4}} 
- e^{\gamma} \log \frac{15}{4}
\leq e^{\gamma} \log \log y - 1.43644
\end{aligned}\end{equation}
for $y\geq 10^{25}$. Hence,
(\ref{eq:veror}) gives us that
\[\begin{aligned}
L_{r_1} &\leq (e^{\gamma} \log \log y - 1.43644) \left(\log
\left(2^{\frac{7}{4}} \left(\frac{3}{8}\right)^{\frac{13}{4}} y^{\frac{13}{15}}
\right) + \frac{80}{9}\right) \\ & +
\log \left(2^{\frac{16}{9}} \left(\frac{3}{8}\right)^{\frac{80}{9}}
 y^{\frac{64}{27}}\right) + \frac{111}{5} \leq
\frac{13}{15} e^\gamma \log y \log \log y +
1.1255 \log y \\ &+  12.3147 \log \log y + 4.78195
\leq (1.8213 \log y + 13.49459) \log \log y.\end{aligned}\]
Moreover, again by (\ref{eq:corple})
\[\sqrt{\digamma(r_1)} \leq
\sqrt{e^{\gamma} \log \log y} - \frac{1.43644}{2 \sqrt{e^{\gamma} \log
\log y}}\] and so, by $y\geq 10^{25}$,
\[\begin{aligned}
&(0.71219 \log \frac{3}{4} y^{\frac{4}{15}}+0.5)
\sqrt{\digamma(r_1)}\\ &\leq
(0.18992 \log y + 0.29512) \left(
\sqrt{e^{\gamma} \log \log y} - \frac{1.43644}{2 \sqrt{e^{\gamma} \log
\log y}}\right)\\
&\leq 0.19505 \sqrt{e^{\gamma} \log \log y} - 
\frac{0.19505\cdot 1.43644\log y}{2 \sqrt{e^\gamma \log \log y}}\\
&\leq
0.26031 \log y \sqrt{\log \log y} - 
3.00147
.\end{aligned}\]
Therefore, by (\ref{eq:basia}),
\[\begin{aligned}g_{y,\varphi}(r_1) &\leq
\frac{0.26031 \log y
\sqrt{\log \log y} +2.5 - 3.00147}{\sqrt{\frac{3}{4} y^{\frac{4}{15}}}} \\ &+ 
\frac{(1.8213 \log y + 13.49459) \log \log y}{\frac{3}{8} y^{\frac{4}{15}}}
+ \frac{3.2 ((\log y)/2)^{1/6}}{y^{1/6}}\\
&\leq
\frac{0.30059 \log y
\sqrt{\log \log y}}{y^{\frac{2}{15}}} + 
\frac{5.48127 \log y \log \log y}{y^{\frac{4}{15}}} +
\frac{0.84323 (\log y)^{1/6}}{y^{1/6}}\\
&\leq \frac{0.30782 \log y \sqrt{\log \log y}}{y^{\frac{2}{15}}},
\end{aligned}\]
where
we use $y\geq 10^{25}$
and verify that the functions 
$t\mapsto 
 (\log t)^{1/6}/t^{1/6-2/15}$ and
$t\mapsto \sqrt{\log \log t}/t^{4/15 - 2/15}$
are decreasing for $t\geq y$ (just taking derivatives).

Since $\varkappa = 49$, one of the terms in (\ref{eq:gypo}) simplifies
easily:
\[\frac{7}{15} + \frac{-2.14938+ \frac{8}{15} \log \varkappa}{\log x + 2 c^-}
\leq \frac{7}{15}.\]
By (\ref{eq:felipa}) and $y = x/\varkappa = x/49$, we conclude that
\begin{equation}\label{eq:comrade}\begin{aligned}
\frac{7}{15} g(r_1) S &\leq
\frac{7}{15} \cdot \frac{0.30782 \log y \sqrt{\log \log y}}{y^{\frac{2}{15}}} 
 \cdot
(0.640209 \log x - 0.021095) x\\
&\leq \frac{0.14365 \log y \sqrt{\log \log y}}{y^{\frac{2}{15}}} 
(0.640209 \log y + 2.4705) x \leq 0.30386 x,
\end{aligned}\end{equation}
where we are using the fact that $y\mapsto (\log y)^2 \sqrt{\log \log y}/y^{2/15}$
is decreasing for $y\geq 10^{25}$ (because 
$y\mapsto (\log y)^{5/2}/y^{2/15}$ is decreasing
for $y\geq e^{75/4}$ and $10^{25}>e^{75/4}$).

It remains only to bound 
\[\frac{2 S}{\log x + 2 c^{-}} \int_{r_0}^{r_1} \frac{g(r)}{r} dr\]
in the expression (\ref{eq:gypo}) for $M$.
We will use the bound on the integral given in (\ref{eq:byrne}).
The easiest term to bound there is $f_1(r_0)$, defined in (\ref{eq:cymba}),
since it depends only on $r_0$: for $r_0 = 150000$,
\[f_1(r_0) = 0.0163662\dotsc.\]
It is also not hard to bound $f_2(r_0,x)$, also defined in (\ref{eq:cymba}):
\[\begin{aligned}f_2(r_0,y) &= 3.2 \frac{((\log y)/2)^{1/6}}{y^{1/6}} \log \frac{
\frac{3}{8} x^{\frac{4}{15}}}{r_0}\\&\leq  
3.2 \frac{(\log y)^{1/6}}{(2 y)^{1/6}}
\left(\frac{4}{15} \log y + 0.05699 - \log r_0\right),\end{aligned}\]
and so, since $r_0 = 150000$ and $y\geq 10^{25}$,
\[f_2(r_0,y) \leq 0.001332.\]

Let us now look at the terms $I_{1,r}$, $c_\varphi$ in (\ref{eq:waslight}).
We already saw in (\ref{eq:dukas}) that
\[c_\varphi = \frac{C_{\varphi,2}/|\varphi|_1}{\log K} \leq \frac{0.07455}{\log 
\frac{\log y}{2}}\leq 0.02219.\]
Since $F(t) = e^{\gamma} \log t + c_\gamma$ with
$c_\gamma = 1.025742$,
\begin{equation}\label{eq:charlemagne}
I_{1,r_0} = F(\log r_0) + \frac{2 e^{\gamma}}{\log r_0} = 5.73826\dotsc
\end{equation}
It thus remains only to estimate $I_{0,r_0,r_1,z}$ for $z=y$ and $z=y/K$,
where $K = (\log y)/2$.

We already know that 
\[\begin{aligned}R_{y,2 r_0} &\leq 0.58341,\;\;\;\;R_{y/K,2r_0} \leq 0.60295,\\
R_{y,2 r_1}&\leq 0.71215,\;\;\;\;R_{y/K,2 r_1} \leq 0.71392\end{aligned}\]
by (\ref{eq:mali1}), (\ref{eq:mali2}), (\ref{eq:mali3}) and (\ref{eq:mali4}).
We also have the trivial bound $R_{z,t}\geq 0.41415$ valid for any $z$ and $t$
for which $R_{z,t}$ is defined. 

Omitting negative terms from (\ref{eq:waslight}), we easily get 
 the following bound, crude but useful enough:
\[
I_{0,r_0,r_1,z} \leq R_{z,2 r_0}^2 \cdot \frac{P_2(\log 2
  r_0)}{\sqrt{r_0}}
+ \frac{R_{z, 2r_1}^2 - 0.41415^2}{\log \frac{r_1}{r_0}}
\frac{P_2^-(\log 2 r_0)}{\sqrt{r_0}},\]
where $P_2(t) = t^2+4t+8$ and $P_2^-(t) = 2 t^2 + 16 t + 48$.
For $z=y$ and $r_0 = 150000$, this gives
\[\begin{aligned}
I_{0,r_0,r_1,y} &\leq 0.58341^2\cdot
\frac{P_2(\log 2
  r_0)}{\sqrt{r_0}}
+ \frac{0.71215^2 - 0.41415^2}{\log \frac{3 y^{4/15}}{8 r_0} 
}\cdot
\frac{P_2^-(\log 2 r_0)}{\sqrt{r_0}}\\
&\leq 0.19115 +
\frac{0.49214}{\frac{4}{15} \log y - \log 800000};
\end{aligned}\]
for $z = y/K$, we proceed in the same way, and obtain
\[I_{0,r_0,r_1,y/K} \leq 0.20416 + \frac{0.49584}{\frac{4}{15} \log y - \log 800000}.\]
This gives us
\begin{equation}\label{eq:marasm}\begin{aligned}
(1-c_\varphi) &\sqrt{I_{0,r_0,r_1,y}} + c_\varphi \sqrt{I_{0,r_0,r_1,\frac{2 y}{\log y}}}\\
&\leq 0.97781\cdot \sqrt{0.19115 +
\frac{0.49214}{\frac{4}{15} \log y - \log 800000}}\\ &+ 0.02219
\sqrt{0.20416 + \frac{0.49584}{\frac{4}{15} \log y - \log 800000}}.
\end{aligned}\end{equation}
We can now conclude the argument in one of two ways. First, we can simply use
the fact that $y\geq 10^{25}$, and obtain that
\[(1-c_\varphi) \sqrt{I_{0,r_0,r_1,y}} + c_\varphi \sqrt{I_{0,r_0,r_1,\frac{2 y}{\log y}}}
\leq 0.68659.\]
Therefore, by (\ref{eq:cymba}),
\[f_0(r_0,y) \leq 
 0.68659\cdot \sqrt{\frac{2}{\sqrt{r_0}} 
5.73827} \leq 0.11819.\]
Again, this is crude,
but it would be just about enough for our purposes.

The alternative is to apply a bound such as
(\ref{eq:marasm}) only for $y$ large. Assume for a moment that
$y\geq 10^{150}$, say. Then
\[\begin{aligned}
R_{y,r_0} &\leq 0.27125\log\left(1 + \frac{\log 4 r_0}{
2 \log \frac{9 (10^{150})^{1/3}}{2.004 r_0}}\right) + 0.41415
\leq 0.43086,
\end{aligned}\]
and, similarly, $R_{2 y/\log y} \leq 0.43113$. Since 
\[0.43086^2\cdot \frac{P_2(\log 2 r_0)}{\sqrt{r_0}} \leq
0.10426,
\;\;\;\;\;\;\;\;
0.43113^2\cdot \frac{P_2(\log 2 r_0)}{\sqrt{r_0}} \leq
0.10439,
\]
we obtain that
 \begin{equation}\label{eq:jijona}\begin{aligned}
(1-c_\varphi) &\sqrt{I_{0,r_0,r_1,y}} + c_\varphi \sqrt{I_{0,r_0,r_1,\frac{2 y}{\log y}}}\\
&\leq 0.97781\cdot \sqrt{0.10426 +
\frac{0.49214}{\frac{4}{15} \log y - \log 800000}}\\ &+ 0.02219
\sqrt{0.10439 + \frac{0.49584}{\frac{4}{15} \log y - \log 800000}}\leq
0.33247
\end{aligned}\end{equation}
for $y\geq 10^{150}$.
For $y$ between $10^{25}$ and $10^{150}$, we evaluate
the left side of (\ref{eq:jijona}) directly, using the definition
(\ref{eq:waslight}) of $I_{0,r_0,r_1,z}$ instead, as well as the bound
$c_\varphi\leq 0.07455/\log ((\log y)/2)$ from (\ref{eq:dukas}).
(It is clear from the second and third lines of (\ref{eq:basmed}) that 
$I_{0,r_0,r_1,z}$ is decreasing on $z$ for $r_0$, $r_1$ fixed, and so
the upper bound for $c_\varphi$ does give the worst case.)
The bisection method
(applied to the interval $\lbrack 25,150\rbrack$ with $30$ iterations,
including $30$ initial iterations)
gives us that
\begin{equation}\label{eq:hasmo}
(1-c_\varphi) \sqrt{I_{0,r_0,r_1,y}} + c_\varphi \sqrt{I_{0,r_0,r_1,\frac{2 y}{\log y}}}
\leq 0.4153461
\end{equation}
for $10^{25}\leq y\leq 10^{140}$. By (\ref{eq:jijona}),
(\ref{eq:hasmo}) is also true for $y>10^{150}$. Hence
\[f_0(r_0,y) \leq 
 0.4153461\cdot \sqrt{\frac{2}{\sqrt{r_0}} 
5.73827} \leq 0.069219.
\]

By (\ref{eq:byrne}), we conclude that
\[\int_{r_0}^{r_1} \frac{g(r)}{r} dr
\leq 0.069219 + 0.016367 + 0.001332 \leq 0.086918.\]
By (\ref{eq:felipa}),
\[\frac{2 S}{\log x + 2 c^-} \leq \frac{2 (0.640209 x \log x - 0.021095 x)}{
\log x + 2 c^-} \leq 2\cdot 0.640209 x = 1.280418 x,\]
where we recall that $c^-=0.6294>0$.
Hence
\begin{equation}\label{eq:casbah}
\frac{2 S}{\log x + 2 c^-} \int_{r_0}^{r_1} \frac{g(r)}{r} dr
\leq 0.111292 x.
\end{equation}

Putting (\ref{eq:roussel}), (\ref{eq:comrade}) 
and (\ref{eq:casbah}) together, we conclude that the quantity $M$
defined in (\ref{eq:gypo}) is bounded by
\begin{equation}\label{eq:bustier}
M\leq 0.36155 x + 0.30386 x + 0.111292 x \leq 0.77671 x.
\end{equation}

Gathering the terms from (\ref{eq:hexelor}),
(\ref{eq:lamber})
and (\ref{eq:bustier}), we see that Theorem \ref{thm:ostop} states
that the minor-arc total
\[Z_{r_0} = \int_{(\mathbb{R}/\mathbb{Z})\setminus \mathfrak{M}_{8,r_0}}
|S_{\eta_*}(\alpha,x)| |S_{\eta_+}(\alpha,x)|^2 d\alpha\]
is bounded by
\begin{equation}\label{eq:rozoj}
\begin{aligned}Z_{r_0} &\leq
\left(\sqrt{\frac{|\varphi|_1 x}{\varkappa} (M+T)} + 
\sqrt{S_{\eta_*}(0,x)\cdot E}\right)^2\\
&\leq \left(\sqrt{|\varphi|_1 (0.77671 + 3.5776\cdot 10^{-4})}
\frac{x}{\sqrt{\varkappa}} + 
\sqrt{1.0532\cdot 10^{-11}} \frac{x}{\sqrt{\varkappa}}
\right)^2\\
&\leq 0.97392 \frac{x^2}{\varkappa}
\end{aligned}\end{equation}
for $r_0=150000$, $x\geq 4.9\cdot 10^{26}$, where we use yet again the
fact that $|\varphi|_1 = \sqrt{\pi/2}$. This is our total minor-arc bound.

\subsection{Conclusion: proof of main theorem}

As we have known from the start,
\begin{equation}\label{eq:masd}\begin{aligned}
\sum_{n_1+n_2+n_3=N} &\Lambda(n_1) \Lambda(n_2) \Lambda(n_3) \eta_+(n_1)
\eta_+(n_2) \eta_*(n_3) \\ &=
\int_{\mathbb{R}/\mathbb{Z}} S_{\eta_+}(\alpha,x)^2 S_{\eta_*}(\alpha,x)
e(-N \alpha) d\alpha.\end{aligned}\end{equation}
We have just shown that, assuming $N\geq 10^{27}$, $N$ odd,
\[\begin{aligned}
\int_{\mathbb{R}/\mathbb{Z}} &S_{\eta_+}(\alpha,x)^2 S_{\eta_*}(\alpha,x)
e(-N \alpha) d\alpha \\ &= 
\int_{\mathfrak{M}_{8,r_0}} S_{\eta_+}(\alpha,x)^2 S_{\eta_*}(\alpha,x)
e(-N \alpha) d\alpha\\ &+ O^*\left(
\int_{(\mathbb{R}/\mathbb{Z})\setminus 
\mathfrak{M}_{8,r_0}} |S_{\eta_+}(\alpha,x)|^2 |S_{\eta_*}(\alpha,x)|
d\alpha\right) \\ &\geq 1.058259 \frac{x^2}{\varkappa}
+ O^*\left(0.97392 \frac{x^2}{\varkappa}\right) \geq
0.08433 \frac{x^2}{\varkappa}
\end{aligned}\]
for $r_0=150000$, where $x = N/(2+9/(196 \sqrt{2 \pi}))$, as in
(\ref{eq:warwar}).
 (We are using (\ref{eq:juventud}) and (\ref{eq:rozoj}).) Recall
that $\varkappa = 49$ and $\eta_*(t) = (\eta_2\ast_M \varphi)(\varkappa t)$,
where $\varphi(t) = t^2 e^{-t^2/2}$.

It only remains to show that the contribution of terms with $n_1$, $n_2$
or $n_3$ non-prime to the sum in (\ref{eq:masd}) is negligible.
(Let us take out $n_1$, $n_2$, $n_3$ equal to $2$ as well, since some
prefer to state the ternary Goldbach conjecture as follows:
every odd number $\geq 9$ is the sum of three {\em odd} primes.)
Clearly
\begin{equation}\label{eq:duke}\begin{aligned}
\mathop{\sum_{n_1+n_2+n_3=N}}_{\text{$n_1$, $n_2$ or $n_3$ even or non-prime}}
 \Lambda(n_1) &\Lambda(n_2) \Lambda(n_3) \eta_+(n_1)
\eta_+(n_2) \eta_*(n_3)\\ \leq
3 |\eta_+|_\infty^2 |\eta_*|_\infty
&\mathop{\sum_{n_1+n_2+n_3=N}}_{\text{$n_1$ even or non-prime}}
 \Lambda(n_1) \Lambda(n_2) \Lambda(n_3)\\
\leq 3 |\eta_+|_\infty^2 |\eta_*|_\infty\cdot 
&(\log N) 
\mathop{\sum_{\text{$n_1\leq N$ non-prime}}}_{\text{or $n_1=2$}}
\Lambda(n_1)
\sum_{n_2\leq N} \Lambda(n_2).
 \end{aligned}\end{equation}
By (\ref{eq:sazar}) and (\ref{eq:macadam}), 
$|\eta_+|_\infty\leq 1.079955$ and $|\eta_*|_\infty \leq 1.414$.
By \cite[Thms. 12 and 13]{MR0137689},
\[\begin{aligned}
\mathop{\sum_{\text{$n_1\leq N$ non-prime}}}_{\text{or $n_1=2$}}
\Lambda(n_1) &< 1.4262 \sqrt{N} + \log 2 < 1.4263 \sqrt{N},\\
\mathop{\sum_{\text{$n_1\leq N$ non-prime}}}_{\text{or $n_1=2$}}
\Lambda(n_1) \sum_{n_2\leq N}
\Lambda(n_2) &= 1.4263 \sqrt{N} \cdot 1.03883 N 
\leq 1.48169 N^{3/2}.
\end{aligned}\]
Hence, the sum on the first line of (\ref{eq:duke}) is at most
\[7.3306 N^{3/2} \log N.\]

Thus, for $N\geq 10^{27}$ odd, 
\[\begin{aligned}
\mathop{\sum_{n_1+n_2+n_3=N}}_{\text{$n_1$, $n_2$, $n_3$ odd primes}}
 &\Lambda(n_1) \Lambda(n_2) \Lambda(n_3) \eta_+(n_1)
\eta_+(n_2) \eta_*(n_3) \\&\geq
0.08433 \frac{x^2}{\varkappa}
- 7.3306 N^{3/2} \log N\\
&\geq 0.00042248 N^2 - 1.4412\cdot 10^{-11} \cdot N^2 \geq 0.000422 N^2
\end{aligned}\]
by $\varkappa=49$ and (\ref{eq:warwar}).
Since $0.000422 N^2>0$, this shows that every odd number $N\geq 10^{27}$
can be written as the sum of three odd primes.

Since the ternary Goldbach conjecture has already been checked for all
$N\leq 8.875\cdot 10^{30}$ \cite{HelPlat},
we conclude that
 every odd number $N>7$ can
be written as the sum of three odd primes, and every odd number $N>5$ can
be written as the sum of three primes. The main theorem is hereby proven:
the ternary Goldbach conjecture is true.

\appendix


\section{Sums over primes}

Here we treat some sums of the type $\sum_n \Lambda(n) \varphi(n)$, where
$\varphi$ has compact support. Since the sums are over all integers (not
just an arithmetic progression) and there is no phase $e(\alpha n)$
involved, the treatment is relatively straightforward.

The following is standard. 
\begin{lem}[Explicit formula]\label{lem:expfor}
Let $\varphi:\lbrack 1,\infty)\to \mathbb{C}$ be continuous and piecewise $C^1$
with $\varphi'' \in \ell_1$; let it also be of compact support contained in
$\lbrack 1,\infty)$. 
Then
\begin{equation}\label{eq:zety}
\sum_n \Lambda(n) \varphi(n) 
 = \int_1^{\infty} \left(1 - \frac{1}{x (x^2-1)} \right) \varphi(x) dx 
- \sum_\rho (M\varphi)(\rho),
\end{equation}
where $\rho$ runs over the non-trivial zeros of $\zeta(s)$.
\end{lem}
The non-trivial zeros of $\zeta(s)$ are, of course, those in the critical
strip
$0<\Re(s)< 1$.

\begin{Rem}
Lemma \ref{lem:expfor} appears as exercise 5 in
\cite[\S 5.5]{MR2061214}; the condition there that $\varphi$ be smooth
can be relaxed, since already the weaker 
assumption that $\varphi''$ be in $L^1$
implies that the Mellin transform $(M\varphi)(\sigma + i t)$ 
decays quadratically on $t$ as $t\to \infty$, thereby guaranteeing 
that the sum $\sum_\rho (M\varphi)(\rho)$ converges absolutely.
\end{Rem}

\begin{lem}\label{lem:crepe}
Let $x\geq 10$. Let $\eta_2$ be as in (\ref{eq:eta2}). Assume that all
non-trivial zeros of $\zeta(s)$ with $|\Im(s)|\leq T_0$ 
lie on the critical line.

Then
\begin{equation}\label{eq:sucre}\begin{aligned}
\sum_n \Lambda(n) \eta_2\left(\frac{n}{x}\right) = x + 
O^*\left(0.135 x^{1/2} + \frac{9.7}{x^2}\right)
+ \frac{\log \frac{e T_0}{2\pi}}{T_0} \left( \frac{9/4}{2 \pi}
+ \frac{6.03}{T_0}\right)  x .
\end{aligned}\end{equation}

In particular, with $T_0 = 3.061\cdot 10^{10}$ 
in the assumption, we have,
for $x\geq 2000$,
\[\sum_n \Lambda(n) \eta_2\left(\frac{n}{x}\right) =
(1+O^*(\epsilon)) x + O^*(0.135 x^{1/2}),
\]
where $\epsilon = 2.73 \cdot 10^{-10}$.
\end{lem}
The assumption that all non-trivial zeros up to $T_0 = 3.061\cdot
10^{10}$ lie on the critical line
was proven rigorously in \cite{Plattpi}; higher values of
$T_0$ have been reached elsewhere (\cite{Wed}, \cite{GD}).
\begin{proof}
By Lemma \ref{lem:expfor},
\[\sum_n \Lambda(n) \eta_2\left(\frac{n}{x}\right) =
\int_1^{\infty} 
 \eta_2\left(\frac{t}{x}\right) dt - 
\int_1^{\infty} \frac{\eta_2(t/x)}{t (t^2-1)} dt
- \sum_\rho (M \varphi)(\rho),\]
where $\varphi(u) = \eta_2(u/x)$ and
$\rho$ runs over all non-trivial zeros of $\zeta(s)$. Since $\eta_2$
is non-negative,
$\int_1^\infty \eta_2(t/x) dt = x |\eta_2|_1 = x$, while
\[\int_1^\infty \frac{\eta_2(t/x)}{t (t^2-1)} dt = O^*\left(
\int_{1/4}^1 \frac{\eta_2(t)}{t x^2 (t^2-1/100)} dt\right) =
O^*\left(\frac{9.61114}{x^2}\right).\]
By (\ref{eq:envy}),
\[\sum_\rho (M\varphi)(\rho) = \sum_\rho M\eta_2(\rho) \cdot x^{\rho} = 
\sum_\rho \left(\frac{1-2^{-\rho}}{\rho}\right)^2 x^{\rho} =
S_1(x) - 2 S_1(x/2) + S_1(x/4),\]
where 
\begin{equation}\label{eq:gormo}
S_m(x) = \sum_\rho \frac{x^{\rho}}{\rho^{m+1}}.\end{equation}
Setting aside the contribution of all $\rho$ with $|\Im(\rho)|\leq T_0$ and
all $\rho$ with $|\Im(\rho)|>T_0$ and $\Re(s)\leq 1/2$, and using the
symmetry provided by the functional equation, we obtain
\[\begin{aligned}
|S_m(x)| &\leq x^{1/2}\cdot \sum_{\rho} \frac{1}{|\rho|^{m+1}} +
x \cdot \mathop{\mathop{\sum_{\rho}}_{|\Im(\rho)|>T_0}}_{|\Re(\rho)| > 1/2} 
\frac{1}{|\rho|^{m+1}}\\
&\leq x^{1/2}\cdot \sum_{\rho} \frac{1}{|\rho|^{m+1}} +
\frac{x}{2} \cdot \mathop{\sum_{\rho}}_{|\Im(\rho)|>T_0}
\frac{1}{|\rho|^{m+1}}.\end{aligned}\]
We bound the first sum by \cite[Lemma 17]{MR0003018} and the second sum
by \cite[Lemma 2]{MR1950435}. We obtain
\begin{equation}\label{eq:shim}|S_m(x)| 
\leq \left(\frac{1}{2 m \pi T_0^m} + \frac{2.68}{T_0^{m+1}}\right)
 x \log \frac{e T_0}{2 \pi} + \kappa_{m} x^{1/2},\end{equation}
where $\kappa_1 = 0.0463$, $\kappa_2=0.00167$ and $\kappa_3 = 0.0000744$.

Hence
\[\begin{aligned}
\left|\sum_\rho (M \eta)(\rho) \cdot x^\rho\right| &\leq
\left(\frac{1}{2 \pi T_0} + \frac{2.68}{T_0^2}\right) \frac{9 x}{4}
\log \frac{e T_0}{2 \pi} + \left(\frac{3}{2} + \sqrt{2}\right) \kappa_1 x^{1/2}.
\end{aligned}\]
For $T_0 = 3.061\cdot 10^{10}$ and $x\geq 2000$, we obtain
\[\sum_n \Lambda(n) \eta_2\left(\frac{n}{x}\right) =
(1+O^*(\epsilon)) x + O^*(0.135 x^{1/2}),
\]
where $\epsilon = 2.73 \cdot 10^{-10}$.
\end{proof}

\begin{cor}\label{cor:austeria}
Let $\eta_2$ be as in (\ref{eq:eta2}). 
Assume that all non-trivial zeros of $\zeta(s)$ with $|\Im(s)|\leq T_0$,
$T_0 = 3.061\cdot 10^{10}$, lie on the critical line.
Then, for all $x\geq 1$,
\begin{equation}\label{eq:jotra}
\sum_n \Lambda(n) \eta_2\left(\frac{n}{x}\right) \leq 
\min\left((1+\epsilon)
x + 0.2 x^{1/2}, 1.04488 x\right),\end{equation}
where $\epsilon = 2.73 \cdot 10^{-10}$.
\end{cor}
\begin{proof}
Immediate from Lemma \ref{lem:crepe} for $x\geq 2000$.
For $x<2000$, we use computation as follows. Since $|\eta_2'|_\infty = 16$ and 
$\sum_{x/4 \leq n\leq x} \Lambda(n)\leq x$ for all $x\geq 0$, computing
$\sum_{n\leq x} \Lambda(n) \eta_2(n/x)$ only for $x\in (1/1000) \mathbb{Z} \cap
\lbrack 0,2000\rbrack$ results in an inaccuracy of at most $(16\cdot 0.0005/
0.9995)x \leq 0.00801 x$. This resolves the matter at all points outside
$(205,207)$ (for the first estimate) or outside $(9.5,10.5)$ and
$(13.5,14.5)$ (for the second estimate). In those intervals,
the prime powers $n$ involved do not change (since whether $x/4 < n \leq x$
depends only on $n$ and $\lbrack x\rbrack$), and thus we can
find the maximum of the sum in (\ref{eq:jotra}) just
by taking derivatives. 
\end{proof}

\section{Sums involving $\phi(q)$}\label{sec:sumphiq}
We need estimates for several sums involving $\phi(q)$ in the denominator.

The easiest are convergent sums, such as
$\sum_q \mu^2(q)/(\phi(q) q)$. We can express this as
$\prod_p (1 + 1/(p (p-1)))$. This is a convergent product,
 and the main task is to bound a tail: for $r$ an integer,
\begin{equation}\label{eq:maloso}\log \prod_{p>r} \left(1 + \frac{1}{p (p-1)}\right)
\leq \sum_{p>r} \frac{1}{p (p-1)} \leq \sum_{n>r} \frac{1}{n (n-1)} = 
\frac{1}{r}.\end{equation}
A quick computation\footnote{Using D. Platt's integer arithmetic package.} 
now suffices to give
\begin{equation}\label{eq:nagasa}
2.591461 \leq \sum_q \frac{\gcd(q,2) \mu^2(q)}{\phi(q) q} <
2.591463\end{equation} and so
\begin{equation}\label{eq:nagasa2}
1.295730 \leq \sum_{\text{$q$ odd}} \frac{\mu^2(q)}{\phi(q) q} <
1.295732,
\end{equation}
since the expression bounded in (\ref{eq:nagasa2}) is exactly half of
that bounded in (\ref{eq:nagasa}).

Again using (\ref{eq:maloso}), we get that
\begin{equation}\label{eq:massacre}
2.826419
 \leq \sum_q \frac{\mu^2(q)}{\phi(q)^2} < 2.826421.\end{equation}
In what follows, we will use values for convergent sums obtained in
much the same way -- an easy tail bound followed by a computation. 


By \cite[Lemma 3.4]{MR1375315},
\begin{equation}\label{eq:ramo}\begin{aligned}
\sum_{q\leq r} \frac{\mu^2(q)}{\phi(q)} 
&= \log r + c_E + O^*(7.284
r^{-1/3}),\\
\mathop{\sum_{q\leq r}}_{\text{$q$ odd}} \frac{\mu^2(q)}{\phi(q)} &= \frac{1}{2} \left(
\log r + c_E + \frac{\log 2}{2}\right) + O^*(4.899 r^{-1/3}),\\
\end{aligned}\end{equation}
where 
\[c_E = \gamma + \sum_p \frac{\log p}{p (p-1)} = 1.332582275+O^*(10^{-9}/3)\]
by \cite[(2.11)]{MR0137689}.
As we already said in (\ref{eq:charpy}), this, supplemented by a
computation for $r\leq 4\cdot 10^7$, gives
\[\log r + 1.312 \leq \sum_{q\leq r} \frac{\mu^2(q)}{\phi(q)} \leq
\log r + 1.354\]
for $r\geq 182$.
In the same way, we get that
\begin{equation}\label{eq:marmo}\frac{1}{2} \log r + 0.83 \leq 
\mathop{\sum_{q\leq r}}_{\text{$q$ odd}} \frac{\mu^2(q)}{\phi(q)} \leq
\frac{1}{2} \log r + 0.85\end{equation}
for $r\geq 195$.
(The numerical verification here goes up to $1.38\cdot 10^8$; for $r>
3.18\cdot 10^8$, use \ref{eq:marmo}.)

Clearly
\begin{equation}\label{eq:dsamo}
\mathop{\sum_{q\leq 2r}}_{\text{$q$ even}} \frac{\mu^2(q)}{\phi(q)} =
\mathop{\sum_{q\leq r}}_{\text{$q$ odd}} \frac{\mu^2(q)}{\phi(q)}.
\end{equation}

We wish to obtain bounds for the sums
\[\sum_{q\geq r} \frac{\mu^2(q)}{\phi(q)^2},\;\;\;\;
\mathop{\sum_{q\geq r}}_{\text{$q$ odd}} \frac{\mu^2(q)}{\phi(q)^2},\;\;\;\;
\mathop{\sum_{q\geq r}}_{\text{$q$ even}} \frac{\mu^2(q)}{\phi(q)^2},\]
where $N\in \mathbb{Z}^+$ and $r\geq 1$. To do this, it will be helpful
to express some of the quantities within these sums as
convolutions.\footnote{The
author would like to thank O. Ramar\'e for teaching him this technique.}
For $q$ squarefree and $j\geq 1$,
\begin{equation}\label{eq:merleau}
\frac{\mu^2(q) q^{j-1}}{\phi(q)^j} = \sum_{ab=q} \frac{f_j(b)}{a},
\end{equation}
where $f_j$ is the multiplicative function defined by
\[f_j(p) = \frac{p^j - (p-1)^j}{(p-1)^j p},\;\;\;\;\;
f_j(p^k) = 0\;\;\;\;\;\text{for $k\geq 2$.}\]
 
We will also find the following estimate useful.
\begin{lem}\label{lem:sidio}
Let $j\geq 2$ be an integer and $A$ a positive real. 
Let $m\geq 1$ be an integer.
Then
\begin{equation}\label{eq:alaspe}
\mathop{\sum_{a\geq A}}_{(a,m)=1} \frac{\mu^2(a)}{a^j}\leq \frac{\zeta(j)/\zeta(2
  j)}{A^{j-1}} \cdot \prod_{p|m} \left(1 + \frac{1}{p^j}\right)^{-1}
.\end{equation}
\end{lem}
It is useful to note that $\zeta(2)/\zeta(4) = 15/\pi^2 = 1.519817\dotsc$ and
$\zeta(3)/\zeta(6) = 1.181564\dotsc$.
\begin{proof}
The right
side of (\ref{eq:alaspe}) decreases as $A$ increases, while the left side depends only
on $\lceil A\rceil$. Hence,
it is enough to prove (\ref{eq:alaspe}) when $A$ is an integer.
 
For $A=1$, (\ref{eq:alaspe}) is an equality. Let
\[C = \frac{\zeta(j)}{\zeta(2
  j)} \cdot \prod_{p|m} \left(1 + \frac{1}{p^j}\right)^{-1}.\]
Let $A\geq 2$. Since
\[\mathop{\sum_{a\geq A}}_{(a,m)=1} \frac{\mu^2(a)}{a^j} = C -
\mathop{\sum_{a<A}}_{(a,m)=1} \frac{\mu^2(a)}{a^j}\]
and
\[\begin{aligned}
C &= \mathop{\sum_a}_{(a,m)=1} \frac{\mu^2(a)}{a^j} < \mathop{\sum_{a<A}}_{
(a,m)=1} \frac{\mu^2(a)}{a^j} + \frac{1}{A^j} 
+ \int_A^\infty \frac{1}{t^j} dt\\ 
&=  \mathop{\sum_{a<A}}_{(a,m)=1}
 \frac{\mu^2(a)}{a^j} + \frac{1}{A^j} + \frac{1}{(j-1) A^{j-1}},
\end{aligned}\]
we obtain
\[\begin{aligned}\mathop{\sum_{a\geq A}}_{(a,m)=1} 
\frac{\mu^2(a)}{a^j} &= \frac{1}{A^{j-1}} \cdot C
+ \frac{A^{j-1}-1}{A^{j-1}} \cdot C
- \mathop{\sum_{a<A}}_{(a,m)=1} \frac{\mu^2(a)}{a^j}\\
&< \frac{C}{A^{j-1}}
+ \frac{A^{j-1}-1}{A^{j-1}} \cdot \left(\frac{1}{A^j} + \frac{1}{
(j-1) A^{j-1}} \right)
- \frac{1}{A^{j-1}} \mathop{\sum_{a<A}}_{(a,m)=1} \frac{\mu^2(a)}{a^j}\\
&\leq \frac{C}{A^{j-1}} +
\frac{1}{A^{j-1}} \left(
\left(1 - \frac{1}{A^{j-1}}\right) \left(\frac{1}{A}
 + \frac{1}{j-1}\right) - 1\right)
.\end{aligned}\]
Since $(1 - 1/A) (1/A+1) < 1$ and $1/A + 1/(j-1)\leq 1$ for $j\geq 3$,
we obtain that 
\[\left(1 - \frac{1}{A^{j-1}}\right) \left(\frac{1}{A}
 + \frac{1}{j-1}\right) < 1\]
for all integers $j\geq 2$, and so the statement follows.
\end{proof}

We now obtain easily the estimates we want: by (\ref{eq:merleau}) and
Lemma \ref{lem:sidio} (with $j=2$ and $m=1$),
\begin{equation}\label{eq:gat1}\begin{aligned}
\sum_{q\geq r} \frac{\mu^2(q)}{\phi(q)^2} &= \sum_{q\geq r}
\sum_{ab = q} \frac{f_2(b)}{a} \frac{\mu^2(q)}{q} \leq \sum_{b\geq 1}
\frac{f_2(b)}{b}
\sum_{a\geq r/b} \frac{\mu^2(a)}{a^2}\\ &\leq \frac{\zeta(2)/\zeta(4)}{r}
\sum_{b\geq 1}  f_2(b) = \frac{\frac{15}{\pi^2}}{r} \prod_p
 \left(1+ \frac{2p-1}{(p-1)^2 p}\right) \leq \frac{6.7345}{r}.
\end{aligned}\end{equation}
Similarly, by (\ref{eq:merleau}) and
Lemma \ref{lem:sidio} (with $j=2$ and $m=2$),
\begin{equation}\label{eq:gat1o}\begin{aligned}
\mathop{\sum_{q\geq r}}_{\text{$q$ odd}}
 \frac{\mu^2(q)}{\phi(q)^2} &= \mathop{\sum_{b\geq 1}}_{\text{$b$ odd}}
\frac{f_2(b)}{b} \mathop{\sum_{a\geq r/b}}_{\text{$a$ odd}}
 \frac{\mu^2(a)}{a^2}
\leq \frac{\zeta(2)/\zeta(4)}{1+1/2^2} \frac{1}{r} \sum_{\text{$b$ odd}}
f_2(b)\\
&= \frac{12}{\pi^2} \frac{1}{r} \prod_{p>2} \left(1 + \frac{2p-1}{(p-1)^2
    p}\right)\leq \frac{2.15502}{r}
\end{aligned}\end{equation}
\begin{equation}\label{eq:gat1e}
\mathop{\sum_{q\geq r}}_{\text{$q$ even}} \frac{\mu^2(q)}{\phi(q)^2} =
\mathop{\sum_{q\geq r/2}}_{\text{$q$ odd}} \frac{\mu^2(q)}{\phi(q)^2} \leq
\frac{4.31004}{r}.\end{equation}

Lastly,
\begin{equation}\label{eq:gatosbuenos}\begin{aligned}
\mathop{\sum_{q\leq r}}_{\text{$q$ odd}} \frac{\mu^2(q) q}{\phi(q)} &= 
\mathop{\sum_{q\leq r}}_{\text{$q$ odd}} \mu^2(q)
\sum_{d|q} \frac{1}{\phi(d)} = \mathop{\sum_{d\leq r}}_{\text{$d$ odd}}
 \frac{1}{\phi(d)} \mathop{\mathop{\sum_{q\leq r}}_{d|q}}_{\text{$q$ odd}}
 \mu^2(q) \leq \mathop{\sum_{d\leq r}}_{\text{$d$ odd}} \frac{1}{2 \phi(d)}
\left(\frac{r}{d} + 1\right)\\ &\leq 
\frac{r}{2} \sum_{\text{$d$ odd}} \frac{1}{\phi(d) d}
+ \frac{1}{2} \mathop{\sum_{d\leq r}}_{\text{$d$ odd}} \frac{1}{\phi(d)}
\leq 0.64787 r + \frac{\log r}{4} + 0.425,\end{aligned}\end{equation}
where we are using (\ref{eq:nagasa2}) and (\ref{eq:marmo}).

\section{Checking small $n$ by checking zeros of $\zeta(s)$}\label{sec:appa}

In order to show that every odd number $n\leq N$ is
the sum of three primes, it is enough to show for some $M\leq N$ that
\begin{enumerate}
\item\label{it:oshp} every even integer $4\leq m\leq M$ can be written as the sum of two primes,
\item\label{it:gaps} the difference between any 
two consecutive primes $\leq N$ is at most $M-4$.
\end{enumerate}
(If we want to show that every
odd number $n\leq N$ is the sum of three {\em odd} primes, we just replace
$M-4$ by $M-6$ in (\ref{it:gaps}).)
The best known result of type (\ref{it:oshp}) is that of  Oliveira 
e Silva, Herzog and Pardi (\cite{OSHP}, $M= 4\cdot 10^{18}$).
As for (\ref{it:gaps}), it was proven in \cite{HelPlat} for
$M = 4\cdot 10^{18}$ and $N = 8.875694 \cdot 10^{30}$ by a direct
computation (valid with $M-4$ or $M-6$ in the statement of
(\ref{it:gaps})). See \S \ref{subs:checkgold}.

Alternatively, one can establish results of type (\ref{it:gaps}) by
means of numerical verifications of the Riemann hypothesis up to a
certain height. This is a classical approach, followed in
\cite{MR0457373} and \cite{MR0457374}, and later in \cite{MR1950435};
we will use the version of (\ref{it:oshp})
 kindly provided by Ramar\'e in \cite{Rpc}.
We carry out this approach in full here, not because it is preferrable to \cite{HelPlat}
-- it is still based on computations, and it is slightly more indirect than
\cite{HelPlat}  --
but simply to show that one can establish what we need by a different
route.

A numerical verification of the Riemann hypothesis up to a certain
height consists simply in checking that all (non-trivial)
zeroes $z$ of the Riemann zeta function up
to a height $H$ (meaning: $\Im(z)\leq H$) lie on the critical line
$\Re(z)=1/2$. 

The height up to which the Riemann hypothesis has actually been fully
verified is not a matter on which there is unanimity. The strongest claim
in the literature is in \cite{GD}, which states that the first $10^{13}$
zeroes of the Riemann zeta function lie on the critical line $\Re(z)=1/2$.
This corresponds to checking the Riemann hypothesis up to height
$H = 2.44599\cdot 10^{12}$.
It is unclear whether this computation was or could be easily made
rigorous; as pointed out in \cite[p. 2398]{MR2684372}, it has not 
been replicated yet.

Before \cite{GD}, the strongest results were those of the ZetaGrid
distributed computing project led by S. Wedeniwski \cite{Wed}; the method
followed in it was more traditional, and should allow rigorous verification
involving interval arithmetic. Unfortunately, the results were never
formally published. The statement that the ZetaGrid project verified the
first $9\cdot 10^{11}$ zeroes (corresponding to $H = 2.419\cdot 10^{11}$)
is often quoted (e.g., \cite[p. 29]{MR2684771}); this is the point
to which the project had got by the time of Gourdon and Demichel's
announcement. Wedeniwski asserts
in private communication that the project verified the first $10^{12}$
zeroes, and that the computation was double-checked (by the same method).

The strongest claim  prior to ZetaGrid was that of van de Lune
($H = 3.293\cdot 10^9$, first $10^{10}$ zeroes; unpublished).
Recently, Platt \cite{Plattpi} checked the first $1.1\cdot 10^{11}$ 
zeroes ($H = 3.061\cdot 10^{10}$)
rigorously, following a method essentially based on that in
\cite{MR2293591}. Note that \cite{Plattpi} uses interval arithmetic,
which is highly desirable for floating-point computations.

\begin{prop}\label{prop:gosto}
Every odd integer $5\leq n\leq n_0$ is the sum of three primes, where
\[n_0 = \begin{cases} 5.90698\cdot 10^{29} &\text{if \cite{GD} is used
-- $H= 2.44\cdot 10^{12}$,}\\
6.15697 \cdot 10^{28} &\text{if 
ZetaGrid results are used ($H = 2.419\cdot 10^{11}$),} \\
1.23163
 \cdot 10^{27} &\text{if \cite{Plattpi} is used (
$H = 3.061\cdot 10^{10}$).}
\end{cases}
\]
\end{prop}
\begin{proof}
For $n\leq 4\cdot 10^{18}+3$, this is immediate from \cite{OSHP}. Let
$4\cdot 10^{18}+3 < n \leq n_0$.
We need to show that there is a prime $p$ in
$\lbrack n-4-(n-4)/\Delta,n-4\rbrack$, where $\Delta$ is large enough for
$(n-4)/\Delta \leq 4\cdot 10^{18}-4$ to hold. We will then have that 
$4\leq n-p \leq 4+(n-4)/\Delta \leq 4\cdot 10^{18}$. Since $n-p$ is even,
\cite{OSHP} will then imply that $n-p$ is the sum of two primes $p'$, $p''$,
and so
\[n = p + p' + p''.\]

Since $n-4>10^{11}$, the interval $\lbrack n-4-(n-4)/\Delta,n-4\rbrack$
with $\Delta=28314000$ must contain a prime \cite{MR1950435}.
This gives the solution for $(n-4)\leq 1.1325 \cdot 10^{26}$, since
then $(n-4)\leq 4\cdot 10^{18}-4$. Note $1.1325\cdot 10^{26}>e^{59}$.

From here onwards, we use the tables in \cite{Rpc} to find acceptable
values of $\Delta$. Since $n-4\geq e^{59}$, we can choose
\[\Delta = \begin{cases}
52211882224 &\text{if \cite{GD} is used (case (a)),}\\
13861486834 &\text{if ZetaGrid is used (case (b)),}\\
307779681 &\text{if \cite{Plattpi} is used (case (c)).} 
\end{cases}\]
This gives us $(n-4)/\Delta \leq 4\cdot 10^{18}-4$ for $n-4<e^{r_0}$, where
$r_0 = 67$ in case (a), $r_0 = 66$ in case (b) and $r_0 = 62$ in case (c).

If $n-4\geq e^{r_0}$, we can choose (again by \cite{Rpc})
\[\Delta = \begin{cases}
146869130682 &\text{in case (a),}\\
15392435100 &\text{in case (b),}\\
307908668 &\text{in case (c).}
\end{cases}\]
This is enough for $n-4<e^{68}$ in case (a), and without further conditions
for (b) or (c).

Finally, if $n-4\geq e^{68}$ and we are in case (a), \cite{Rpc}
assures us that the choice
\[\Delta = 147674531294\]
is valid; we verify as well that $(n_0-4)/\Delta \leq 4\cdot 10^{18}-4$.
\end{proof}

In other words, the rigorous results in \cite{Plattpi} are enough to show
the result for all odd $n\leq 10^{27}$. Of course, \cite{HelPlat} is also
more than enough, and gives stronger results than Prop.~\ref{prop:gosto}.
\bibliographystyle{alpha}
\bibliography{arcs}

\def\cprime{$'$} \def\cprime{$'$} \def\cprime{$'$}
\begin{thebibliography}{OeSHP13}

\bibitem[BBO10]{Mellin}
J.~Bertrand, P.~Bertrand, and J.-P. Ovarlez.
\newblock Mellin transform.
\newblock In A.~D. Poularikas, editor, {\em Transforms and applications
  handbook}. CRC Press, Boca Raton, FL, 2010.

\bibitem[Bom74]{MR0371840}
E.~Bombieri.
\newblock {\em Le grand crible dans la th\'eorie analytique des nombres}.
\newblock Soci\'et\'e Math\'ematique de France, Paris, 1974.
\newblock Avec une sommaire en anglais, Ast{\'e}risque, No. 18.

\bibitem[Bom10]{MR2684771}
E.~Bombieri.
\newblock The classical theory of zeta and {$L$}-functions.
\newblock {\em Milan J. Math.}, 78(1):11--59, 2010.

\bibitem[Boo06]{MR2293591}
A.~R. Booker.
\newblock Artin's conjecture, {T}uring's method, and the {R}iemann hypothesis.
\newblock {\em Experiment. Math.}, 15(4):385--407, 2006.

\bibitem[Bou99]{MR1726234}
J.~Bourgain.
\newblock On triples in arithmetic progression.
\newblock {\em Geom. Funct. Anal.}, 9(5):968--984, 1999.

\bibitem[CW89]{MR1046491}
J.~R. Chen and T.~Z. Wang.
\newblock On the {G}oldbach problem.
\newblock {\em Acta Math. Sinica}, 32(5):702--718, 1989.

\bibitem[Dav67]{MR0217022}
H.~Davenport.
\newblock {\em Multiplicative number theory}, volume 1966 of {\em Lectures
  given at the University of Michigan, Winter Term}.
\newblock Markham Publishing Co., Chicago, Ill., 1967.

\bibitem[Des77]{MR0466050}
J.-M. Deshouillers.
\newblock Sur la constante de \v {S}nirel\cprime man.
\newblock In {\em S\'eminaire {D}elange-{P}isot-{P}oitou, 17e ann\'ee:
  (1975/76), {T}h\'eorie des nombres: {F}ac. 2, {E}xp. {N}o. {\rm {G}}16},
  page~6. Secr\'etariat Math., Paris, 1977.

\bibitem[DEtRZ97]{MR1469323}
J.-M. Deshouillers, G.~Effinger, H.~te~Riele, and D.~Zinoviev.
\newblock A complete {V}inogradov {$3$}-primes theorem under the {R}iemann
  hypothesis.
\newblock {\em Electron. Res. Announc. Amer. Math. Soc.}, 3:99--104, 1997.

\bibitem[Dic66]{MR0245499}
L.~E. Dickson.
\newblock {\em History of the theory of numbers. {V}ol. {I}: {D}ivisibility and
  primality.}
\newblock Chelsea Publishing Co., New York, 1966.

\bibitem[GD04]{GD}
X.~Gourdon and P.~Demichel.
\newblock The first {$10^{13}$} zeros of the {R}iemann zeta function, and zeros
  computation at very large height.
\newblock
  \url{http://numbers.computation.free.fr/Constants/Miscellaneous/zetazeros1e1%
3-1e24.pdf}, 2004.

\bibitem[GR00]{MR1773820}
I.~S. Gradshteyn and I.~M. Ryzhik.
\newblock {\em Table of integrals, series, and products}.
\newblock Academic Press Inc., San Diego, CA, sixth edition, 2000.
\newblock Translated from the Russian, Translation edited and with a preface by
  Alan Jeffrey and Daniel Zwillinger.

\bibitem[HB85]{MR834356}
D.~R. Heath-Brown.
\newblock The ternary {G}oldbach problem.
\newblock {\em Rev. Mat. Iberoamericana}, 1(1):45--59, 1985.

\bibitem[Hela]{HelfMaj}
H.~A. Helfgott.
\newblock Major arcs for {G}oldbach's problem.
\newblock Preprint. Available at {\texttt{arXiv:1203.5712}}.

\bibitem[Helb]{Helf}
H.~A. Helfgott.
\newblock Minor arcs for {G}oldbach's problem.
\newblock Preprint. Available as {\texttt{arXiv:1205.5252}}.

\bibitem[HL23]{MR1555183}
G.~H. Hardy and J.~E. Littlewood.
\newblock Some problems of `{P}artitio numerorum'; {III}: {O}n the expression
  of a number as a sum of primes.
\newblock {\em Acta Math.}, 44(1):1--70, 1923.

\bibitem[HP]{HelPlat}
H.~A. Helfgott and D.~Platt.
\newblock Numerical verification of the ternary {G}oldbach conjecture up to up
  to {$8.875e30$}.
\newblock To appear in {\em Experiment. Math.} Available at
  {\texttt{arXiv:1305.3062}}.

\bibitem[IK04]{MR2061214}
H.~Iwaniec and E.~Kowalski.
\newblock {\em Analytic number theory}, volume~53 of {\em American Mathematical
  Society Colloquium Publications}.
\newblock American Mathematical Society, Providence, RI, 2004.

\bibitem[KP{\v{S}}72]{MR0414506}
N.~I. Klimov, G.~Z. Pil{\cprime}tja{\u\i}, and T.~A. {\v{S}}eptickaja.
\newblock An estimate of the absolute constant in the {G}oldbach-\v
  {S}nirel\cprime man problem.
\newblock In {\em Studies in number theory, {N}o. 4 ({R}ussian)}, pages 35--51.
  Izdat. Saratov. Univ., Saratov, 1972.

\bibitem[Lam08]{Lamb}
B.~Lambov.
\newblock Interval arithmetic using {SSE}-2.
\newblock In {\em Reliable Implementation of Real Number Algorithms: Theory and
  Practice. {I}nternational Seminar {D}agstuhl Castle, {G}ermany, January 8-13,
  2006}, volume 5045 of {\em Lecture Notes in Computer Science}, pages
  102--113. Springer, Berlin, 2008.

\bibitem[Lan12]{Land}
E.~Landau.
\newblock {Gel\"oste und ungel\"oste Probleme aus der Theorie der
  Primzahlverteilung und der {\it Riemann}schen Zetafunktion.}
\newblock In {\em Proceedings of the fifth Itnernational Congress of
  Mathematicians}, volume~1, pages 93--108. Cambridge, 1912.

\bibitem[Lin41]{MR0004266}
U.~V. Linnik.
\newblock {T}he large sieve.
\newblock {\em C. R. (Doklady) Acad. Sci. URSS (N.S.)}, 30:292--294, 1941.

\bibitem[LW02]{MR1932763}
M.-Ch. Liu and T.~Wang.
\newblock On the {V}inogradov bound in the three primes {G}oldbach conjecture.
\newblock {\em Acta Arith.}, 105(2):133--175, 2002.

\bibitem[Mon71]{MR0337847}
H.~L. Montgomery.
\newblock {\em Topics in multiplicative number theory}.
\newblock Lecture Notes in Mathematics, Vol. 227. Springer-Verlag, Berlin,
  1971.

\bibitem[MV73]{MR0374060}
H.~L. Montgomery and R.~C. Vaughan.
\newblock The large sieve.
\newblock {\em Mathematika}, 20:119--134, 1973.

\bibitem[OeSHP13]{OSHP}
T.~Oliveira~e Silva, S.~Herzog, and S.~Pardi.
\newblock Empirical verification of the even goldbach conjecture, and
  computation of prime gaps, up to $4\cdot 10^{18}$.
\newblock Accepted for publication in Math. Comp., 2013.

\bibitem[OLBC10]{MR2723248}
F.~W.~J. Olver, D.~W. Lozier, R.~F. Boisvert, and Ch.~W. Clark, editors.
\newblock {\em N{IST} handbook of mathematical functions}.
\newblock U.S. Department of Commerce National Institute of Standards and
  Technology, Washington, DC, 2010.
\newblock With 1 CD-ROM (Windows, Macintosh and UNIX).

\bibitem[Plaa]{Plattpi}
D.~Platt.
\newblock Computing $\pi(x)$ analytically.
\newblock To appear in {\em Math. Comp.}. Available as
  {\texttt{arXiv:1203.5712}}.

\bibitem[Plab]{Plattfresh}
D.~Platt.
\newblock Numerical computations concerning {G}{R}{H}.
\newblock Preprint. Available at {\texttt{arXiv:1305.3087}}.

\bibitem[Pla11]{Platt}
D.~Platt.
\newblock {\em Computing degree $1$ L-functions rigorously}.
\newblock PhD thesis, Bristol University, 2011.

\bibitem[Ram]{Rpc}
O.~Ramar{\'e}.
\newblock Short effective intervals containing primes, ii.
\newblock Preprint.

\bibitem[Ram95]{MR1375315}
O.~Ramar{\'e}.
\newblock On \v {S}nirel\cprime man's constant.
\newblock {\em Ann. Scuola Norm. Sup. Pisa Cl. Sci. (4)}, 22(4):645--706, 1995.

\bibitem[Ram09]{MR2493924}
O.~Ramar{\'e}.
\newblock {\em Arithmetical aspects of the large sieve inequality}, volume~1 of
  {\em Harish-Chandra Research Institute Lecture Notes}.
\newblock Hindustan Book Agency, New Delhi, 2009.
\newblock With the collaboration of D. S. Ramana.

\bibitem[Ram10]{MR2607306}
O.~Ramar{\'e}.
\newblock On {B}ombieri's asymptotic sieve.
\newblock {\em J. Number Theory}, 130(5):1155--1189, 2010.

\bibitem[Ric01]{MR1836932}
J.~Richstein.
\newblock Verifying the {G}oldbach conjecture up to {$4\cdot 10^{14}$}.
\newblock {\em Math. Comp.}, 70(236):1745--1749 (electronic), 2001.

\bibitem[Ros41]{MR0003018}
B.~Rosser.
\newblock Explicit bounds for some functions of prime numbers.
\newblock {\em Amer. J. Math.}, 63:211--232, 1941.

\bibitem[RS62]{MR0137689}
J.~B. Rosser and L.~Schoenfeld.
\newblock Approximate formulas for some functions of prime numbers.
\newblock {\em Illinois J. Math.}, 6:64--94, 1962.

\bibitem[RS75]{MR0457373}
J.~Barkley Rosser and Lowell Schoenfeld.
\newblock Sharper bounds for the {C}hebyshev functions {$\theta (x)$} and
  {$\psi (x)$}.
\newblock {\em Math. Comp.}, 29:243--269, 1975.
\newblock Collection of articles dedicated to Derrick Henry Lehmer on the
  occasion of his seventieth birthday.

\bibitem[RS03]{MR1950435}
O.~Ramar{\'e} and Y.~Saouter.
\newblock Short effective intervals containing primes.
\newblock {\em J. Number Theory}, 98(1):10--33, 2003.

\bibitem[RV83]{MR706639}
H.~Riesel and R.~C. Vaughan.
\newblock On sums of primes.
\newblock {\em Ark. Mat.}, 21(1):46--74, 1983.

\bibitem[Sch33]{MR1512821}
L.~Schnirelmann.
\newblock \"{U}ber additive {E}igenschaften von {Z}ahlen.
\newblock {\em Math. Ann.}, 107(1):649--690, 1933.

\bibitem[Sch76]{MR0457374}
L.~Schoenfeld.
\newblock Sharper bounds for the {C}hebyshev functions {$\theta (x)$} and
  {$\psi (x)$}. {II}.
\newblock {\em Math. Comp.}, 30(134):337--360, 1976.

\bibitem[SD10]{MR2684372}
Y.~Saouter and P.~Demichel.
\newblock A sharp region where {$\pi(x)-{\rm li}(x)$} is positive.
\newblock {\em Math. Comp.}, 79(272):2395--2405, 2010.

\bibitem[Sha]{XShao}
X.~Shao.
\newblock A density version of the {V}inogradov three prime theorem.
\newblock Preprint. Available as {\texttt{arXiv:1206.6139}}.

\bibitem[Shu92]{EB}
F.~H. Shu.
\newblock {T}he {C}osmos.
\newblock In {\em Encyclopaedia Britannica, Macropaedia}, volume~16, pages
  762--795. Encyclopaedia Britannica, Inc., 15 edition, 1992.

\bibitem[Tao]{Tao}
T.~Tao.
\newblock Every odd number greater than 1 is the sum of at most five primes.
\newblock Preprint. Available as {\texttt{arXiv:1201.6656}}.

\bibitem[Vau77]{MR0437478}
R.~C. Vaughan.
\newblock On the estimation of {S}chnirelman's constant.
\newblock {\em J. Reine Angew. Math.}, 290:93--108, 1977.

\bibitem[Vin37]{Vin}
I.~M. Vinogradov.
\newblock Representation of an odd number as a sum of three primes.
\newblock {\em Dokl. Akad. Nauk. SSR}, 15:291--294, 1937.

\bibitem[Wed03]{Wed}
S.~Wedeniwski.
\newblock {Z}eta{G}rid - {C}omputational verification of the {R}iemann
  hypothesis.
\newblock Conference in {N}umber {T}heory in honour of {P}rofessor {H}. {C}.
  {W}illiams, Banff, Alberta, Canada, May 2003.

\end{thebibliography}
\end{document}